\documentclass{amsart}

\usepackage{amsmath,amssymb,amsthm,amsfonts,enumitem,graphicx,xcolor,subcaption,tikz,bbm}
\usepackage{tikz-3dplot}
\usepackage{mathtools}
\usepackage{mathrsfs}
\usepackage[margin=1.5in]{geometry}
\usepackage[colorinlistoftodos,bordercolor=orange,backgroundcolor=orange!20,linecolor=orange,textsize=scriptsize]{todonotes}
\usepackage{etoolbox}
\makeatletter
\let\ams@starttoc\@starttoc
\makeatother
\usepackage[parfill]{parskip}
\makeatletter
\let\@starttoc\ams@starttoc
\patchcmd{\@starttoc}{\makeatletter}{\makeatletter\parskip\z@}{}{}
\makeatother 
\usepackage{multirow}
\usepackage{tabularx}
\usepackage{hyper ref}
\usepackage{framed}

\let\oldtocsection=\tocsection
  
\let\oldtocsubsection=\tocsubsection
 
\let\oldtocsubsubsection=\tocsubsubsection
 
\renewcommand{\tocsection}[2]{\hspace{0em}\oldtocsection{#1}{#2}}
\renewcommand{\tocsubsection}[2]{\hspace{1em}\oldtocsubsection{#1}{#2}}
\renewcommand{\tocsubsubsection}[2]{\hspace{2em}\oldtocsubsubsection{#1}{#2}}

\makeatletter
\def\subsubsection{\@startsection{subsubsection}{3}%
  \z@{.5\linespacing\@plus.7\linespacing}{.1\linespacing}%
  {\normalfont\itshape}}
\makeatother

\newtheorem{theorem}{Theorem}[section]
\newtheorem{notation}[theorem]{General set-up}
\newtheorem{definition}[theorem]{Definition}

\newtheorem{proposition}[theorem]{Proposition}
\newtheorem{lemma}[theorem]{Lemma}
\newtheorem{corollary}[theorem]{Corollary}

\theoremstyle{definition}
\newtheorem{remark}[theorem]{Remark}
\newtheorem{example}[theorem]{Example}

\newtheorem*{theorem*}{Theorem}

\newtheorem*{corollary*}{Corollary}
\newtheorem*{proposition*}{Proposition}
\newtheorem{conjecture}[theorem]{Conjecture}

\definecolor{cAlex}{rgb}{0.1,0.45,0.03}
\definecolor{cChristian}{rgb}{0.7,0.15,0.63}
\definecolor{cLena}{rgb}{0.6,0.5,0.0}

\newcommand{\rem}[1]{} 
\numberwithin{equation}{section}

\newcommand{\Z}{\mathbb{Z}}
\newcommand{\R}{\mathbb{R}}
\newcommand{\C}{\mathbb{C}}
\newcommand{\N}{\mathbb{Z}_{\geq 1}}

\newcommand{\fd}{K}  
\newcommand{\suchthat}{ \ : \ } 
\newcommand{\Q}{\mathbb{Q}}
\newcommand{\fn}{\mathbf{0}}
\newcommand{\origin}{\fn}
\newcommand{\cur}{C} 
\newcommand{\rat}{b} 
\newcommand{\PP}{\mathbb{P}}

\newcommand{\new}{\operatorname{NP}}
\newcommand{\homm}{\theta} 
\newcommand{\dimm}{n} 
\newcommand{\env}[1]{#1^c} 
\newcommand{\Hom}{\operatorname{Hom}} 
\newcommand{\maxi}{\mathfrak{m}}

\newcommand{\pol}{P} 
\newcommand{\conv}{\operatorname{conv}} 
\newcommand{\fan}{\Sigma} 
\newcommand{\ray}{\rho} 
\newcommand{\cone}{\operatorname{cone}} 

\newcommand{\sta}{\operatorname{star}} 
\newcommand{\rg}{u_{\ray}}  
\newcommand{\con}{\sigma} 
\newcommand{\ldc}{\tau} 
\newcommand{\np}{\operatorname{NP}}  

\newcommand{\supp}{\operatorname{supp}} 
\newcommand{\hyp}{H}  
\newcommand{\face}{F} 
\newcommand{\mv}{\operatorname{MV}} 
\newcommand{\nor}[1]{\linf_{#1}} 
\newcommand{\sun}[2]{\operatorname{sun}(#1,#2)} 
\newcommand{\ver}{\operatorname{vert}} 
\newcommand{\llen}{\operatorname{length}} 
\newcommand{\core}{\operatorname{core}} 
\newcommand{\dir}{v} 
\newcommand{\ls}{L} 
\newcommand{\relw}[2]{\operatorname{length}(#1,#2)} 
\newcommand{\wid}[1]{\operatorname{width}_{#1}} 
\newcommand{\linf}{u} 
\newcommand{\relint}{\operatorname{relint}} 
\newcommand{\feasible}[3]{#1_{#2}(#3)} 

\newcommand{\lat}{N} 
\newcommand{\latm}{M} 
\newcommand{\tor}{\mathbb{T}} 
\newcommand{\orbc}[1]{V(#1)} 
\newcommand{\var}{X} 
\newcommand{\cox}{S} 
\newcommand{\sv}{Z} 
\newcommand{\point}{z} 
\newcommand{\hir}{\mathscr{H}} 

\newcommand{\flag}{Y_{\bullet}} 
\newcommand{\zdn}[1]{N_{#1}} 
\newcommand{\zn}[2]{#1^{-}_{#2}} 
\newcommand{\zp}[2]{#1^{+}_{#2}} 
\newcommand{\nob}[2]{\Delta_{#1}(#2)}
\newcommand{\fct}{\varphi} 
\newcommand{\val}{\operatorname{val}} 
\newcommand{\seh}{\varepsilon} 
\newcommand{\op}{m} 
\newcommand{\gen}{R} 
\newcommand{\vertex}{p} 
\newcommand{\pwl}{\Psi}  
\newcommand{\djs}[2]{\mathfrak{s}(#1,#2)} 
\newcommand{\qco}{\eta} 
\renewcommand{\a}{a} 
\renewcommand{\b}{b} 
\newcommand{\adjoint}[2]{#1^{\operatorname{FA}\!(#2)}} 
\newcommand{\codeg}[1]{\qco^{\operatorname{F}}\!(#1)}   
\newcommand{\fineCore}[1]{\core^{\operatorname{F}}\!(#1)} 

\newcommand{\ord}{\operatorname{ord}}
\newcommand{\ds}{D} 
\newcommand{\she}{\mathcal{O}} 
\newcommand{\spf}{\operatorname{SF}} 
\newcommand{\dsl}{L} 
\newcommand{\divof}{\operatorname{div}} 
\newcommand{\lb}{\mathscr{L}} 
\newcommand{\dc}{a} 
\newcommand{\lbv}{V_{\lb}} 
\newcommand{\divp}[1]{P_{#1}} 
\newcommand{\cand}[1]{K_{#1}} 
\newcommand{\sps}{H^0} 
\newcommand{\bcone}{\operatorname{Big}} 

\title[Toric Newton--Okounkov functions]{Toric Newton--Okounkov functions with an application to the rationality of certain Seshadri constants on surfaces}
	
	\author{Christian Haase}
\address{Christian Haase, Fachbereich Mathematik und Informatik, Freie Universit\"at Berlin, Arnimallee 3, 14195 Berlin, Germany}	
	\email{haase@math.fu-berlin.de}
	
	\author{Alex K\"uronya}
	\address{Alex K\"uronya, Institut f\"ur Mathematik, Goethe--Universit\"at Frankfurt, Robert-Mayer-Str. 6--10, D-60325 Frankfurt am Main,  Germany}
\email{kuronya@math.uni-frankfurt.de}

	\author{Lena Walter}
	\address{Lena Walter, Fachbereich Mathematik und Informatik, Freie Universit\"at Berlin, Arnimallee 3, 14195 Berlin, Germany}
\email{lenawalter@math.fu-berlin.de}

\begin{document}

\begin{abstract}
We initiate a combinatorial study of Newton--Okounkov functions on
toric varieties with an eye on the rationality of asymptotic
invariants of line bundles. In the course of our efforts we identify a
combinatorial condition which ensures a controlled behavior of the
appropriate Newton--Okounkov function on a toric surface.
Our approach yields the rationality of many Seshadri constants that
have not been settled before.
\end{abstract}

\maketitle

\tableofcontents

\section{Introduction}

In the present paper, we start to develop methods to determine
Newton--Okounkov functions in the case of toric varieties. As a
by-product, we can show rationality of Seshadri constants for many new
examples of toric surfaces.\\
\\
Newton--Okounkov bodies are convex bodies which encode various facets of
algebraic and symplectic geometry, such as the local positivity of line 
bundles on varieties and going as far as geometric
quantization~\cite{HHK}. To be more specific, it is possible to gain
information on asymptotic invariants (Seshadri constants,
pseudo-effective thresholds, Diophantine approximation constants) from
well-chosen Newton--Okounkov bodies and concave
(Newton--Okounkov) func\-tions on them.\\
\\
Over the past decade, Newton--Okounkov theory has attracted a lot of
attention. Many deep structural results have been proven, extracting
information about varieties and their line bundles from
Newton--Okounkov bodies. At the same time it also became apparent that
it is often very difficult to obtain precise information about
Newton--Okounkov bodies and Newton--Okounkov functions in
concrete cases.

\subsubsection*{Newton--Okounkov Functions}
Newton--Okounkov functions are concave functions on Newton--Okounkov
bodies arising  from  multiplicative filtrations on the section ring
of a line bundle. They have proven to be more evasive than
Newton--Okoun\-kov bodies themselves.\\
\\
The first definition of Newton--Okounkov functions (in other terminology, concave transforms of multiplicative filtrations) in print is due to Boucksom--Chen~\cite{BC}. These functions on the Newton--Okounkov body yield refined
information~\cite{BKMS,DKMS2,KFuj,KMSwB,KMR,McKR} about the arithmetic and the geometry of the underlying variety.

Already the most basic invariants of Newton--Okounkov functions
contain highly non-trivial information. Perhaps the most notable
example is the average of such a function --- called the $\beta$-invariant of the line bundle and the filtration --- ,  which is closely related to
Diophantine approximation~\cite{McKR}, and K-stability~\cite{KFuj}. 
By the  connection of the $\beta$-invariant to Seshadri constants~\cite{KMR}, 
its rationality could decide Nagata's conjecture~\cite{DKMS}. Not surprisingly, concrete descriptions of these functions are very
hard to obtain.\\
\\
A structure theorem of~\cite{KMR} identifies the subgraph of a
Newton--Okounkov function coming from a geometric situation as
the Newton--Okounkov body of a projective bundle over the variety in
question. (Compare Section~\ref{sec:subgraph} below.)\\
\\
Based on earlier work of Donaldson~\cite{Don}, Witt--Nystr\"om~\cite{WN} made the observation that the Newton--Okounkov function
coming from a fully toric situation (meaning all of the line bundle, admissible flag, and filtration are torus-invariant) 
is piecewise affine linear with rational coefficients on the
underlying Newton--Okounkov body, which happens to coincide with the appropriate moment polytope. In this very special
situation, the function is in fact linear. (Compare Proposition~\ref{prop_distance} 
below.)\\
\\
The next interesting case arises when we keep the toric polytope (that is, we work with a torus-invariant line bundle and a torus-invariant admissible flag),
but we consider the order of vanishing at a general point to define the
function. To our knowledge, no such function has been computed for
toric varieties other than projective space.\\
\\
 While the usual dictionary between geometry and combinatorics is very
 effective in explaining torus-invariant geometry, when it comes to
 non-torus-invariant phenomena, one does need the more general framework
 of Newton--Okounkov theory. Broadly speaking Newton--Okounkov theory
 would be toric geometry without a torus action; in more technical
 terms Newton--Okounkov theory replaces the natural gradings on
 cohomology spaces by filtrations.\\
\\
In determining Newton--Okounkov functions in a not completely toric setting, our first goal is to devise a strategy to determine 
Newton--Okounkov functions associated to orders of vanishing  on toric surfaces, and to apply it to
interesting examples. The trick is to avoid blowing up the valuation
point,  which could result in losing control of the Mori cone.\\
 \\
Instead, we change the flag defining the Newton--Okounkov
polytope to one which contains the valuation point and show that there
is a piecewise linear transformation of the moment polytope into the
new Newton--Okounkov polytope (see Corollary~\ref{cor_mv}). This is
reminiscent of the transformation constructed in~\cite{escobar2019wall}. But the connection is, as of yet, unclear.
It is worth mentioning  at this point that certain pairs of
subgraphs of Newton--Okounkov functions associated to torus-invariant
and non-torus-invariant flags happen to be equidecomposable
(compare Remark~\ref{rem:equidecomposable}). This is an exciting and
unexpected phenomenon with possible ties to the mutations studied in~\cite{CFKLRS}.
We offer a conjectural explanation for this phenomenon.\\
\\
We can then employ arguments from convex geometry to provide upper and
lower bounds for the desired function. We study combinatorial
conditions which guarantee that the obtained upper and lower bounds
agree.\\
\\
In the case of anti-blocking polyhedra in the sense of
Fulkerson~\cite{FulkersonBlockingAntiBlocking,FulkersonAntiBlocking}
we obtain a particularly easy answer.
Nevertheless, the strategy works much more generally.

\begin{theorem*}[Newton--Okounkov functions on toric surfaces, Theorem~\ref{prop_a+b}, Corollary~\ref{cor_par}]
Let $\var$ be a smooth projective toric surface, $\ds$ an ample
divisor, and $\flag$ an admissible torus-invariant flag on $\var$
so that the Newton--Okounkov body
$\nob{\flag}{\ds}$
is anti-blocking.

Let $\flag'$ be a torus-invariant flag opposite to the origin.
Then the Newton--Okounkov function
$\fct_R$ on $\nob{\flag'}{\ds}$ coming from the geometric valuation
$\ord_R$ in a general point $R\in \var$ is linear with integral
slope.
\end{theorem*}

Along the way, we formulate and prove existence and uniqueness of
Zariski decomposition on toric surfaces in the language of  polyhedra. One
can `see' the decomposition in terms of the polygons, compare Theorem~\ref{thm_toriczd}.

\subsubsection*{Local Positivity}
Newton--Okounkov theory reveals a lot about positivity properties of  line
bundles. Just like in the toric case,  one can use    convex
geometric information to  decide for instance if the  underlying line bundle  is ample or nef \cite{KL_inf,KL_noninf}, 
one can even  obtain localized information. 
A line bundle is called positive or ample at a point of our variety if global sections of a
high enough multiple yield an embedding of an open neighborhood of
the point. Local positivity can be decided and measured via Newton--Okounkov bodies
\cite{KL_inf,KL_noninf,Roe}. \\
\\
Local positivity is traditionally measured by Seshadri constants~\cite{KL_Geom,lazarsfeld2004positivity}. Originally invented by
Demailly~\cite{Dem92} to attack Fujita's conjecture on global generation,
Seshadri constants have become the main numerical asymptotic positivity
invariant (compare~\cite[Chapter
5]{lazarsfeld2004positivity},\cite{BauerSeshadri}).
While there has been considerable interest in this invariant's
behavior, many of its properties are still shrouded in
mystery~\cite{Szemberg}.\\
\\
One interesting question about Seshadri constants is if they are
always rational numbers. This is widely believed to be false, but
there has only been sporadic progress towards this issue. On surfaces,
the rationality of Seshadri constants would imply the failure of
Nagata's conjecture~\cite{DKMS}. \\
\\
The rationality of Seshadri constants and related asymptotic invariants often follows from finite generation of an 
appropriate multi-graded ring or semigroup~\cite{ELMNP,CasciniLazic}. However, finite generation questions tend to be wide open, and are typically skew to the major finite generation theorems of birational geometry. \\
\\
In the literature around Newton--Okounkov bodies  the involved valuation semigroups are frequently
assumed, from the outset, to be finitely generated (see~\cite{HK,KavehManon}, this is to obtain a toric degeneration as in~\cite{And}). However,  deciding
finite generation of multigraded algebras or semigroups arising from a geometric setting  is an utterly hard question.\\
\\
We obtain results on the rationality of Seshadri constants in general
points of toric surfaces using asymptotic considerations and convexity
to circumvent some of these difficulties.\\
\\
Previously, Ito~\cite{ito13, ito14}, Lundman~\cite{lundman2015computing}, and Sano~\cite{sano2014seshadri} have verified rationality of these same
Seshadri constants for restricted classes of (line bundles on) toric
surfaces.
As it turns out, a condition we call `weakly zonotopally well-covered',
is sufficient to guarantee rationality of the
Seshadri constant. 


\begin{theorem*}[Rationality of certain Seshadri constants, Theorem~\ref{thm_seshadri}]
Let $\var$ be a smooth projective toric surface and $\ds$ an ample torus-invariant divisor on $\var$ with associated Newton--Okounkov body $\nob{\flag}{\ds}$ for an admissible torus-invariant flag $\flag$. If the polytope $\nob{\flag}{\ds}$ is weakly zonotopally well-covered, then
\begin{enumerate}
    \item we can determine $\int_{\nob{\flag}{\ds}}{\fct_\gen}$. 
    \item the Seshadri constant $\seh(\var, \ds;\gen)$ is rational. 
    \item the maximum $\max_{\nob{\flag}{\ds}}{\fct_\gen}$ is attained at the boundary of $\nob{\flag}{\ds}$.
\end{enumerate}
\end{theorem*}

This theorem reproves some of the cases covered
in~\cite{ito13,ito14,lundman2015computing,sano2014seshadri}
and adds many new
cases, even some, where we can only conjecture what the
Newton--Okounkov function looks like. We construct specific examples
where the methods of Ito, Sano, or Lundman do not apply, see Theorem~\ref{thm:seshadri-construction}.

\subsubsection*{Organization of the Paper}

We start in Section~\ref{sec_background_notation} by fixing notation and giving necessary background information. Since our work sits on the fence between two areas, we give ample information on both. Section~\ref{sec_nob_toric} is devoted to a self-contained combinatorial proof of Zariski decomposition on toric surfaces and the existence of the `tilting isomorphism' between certain Newton--Okounkov bodies. In Section~\ref{sec_function} we give description of Newton--Okounkov functions/concave transforms in the two relevant cases: when every actor is torus-invariant (Subsection~\ref{sec:torictoric}) and when we are looking at the order of vanishing filtration coming from a general point (Subsection~\ref{sec:general-point}). The latter part contains the outline of our general strategy. Moreover, we give a polyhedral description of the interpretation of a subgraph as a Newton--Okounkov body (Subsection~\ref{sec:subgraph}). Section~\ref{sec:seshadri} contains the application of our results on Newton--Okounkov functions to the rationality of Seshadri constants. 
 
\subsubsection*{Acknowledgments}

We would like to thank Joaquim Ro\'e and Sandra Di Rocco for helpful
discussions, and Atsushi Ito for important remarks on a previous version of this manuscript. Ito informed us that he computed the Seshardri constants for a specific sequence of toric surfaces of the kind we construct in Theorem~\ref{thm:seshadri-construction} (unpublished).

The second author was partially supported by the LOEWE
grant `Uniformized Structures in Algebra and Geometry'.
This cooperation started at the coincident Mini-workshops `Positivity
in Higher-dimensional Geometry: Higher-codimensional Cycles and
Newton-Okounkov Bodies' and `Lattice Polytopes: Methods, Advances,
Applications' at the Mathematische Forschungsinstitut Oberwolfach in
2017. We appreciate the stimulating atmosphere and the excellent
working conditions.

\section{Background and Notation}\label{sec_background_notation}

We work over an algebraically closed field $\fd$. Although no arguments  depend on characteristic zero,  for convenience we will assume $\fd= \C$. 

\subsection{Toric Varieties}
We review some basic results and fix notation regarding toric varieties. We follow the conventions used in \cite{cox2011toric}.

Let $\var$ be an $\dimm$--dimensional smooth projective \emph{toric} variety. Then $\var= \var_\fan$ is determined by a complete unimodular fan $\fan$ in $\lat_\R = \lat \otimes_\Z \R \simeq \R^\dimm$, where $\lat \simeq \Z^\dimm$ denotes the underlying lattice of one--parameter subgroups. Its dual, the lattice of characters is denoted by $\latm= \Hom_\Z(\lat,\Z)$ and the associated vector space by $\latm_\R=\latm \otimes_\Z \R$. We denote the underlying torus by $\tor= \lat \otimes_\Z \fd ^*$. 

Let $\fan(i)$ denote the set of $i$--dimensional cones of the fan. Each ray $\ray \in \fan(1)$ is determined by a primitive ray generator $\rg \in \lat$. Since $\fan$ is unimodular, the primitive ray generators of each maximal cone $\con \in \fan$ form a basis of $\lat$. The toric patches will be denoted by $U_\con$ for $\con \in \fan$.

Since $\var$ is smooth, all Weil divisors are Cartier, i.e. $\text{Pic}(\var)=\text{Cl}(\var)$. Due to the Orbit-Cone correspondence a ray $\ray \in \fan(1)$ gives a codimension-one orbit whose closure $\orbc{\ray}$ is a torus-invariant prime divisor on $\var$ which we denote by $\ds_\ray$. We write  $\cand{\var} = -\sum_{\ray}{\ds_\ray}$ for the canonical divisor on $\var$.  

Given a torus-invariant divisor $\ds=\sum_\ray{\dc_\ray \ds_\ray}$ on $\var$, it determines a polytope
\begin{equation}\label{eqn_divpoly}
\divp{\ds} \coloneqq \left\{ m \in \latm_\R \suchthat \langle m, u_\ray \rangle \geq -\dc_\ray \text{ for all } \ray \in \fan(1) \right\}\ ,
\end{equation}
since $\fan$ is complete. Denote the normal fan of $\divp{\ds}$ by $\fan_{\divp{\ds}}$.

To a Cartier divisor $\ds$ on $ \var$ we can associate the sheaf $\she_\var(\ds)$, which is the sheaf of sections of a line bundle $\pi \colon \lbv \to \var$ which we will denote by $\lb$ for short. We can describe a Cartier divisor $\ds=\sum_{\ray}{a_\ray\ds_\ray}$ in terms of its support function $\spf_\ds \colon |\fan| \to \R$, which is linear on each $\con \in \fan$ with $\spf_\ds(u_\ray)=-a_\ray$ for all $\ray \in \fan(1)$. A Cartier divisor is determined by its Cartier data $\{m_\con\}_{\con\in \fan}$, where the $m_\con$ satisfy $\ds|_{U_\con}= \divof (\chi^{-m_\con})|_{U_\con}$ for all $\con \in \fan$. The vector space of global sections arises from the characters for the lattice points inside the polytope, namely 
\begin{equation}\label{enq_basis_sec}
\sps(\var,\she_\var(\ds))= \bigoplus_{m \in \divp{\ds} \cap \latm}{\C \cdot \chi^m}\ .
\end{equation} 
A divisor $\ds$ is big if and only if its associated polytope $\divp{\ds}$ is full-dimensional, it is ample if and only if the normal fan $\fan_{\divp{\ds}}$ of the polytope $\divp{\ds}$  is exactly $\fan$, and  it is nef precisely if all defining inequalities of $\divp{\ds}$ are tight and the fan $\fan$ is a refinement of the normal fan $\fan_{\divp{\ds}}$.   

Let $\ds$ be an ample divisor on $\var$ with corresponding polytope $\divp{\ds}$. Then for each vertex $p \in \ver(\divp{\ds})$ the primitive vectors $m_1^p,\ldots,m_\dimm^p$ in its adjacent edge directions are a lattice basis for $\latm$, and therefore specify an associated coordinate system of the space $\latm_\R$. 

For a strongly convex rational polyhedral cone $\ldc$ in $\lat_\R$, let $\lat_\ldc$ be the sublattice of $\lat$ spanned by points in $\lat \cap \ldc$. Then we denote the quotient lattice by $\lat(\ldc)=\lat / \lat_\ldc$. Let $\fan$ be a fan in $\lat_\R$ and $\ldc \in \fan$. We consider the quotient map $\lat_\R \to \lat(\ldc)_\R$ and denote by $\overline{\con}$ the image of a cone $\con \in \fan$ containing $\ldc$. Then
\begin{equation*}
\sta(\ldc) \coloneqq \{\overline{\con} \subseteq \lat(\ldc)_\R \suchthat \ldc \preceq \con \in \fan\}
\end{equation*} 
is a fan in $\lat(\ldc)_\R$.

There is again a toric variety associated to this fan. Let $P \subseteq \latm_\R$ be a full dimensional lattice polytope with normal fan $\fan_P$ and associated toric variety $\var_P$. Then each face $Q \preceq P$ corresponds to a cone $\con_Q \in \fan_P$. According to Propositions 3.2.7 and 3.2.9 in \cite{cox2011toric} we obtain  isomorphisms 
\begin{equation*}
\var_{\sta(\con_Q)} \simeq \orbc{\con_Q} \simeq \var_Q
\end{equation*} 
between the resulting varieties, where $\var_Q$ is the variety that is associated to the lattice polytope $Q$. 

\subsection{Measuring Polytopes}

Let $\pol \subseteq \latm_\R$ be a polytope. The \emph{width of $\pol$
  with respect to a linear functional $\linf \in \lat$} is defined as
\begin{equation*}
\wid{\linf}(\pol) \coloneqq \max_{\vertex,T \in \pol}{|\linf(\vertex)-\linf(T)|}.
\end{equation*} 
For a rational line segment $\ls$ there is the notion of \emph{lattice
  length}, denoted by $\llen_{\latm}(\ls)$. Let therefore $\ls$ be the
segment connecting the rational points $\vertex,T \in \latm_\Q$ and
denote by $m \in \latm$ the shortest lattice vector on the ray spanned
by $\vertex-T$. Then we define $\llen_{\latm}(\ls) \coloneqq |j|$, where
$j \in \Q $ such that $\vertex-T=j m $.

Let $\dir \in \Z^\dimm$ be a primitive vector.
For a point $\op \in \pol$ we define the
\emph{length of $\pol$ at $\op$ with respect to $\dir$} to be
\begin{equation*}
\relw{\pol,\op}{\dir}\coloneqq \max{ \left\{ t \in \R \suchthat \op +
    t\dir \in \pol \right\} }, 
\end{equation*}
and $\relw{\pol}{\dir}$ is the maximal length over all $\op \in \pol$.


\subsection{Newton--Okounkov Bodies}
The rich theory of toric varieties provides a very useful dictionary between algebraic geometry and convex geometry. In the 90's in \cite{Ok1} Okounkov laid down the groundwork to generalize this idea to arbitrary projective varieties motivated by questions coming from representation theory. Based on this, Lazarsfeld--Musta\c{t}\u{a} \cite{MR2571958} and Kaveh--Khovanskii \cite{kaveh2012newton} independently developed a systematic theory of Newton--Okounkov bodies about ten years later. It lets one assign a collection of convex bodies to a given pair $(\var,\ds)$ that captures much of the asymptotic information about its geometry.  We follow that notation of \cite{MR2571958} (see also \cite{Bou14,KL_Geom}). 

Given an  irreducible projective variety $\var$ of dimension $\dimm$ and  an admissible flag 
\begin{equation*}
\flag \colon \var= Y_0 \supseteq Y_1 \supseteq \cdots \supseteq Y_\dimm
\end{equation*} 
of irreducible subvarieties, one defines the Newton--Okounkov body $\nob{\flag}{\ds}$ of  a big Cartier divisor $\ds$  on $\var$ in the following way.

\begin{definition}
The \emph{Newton--Okounkov body} $\nob{\flag}{\ds}$ of $\ds$ (with respect to the flag $\flag$) is defined to be the set
\begin{equation*}
\nob{\flag}{\ds} \coloneqq \overline{\bigcup_{k \geq 1}{ \frac{1}{k} \left\{ \val_{\flag} (s) \suchthat s \in \sps(\var,\she_\var(k\ds))\setminus \{0\} \right\} }} \subseteq \R^\dimm.
\end{equation*}
\end{definition}

\begin{example}\label{ex_build_nob}
We consider the Hirzebruch surface $\var= \hir_1$  associated to the fan in Figure~\ref{fig_fan_hir1}, where the torus-invariant prime divisor $\ds_i$ corresponds to the ray $\ray_i \in \fan(1)$ for $1 \leq i\leq 4$.  

\begin{figure}[h!]
\begin{center}
    \includegraphics[scale=0.11]{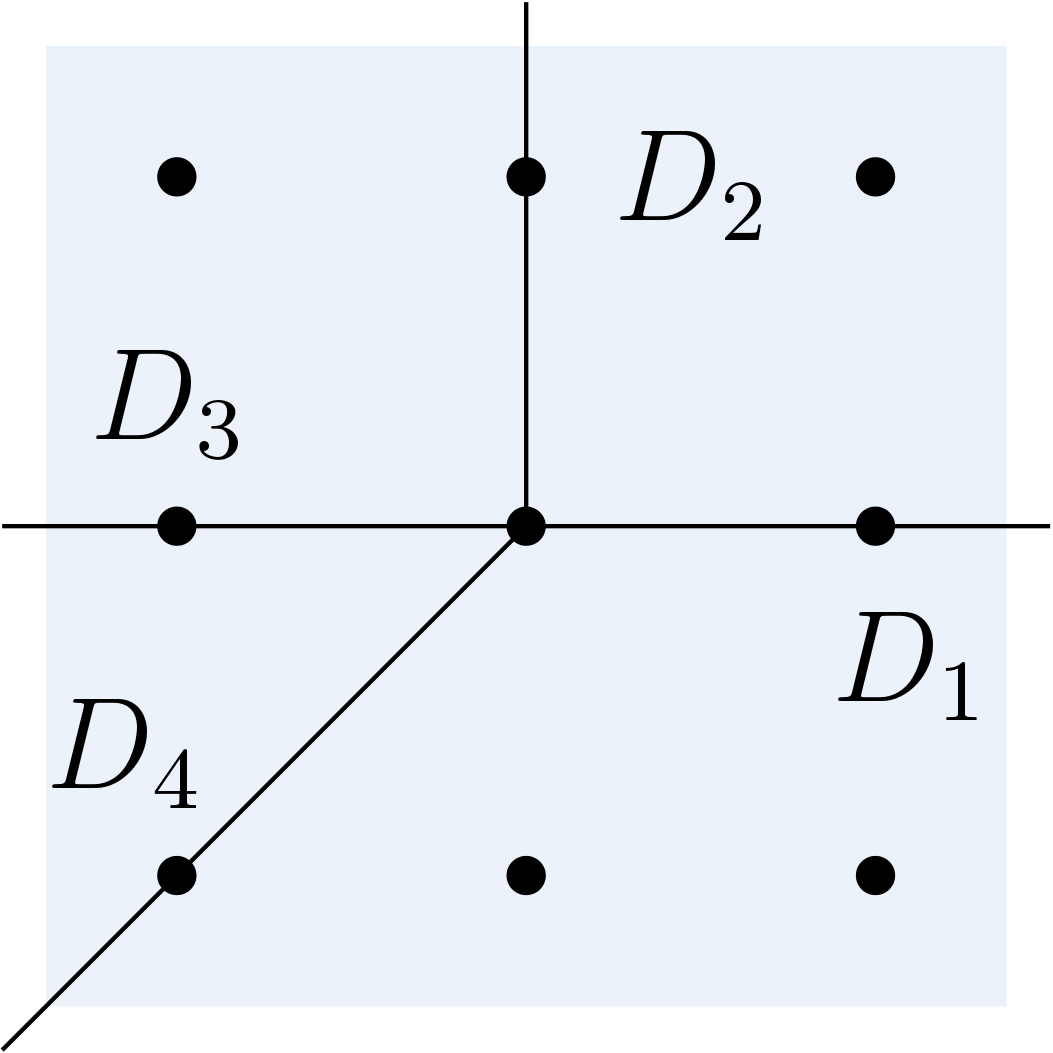}
\caption{The fan $\fan$ of the first Hirzebruch surface $\var_\fan= \hir_1$.} \label{fig_fan_hir1}
\end{center}
\end{figure} 

As an admissible flag $\flag$ choose the curve $Y_1=\ds_1$ and $Y_2=\ds_1 \cap \ds_2$ as a smooth point on it. Then we have a local system of coordinates $x,y$ such that $Y_1= \overline{\{x=0\}}$ and $Y_2=(0,0)$. The ample divisor $\ds=\ds_3+2\ds_4$ determines a polytope $\divp{\ds}$ and, due to~(\ref{enq_basis_sec}), the global sections of $\sps(\var,\she_\var(\ds))$ involve the monomials $1,x,y,xy,$ and $y^2$, given in local coordinates as depicted in Figure~\ref{fig_globalsections}. 
The global section $s(x,y)=2x+8xy$ for instance gets mapped to $(1,0)$ by the map $\val_{\flag}$, because its divisible by $x$, but $s_1(x,y)=2+8y$ is not divisible by $y$. Altogether, computing the Newton--Okounkov body $\nob{\flag}{\ds}$ recovers the polytope $\divp{\ds}$, which is not a coincidence.
\begin{figure}[h!]
\begin{center}
    \includegraphics[scale=0.11]{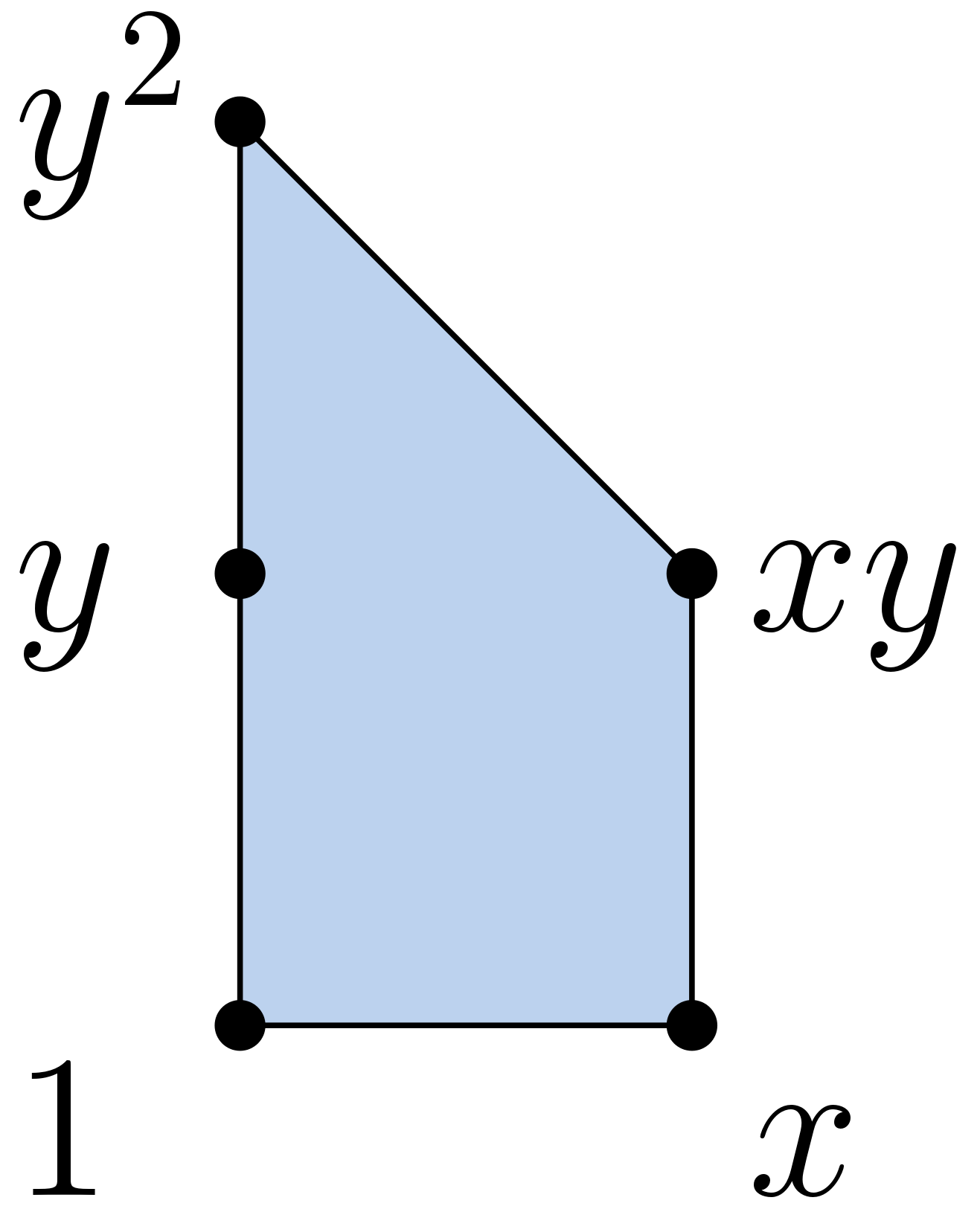}
\caption{The polytope $\divp{\ds}$ with the monomials  corresponding to its lattice points.} \label{fig_globalsections}
\end{center}
\end{figure} 
\end{example}

Given the data $\var$, $\flag$ and $\ds$ there is no straightforward way to compute the corresponding Newton--Okounkov body that works in general. For the case of surfaces the existence of Zariski decomposition leads to a promising approach. In its original form it goes back to Zariski \cite{10.2307/1970376}, where he showed how to uniquely decompose a given effective $\Q$--divisor $\ds$ into a positive part $\zp{\ds}{}$ and a negative part $\zn{\ds}{}$. This result was reproved by Bauer \cite{bauer2009simple} and also Fujita provided an alternative proof in \cite{fujita1979zariski} which also extends to pseudo-effective $\R$--divisors. Here we review the statement in its most general form.

\begin{theorem}[\cite{kawamata1987}, Theorem 7.3.1]\label{thm_zd}
Let $\ds$ be a pseudo-effective $\R$-divisor on a smooth projective surface. Then there exists a unique effective $\R$-divisor
\begin{equation*} 
\zn{\ds}{}=\sum_{i=1}^{\ell}{a_i N_i} 
\end{equation*}
such that 
\begin{enumerate}
	\item $\zp{\ds}{} = \ds-\zn{\ds}{}$ is nef,
	\item $\zn{\ds}{}$ is either zero or its intersection matrix $(N_i . N_j)_{i,j}$ is negative definite,
	\item $\zp{\ds}{} . N_i=0$ for $i \in \{1,\ldots,\ell\}$.
\end{enumerate}
Furthermore, $\zn{\ds}{}$ is uniquely determined as a cycle by the numerical equivalence class of $\ds$; if $\ds$ is a $\Q$-divisor, then so are $\zp{\ds}{}$ and $\zn{\ds}{}$. The decomposition 
\begin{equation*}
\ds=\zp{\ds}{}+\zn{\ds}{}
\end{equation*}
is called the \emph{Zariski decomposition} of $\ds$.
\end{theorem}

To obtain information about the shape of the resulting Newton--Okounkov body it is important to know how that decomposition varies once we perturb the divisor. Let $\cur=Y_1$ denote the curve in the flag. Start at $\ds$, move in direction of $-\cur$ towards the boundary of the big cone $\bcone(\var)$ and keep track of the variation of the Zariski decomposition of $\ds_t\coloneqq \ds - t \cur$. For more details see \cite{BKS}, also \cite[Chapter 2]{KL_Geom}. The next  applies this procedure in order to compute the Newton--Okounkov body $\nob{\flag}{\ds}$.

\begin{theorem}[\cite{MR2571958}, Theorem 6.4]\label{thm_function}

Let $\var$ be a smooth projective surface, $\ds$ a big divisor (or more generally, a big divisor class), $Y_{\bullet} \colon \var \supseteq \cur \supseteq \{\point\}$ an admissible flag on $\var$. Then there exist continuous functions $\alpha,\beta \colon \left[ \nu, \mu \right] \to \R_{\geq 0}$ such that $0 \leq \nu \leq \mu =: \mu_C(\ds)$ are real numbers, 
\begin{enumerate}
	\item $\nu=$ the coefficient of $C$ in $\zn{\ds}{}$,
	\item $\alpha(t)= \ord_\point \zn{\ds}{t}|_C,$
	\item $\beta(t)= \alpha(t) + (\zp{\ds}{t}. C)$.
\end{enumerate}
Then the associated Newton--Okounkov body is given by 
\begin{equation*}
 \nob{\flag}{\ds}=\left\{ (t,m) \in \R^2 \ | \ \nu \leq t \leq \mu, \ \alpha(t) \leq m \leq \beta(t) \right\}.
\end{equation*}
Moreover, $\alpha$ is convex, $\beta$ is concave and both are piecewise linear. 
\end{theorem}

This result can be used to show \cite{KLM} that  Newton--Okounkov bodies of surfaces will always be  polygons in $\R^2$.

\subsection{Functions on Newton--Okounkov Bodies Coming from Geometric Valuations}

The construction of Newton--Okounkov functions in the sense of concave transforms of filtrations goes back   to Boucksom--Chen \cite{BC} and Witt--Nystr\"om \cite{MR3297156} who introduced them from different perspectives. We will focus on functions coming from geometric valuations as dealt with in \cite{KMSwB} and recall the definition restricted to that case.

Given an irreducible projective variety $\var$, an admissible flag $\flag$ and a big divisor $\ds$ let $\nob{\flag}{\ds}$ be the corresponding Newton--Okounkov body. Let now $\sv \subseteq \var$ be a smooth irreducible subvariety. Then we define a Newton--Okounkov function $\fct_\sv$ in a two-step process. A point $\op \in \nob{\flag}{\ds}$ is called a \emph{valuative} point, if 
\begin{equation*}
\op \in \text{Val}_{\flag} \coloneqq \bigcup_{k \geq 1}{ \frac{1}{k} \left\{ \val_{\flag} (s) \suchthat s \in \sps(\var,\she_\var(k\ds))\setminus \{0\} \right\} }.
\end{equation*}  

For a valuative point $\op \in \nob{\flag}{\ds}$ set 
\begin{eqnarray*}
\tilde{\fct}_\sv \colon \text{Val}_{\flag} &\to & \R \\ 
\op & \mapsto & \lim_{k \to \infty}{\frac{1}{k} \sup \{ t \in \R \suchthat \text{it exists } s \in \sps(\var,\she_\var(k\ds))  } \\
& & \text{ with } \val_{\flag}(s)=k\op, \ \ord_\sv(s) \geq t\}. \\
\end{eqnarray*}

Due to Lemma 2.6 in \cite{KMSwB} the set of valuative points $\text{Val}_{\flag}$ is dense in  $\nob{\flag}{\ds}$. For all non-valuative points $\op \in \nob{\flag}{\ds}\setminus \text{Val}_{\flag}$ set $\tilde{\fct}_\sv(\op) \coloneqq 0$. To define a meaningful function on the whole Newton--Okounkov body we use the concave envelope.   

\begin{definition}
Let $\Delta \subseteq \R^\dimm$ be a compact convex set and $f \colon \Delta \to \R$ a bounded real-valued function on $\Delta$. The \emph{closed convex envelope} $\env{f}$ of $f$ is defined as
\begin{equation*}
\env{f} \coloneqq \inf\{g(x) \suchthat g \geq f \text{ and } g \colon \Delta \to \R \text{ is concave and upper-semicontinuous}\}.
\end{equation*}
\end{definition}

\begin{definition}
Define the \emph{Newton--Okounkov function} $\fct_\sv$ coming from the geometric valuation associated to $\sv$ as
\begin{eqnarray*}
\fct_\sv \colon \nob{\flag}{\ds} & \to & \R \\
\op &\mapsto & \env{\tilde{\fct}_\sv}(\op).
\end{eqnarray*}
\end{definition}

Due to Lemma 4.4 in \cite{KMSwB} taking the concave envelope does not effect the values of the underlying function $\tilde{\fct}_\sv(\op)$ for valuative points $ \op \in  \text{Val}_{\flag}$.

Computing the actual values of a Newton--Okounkov function $\fct_\sv$ becomes extremely difficult and thus the functions are unknown even in some of the easiest cases. In general as far as the formal properties of $\fct_Z$ we will make use of the following know facts. 

\begin{itemize}
	\item $\fct_Z$ is non-negative and concave (\cite{MR3297156} or \cite{BC}, Lemma 1.6, 1.7).
	\item $\fct_Z$ depends only on the numerical equivalence class of $\ds$ (\cite{KMSwB}, Proposition 5.6).
	\item $\fct_Z$ is continuous if $\nob{\flag}{\ds}$ is a polytope (\cite{KMSwB}, Theorem 1.1).
	\item The numbers 
\begin{equation*}
\max_{\nob{\flag}{\ds}}{\fct_\sv} \quad \text{and} \quad \int_{\nob{\flag}{\ds}}{\fct_\sv} 
\end{equation*}
are independent of the choice of $\flag$ (\cite{DKMS2}, Theorem 2.4, \cite{BC}, Corollary 1.13). 
	 
\end{itemize}

\section[Zariski Decomposition]{Zariski Decomposition for Toric Varieties in Combinatorial Terms}\label{sec_nob_toric}

The focus of this section is the determination of Newton--Okounkov bodies in the toric case. For that we will first review the key-correspondence between the Newton--Okounkov body and the polytope associated to a torus-invariant divisor and give an interpretation in terms of Newton polytopes. Then we give a combinatorial way to find a torus-invariant representative for a class of certain divisors, see Proposition~\ref{div_koeff}. This leads to a combinatorial version of Zariski decomposition for the toric case, see Theorem~\ref{thm_toriczd}. Building on this, we illustrate a combinatorial way to define a piecewise linear isomorphism between the involved Newton Okounkov bodies, when changing to a non-invariant flag on a toric surface, see Subsection~\ref{sec_tilting}, in particular Corollary~\ref{cor_mv}.\\
\\
Given a smooth projective toric variety $\var$ of dimension $\dimm$, a torus-invariant flag $\flag$, and a big divisor $\ds$, then the construction of the Newton--Okounkov body $\nob{\flag}{\ds}$ recovers the polytope $\divp{\ds}$ by Proposition~6.1 in~\cite{MR2571958}. This can be seen as follows. \\
\\
Let $\ds_{\ray_1},\ldots,\ds_{\ray_d}$ denote the torus-invariant prime divisors. Since the flag $\flag$ is torus-invariant, we can assume an ordering of the divisors such that the subvarieties in the flag are given as $Y_i= \ds_{\ray_1} \cap \cdots \cap \ds_{\ray_i}$ for $1 \leq i \leq \dimm$. The divisor $\sum_{i=1}^{d} \ds_{\ray_i}$ has simple normal crossings, hence the orders of vanishing of a section $s \in \sps(\var,\she_\var(\ds))$ that has $\sum_{i=1}^{d}\dc_{\ray_i} \ds_{\ray_i}$ as its divisor of zeros can be directly read off as $\val_{\flag}(s)=(\dc_{\ray_1},\ldots,\dc_{\ray_\dimm})$. \\
\\
The underlying fan $\fan$ is smooth and thus the primitive ray generators $u_{\ray_1},\ldots, u_{\ray_\dimm} \in \lat$ span a maximal cone $\con$ and form a basis of the lattice $\lat$. This gives an isomorphism $\lat \cong \Z^{\dimm}$ and the dual isomorphism is given by
\begin{eqnarray}
\Phi \colon \latm & \to & \Z^\dimm \nonumber \\
m & \mapsto & (\langle m , u_{\ray_i} \rangle)_{1 \leq i \leq \dimm},
\end{eqnarray}    
which extends linearly to the map $\Phi_\R \colon \latm_\R \overset{\cong}{\longrightarrow} \R^\dimm$.\\
\\
The Newton--Okounkov body remains the same if one changes the divisor $\ds$ within its linear equivalence class. Hence we can assume $\ds|_{U_\con}=0$, i.e., if the divisor is given as $\ds= \sum_{\ray\in \fan(1)} \dc_\ray \ds_\ray$, then we have $\dc_\ray=0$ for all $\ray \in \con(1)$.\\
\\
The characters $\chi^m$ of points $m$ in $\divp{\ds}$ are exactly the characters of $\tor$ that extend to sections of $\she_\var(\ds)$ on $\var$ and according to~(\ref{enq_basis_sec}) the characters associated to the lattice points of $\divp{\ds}$ form a basis of the vector space of global sections. Given a lattice point $m \in \divp{\ds} \cap \latm$ its associated character $\chi^m$ has divisor of zeros $\ds+ \sum_{i=1}^{d} \langle m, u_{\ray_i} \rangle \ds_{\ray_i}$. Thus for $\ray \in \fan(1)$ the inequality $\langle m, u_\ray \rangle \geq -\dc_\ray$ reflects the condition that $\chi^m$ is regular at the generic point of the divisor $\ds_\ray$. 
Since we assumed $\ds|_{U_\con}=0$ this yields 
\begin{equation*}
\val_{\flag} (\chi^m)=(\langle m , u_{\ray_1} \rangle, \ldots, \langle m , u_{\ray_n} \rangle)=\Phi(m).
\end{equation*}
Thus 
\begin{equation}\conv \left( \left\{ \val_{\flag} (s) \suchthat s \in \sps(\var,\she_\var(\ds))\setminus \{0\} \right\}  \right)= \Phi(\divp{\ds} \cap \latm ).
\end{equation}
We have $\dim (\she_\var(\ds))= |\divp{\ds}\cap \latm|$. For all $k \geq 1$ it holds that $\divp{k \ds}=k \divp{\ds}$. This gives $\nob{\flag}{\ds}=\Phi_\R(\divp{\ds})$.\\
\\
We interpret this identification in terms of Newton polytopes. This approach will play a key role for `guessing' suitable global sections in Section~\ref{sec_function}. For convenience, we assume $\ds$ to be ample. Each divisor $\ds_\ray$ corresponds to a facet $F_\ray$ of $\divp{\ds}$ and all facets $F_{\ray_1},\ldots,F_{\ray_\dimm}$ intersect in a vertex $\vertex_\con$ that is associated to $\con$. Assuming $\ds|_{U_\con}=0$ on the polytope side means to embed the polytope $\divp{\ds}$ in $\R^\dimm$ such that the vertex $\vertex_\con$ is translated to the origin. \\
\\
Let $s \in \sps(\var,\she_\var(\ds))$ be a global section with Newton polytope $\np(s) \subseteq\divp{\ds}$. Then the order of vanishing of $s$ along $Y_1=\ds_{\ray_1}$ is given by the minimal lattice distance to $\face_{\ray_1}$, that is
\begin{equation}
\ord_{Y_1}(s)=\min_{\op \in \np(s)}\langle u_{\ray_1},\op \rangle.
\end{equation}
Let $F_1 \preceq \np(s)$ denote the face of the Newton polytope $\np(s)$ that has minimal lattice distance to $F_{\ray_1}$. Then $\ord_{Y_2}(s_1)=\min_{\op \in F_1 }\langle u_{\ray_2},\op \rangle$ and in general we have
\begin{equation}
\ord_{Y_{i+1}}(s_i)=\min_{\op \in F_i}\langle u_{\ray_{i+1}},\op \rangle
\end{equation}
for $1\leq i \leq \dimm-1$. Thus the map $\val$ sends the section $s$ to the point $\op \in \np(s)$ whose coordinates are lexicographically the smallest among all points of the Newton polytope. A similar argument applies to $k > 1$. Thus we obtain $\nob{\flag}{\ds} \subseteq \divp{\ds}$. Since in particular all the vertices $\vertex \preceq \divp{\ds}$ correspond to respective global sections extending characters $\chi^\vertex$, this yields $\divp{\ds} \subseteq  \nob{\flag}{\ds}$ and therefore $\divp{\ds} \cong  \nob{\flag}{\ds}$.   
For our convenience we identify $\divp{\ds}$ with its image under $\Phi_\R$.  \\
\\
As the Newton--Okounkov bodies  only depend on the numerical equivalence class of $\ds$, we can and often want to choose a torus-invariant representative. If $\ds$ is given by a defining local equation, then there is a combinatorial way to find one.

\begin{proposition}\label{div_koeff} 

Let $\var$ be a smooth projective toric variety with associated fan $\fan$, and $\ds$ a divisor on $\var$ that is given by the local equation $f$ in the torus for some $f \in \C(\var) \setminus \{0\}$. Then $\ds' \coloneqq \sum_{\ray\in \fan(1)}{- \dc_{\ray} \ds_{\ray}}$ with coefficients
\begin{equation}\label{eqn_coeff}
\dc_{\ray} \coloneqq  \min_{m \in \supp(f) }  \langle m,\rg \rangle 
\end{equation}
is a torus-invariant divisor that is linearly equivalent to $\ds$.
\end{proposition}

\begin{proof}
Consider the Cox ring $\cox= \C \left[x_{\ray} \suchthat \ray \in \fan(1) \right]$ which is graded by the class group $\text{Cl}(\var)$, see  Chapter~5 in~\cite{cox2011toric} for details. For a cone $\con \in \fan$ we denote by ${x^{\hat{\con}}= \prod_{\ray \notin \con(1)}x_\ray}$ the associated monomial in $\cox$ and by $\cox_{x^{\hat{\con}}}$ the localization of $\cox$ at $x^{\hat{\con}}$. Applying Lemma~2.2 in ~\cite{cox1995homogeneous} to the $\{0\}$-cone $\sigma_0 \in \fan$ gives an isomorphism of rings
\begin{equation*}
\C\left[M \right] = \C\left[\con^\vee_0 \cap M \right] \cong \left( \cox_{x^{\hat{\con}_0}} \right)_0,
\end{equation*}
where $x^{\hat{\con}_0}= \prod_{\ray }x_\ray$ and $\left( \cox_{x^{\hat{\con}_0}} \right)_0$ is the graded piece of degree $0$. \\
\\
Given a lattice point $m \in \latm$, the character $\chi^m$ is homogenized  to the monomial \linebreak $ x^{\langle m \rangle} = \prod_{\ray} x_\ray^{\langle m,\rg\rangle}$ by the corresponding map $\homm \colon \C\left[M \right] \to \left( \cox_{x^{\hat{\con}_0}} \right)_0 $. Thus homogenizing $f=\sum_{m \in \supp(f)} b_m\chi^m \in \C\left[ \latm \right]$ yields 
\begin{equation*}
\tilde{f} =\homm(f)= \homm\left(\sum_{m \in \supp(f)}{b_m\chi^m}\right)= \sum_{m \in \supp(f) } b_m \prod_{\ray} x_\ray^{\langle m,\rg\rangle}= \frac{g}{\left( \prod_{\ray}{x_{\ray}} \right)^{k}}
\end{equation*} 
for some homogeneous $g \in \cox$ and some $k \in \N$. We can rewrite this as 
\begin{equation}\label{eqn_ausklammern}
\tilde{f}=\frac{g}{\left( \prod_{\ray}{x_{\ray}} \right)^{k}}= \prod_{\ray}{ x_{\ray}}^{\dc_{\ray}} \cdot h
\end{equation} 
for $h \in \cox$ coprime with $ \prod_{\ray}{x_{\ray}}$ and uniquely determined $\dc_{\ray} \in \Z$.\\ 
\\
Since $\tilde{f}$ is homogeneous of degree $0$, it gives a rational function on $\var$ and we have \linebreak $ 0 \sim \divof(\tilde{f}) = \divof\left( \prod_{\ray}{ x_{\ray}}^{ \dc_{\ray}} \right) +\divof{(h)}$. On the torus the zero sets of $f$ and $h$ agree. Since $h$ is coprime with $x_{\ray}$ for all $\ray \in \fan(1)$, it has no zeros or poles along the boundary components. Altogether we have 
\begin{equation*}
\ds = \divof{(h)} \sim \divof\left( \prod_{\ray}{ x_{\ray}}^{- \dc_{\ray}} \right) =: \ds'.
\end{equation*}
Then $\ds'$ is torus-invariant by construction. It remains to show, that the coefficients $\dc_\ray$ from Equation~(\ref{eqn_ausklammern}) satisfy Equation~(\ref{eqn_coeff}). To see that, note, that the homogenization of $f$ consists of summands of the form $b_m \prod_{\ray} x_\ray^{\langle m,\rg\rangle}$, where we sum over $m \in \supp(f)$. But $h$ is an element of the Cox ring and it is supposed to be coprime with $ \prod_{\ray}{x_{\ray}}$. Therefore, to obtain the expression in Equation~(\ref{eqn_ausklammern}), we have to bracket the factor $x_\ray^j$ for $j$ maximal that is a common factor of all the summands for each $\ray \in \fan(1)$. The maximal $j$ is precisely
\begin{equation*}
\dc_{\ray}= \min_{m \in \supp(f) }  \langle m,\rg \rangle
\end{equation*}
as claimed.        
\end{proof}

We give an example to illustrate the proof of Proposition~\ref{div_koeff}.

\begin{example}\label{ex_div_coeff}
We consider the Hirzebruch surface $\var= \hir_1$ as in Example~\ref{ex_build_nob} and work with the divisor
\begin{equation*}
\ds=\overline{ \{ (x,y) \in \tor \suchthat f(x,y)=xy^{-2}-1=0 \}}.
\end{equation*}
Then the Cox ring is given by $\cox=\C\left[ x_{1},x_{2},x_{3},x_{4}\right]=\C \left[x,y,x^{-1},x^{-1}y^{-1}\right]$, where we write $x_i$ for $x_{\ray_i}$.  Homogenizing $f$ yields 
\begin{eqnarray*}
\homm(f)&=& \sum_{m \in \supp(f) } b_m \prod_{\ray} x_\ray^{\langle m,\rg\rangle}= x_{1}^{1}x_{2}^{-2}x_{3}^{-1}x_{4}^{1}-1 \\
	&=& \frac{g}{\left( \prod_{\ray}{x_{\ray}} \right)^{k}} = \frac{x_1^3x_3x_4^3-x_1^2x_2^2x_3^2x_4^2}{(x_{1}x_{2}x_{3}x_{4})^2} \\
	&=& \prod_{\ray}{ x_{\ray}}^{\dc_{\ray}} \cdot h = x_2^{-2}x_3^{-1} \cdot  (x_1x_4-x_2^2x_3)
\end{eqnarray*}
with exponents $a_{\ray_1}=a_{\ray_4}=0, \ a_{\ray_2}=-2$, and $a_{\ray_3}=-1$ and $h=x_1x_4-x_2^2x_3$ is coprime with $x_1x_2x_3x_4$. The same coefficients are obtained using Proposition~\ref{div_koeff}. 
\begin{eqnarray*}
a_{\ray_1}&=&\min(\langle (0,0),(1,0) \rangle,\langle (1,-2),(1,0) \rangle)=0, \\
a_{\ray_2}&=&\min(\langle (0,0),(0,1) \rangle,\langle (1,-2),(0,1) \rangle)=-2, \\
a_{\ray_3}&=&\min(\langle (0,0),(-1,0) \rangle,\langle (1,-2),(-1,0) \rangle)=-1, \\
a_{\ray_4}&=&\min(\langle (0,0),(-1,-1) \rangle,\langle (1,-2),(-1,-1) \rangle)=0. \\
\end{eqnarray*}
Thus $\ds' = \sum_{\ray \in \fan(1)}{- \dc_{\ray} \ds_{\ray}}=2\ds_2+\ds_3$ is a torus-invariant divisor which is linearly equivalent to $\ds$.
\end{example}


With Proposition~\ref{div_koeff} in hand, we can provide a  combinatorial proof for  the existence and uniqueness of Zariski decomposition for smooth toric surfaces independently of Theorem~\ref{thm_zd}. 

\begin{theorem}\label{thm_toriczd}
Let $\var$ be a smooth projective toric surface associated to the fan $\fan$ and let $\ds$ be a pseudo-effective torus-invariant $\R$-divisor on $\var$. Then there exists a unique effective $\R$-divisor 
\begin{equation*}
\zn{\ds}{} = \sum_{i=1}^{\ell}{c_i \zdn{i}}
\end{equation*} 
such that 
\begin{enumerate}
	\item $\zp{\ds}{} = \ds-\zn{\ds}{}$ is nef,
	\item $\zn{\ds}{}$ is either zero or its intersection matrix $(\zdn{i} . \zdn{j})_{i,j}$ is negative definite, and
	\item $\zp{\ds}{} . \zdn{i}=0$ for $i \in \{1,\ldots,\ell\}$.
\end{enumerate}
If $\ds$ is a $\Q$-divisor, then so are $\zp{\ds}{}$ and $\zn{\ds}{}$. 
\end{theorem}

For the proof  we will need the following Lemma.

\begin{lemma}\label{lem_intnum}
Let $\var$ be the toric surface associated to the fan $\fan$. Let $\ds_0,\ldots,\ds_{k+1}$ be torus-invariant prime divisors with adjacent associated primitive ray generators $u_0,\ldots,u_{k+1} \in \R^2$ such that $\cone(u_{0},u_{k+1})$ is pointed and $u_1,\ldots,u_{k} \in \cone(u_0,u_{k+1})$. Then 
\begin{equation}\label{eqn_intnum}
\det((-\ds_i.\ds_j)_{1 \leq i,j \leq k}) = \det(u_0,u_{k+1}).
\end{equation} 
\end{lemma}

\begin{proof}
The intersection numbers of the torus-invariant prime divisors ${\ds_1,\ldots,\ds_k}$ are given as
\begin{itemize}
\item  $\ds_i.\ds_i=-\lambda_i$, where $u_{i-1}+u_{i+1}=\lambda_i u_i$ \\
\item and for $i \neq j$ as
	\begin{equation*}
	\ds_{i}.\ds_{j}=
	\begin{cases} 1 & \text{if }\ray_i \text{ and } \ray_j \text{ are adjacent}\\ 
		0 & \text{otherwise.}
	\end{cases}
	\end{equation*}
\end{itemize}

Thus the intersection matrix is of the form
\begin{equation}\label{eqn_blockmatrix}
 A_k \coloneqq (-\ds_i.\ds_j)_{1 \leq i,j \leq k}=
\begin{pmatrix}
 \lambda_1 & -1 & 0 & \cdots & \cdots & 0 \\
 -1 & \lambda_2 & -1 & 0 & \cdots & 0 \\
 0 & \ddots & \ddots & \ddots & \cdots & 0 \\
  0 & \cdots & \ddots & \ddots & \ddots & 0 \\
0& \cdots & 0 & -1 &\lambda_{k-1 }& -1 \\
0 & \cdots & \cdots & 0 & -1 & \lambda_k \\
\end{pmatrix}.
\end{equation}


We will prove by induction on $k$ that~(\ref{eqn_intnum}) holds.

\textbf{Base case:} For $k=1$ we have 
\begin{eqnarray*}
\det(u_0,u_2) & =& u_0^{(1)}u_2^{(2)}-u_0^{(2)}u_2^{(1)} \\
 &=& \lambda_1 \left( u_0^{(1)} u_1^{(2)}- u_0^{(2)} u_1^{(1)} \right) \\
 &=& \lambda_1,
\end{eqnarray*}
since $\fan$ is smooth. A similar computation applies to $k=2$.

\textbf{Induction step:}
Let $k \geq 3$ be given and suppose~(\ref{eqn_intnum}) is true for all integers smaller than $k$. Note that the determinant of the tridiagonal matrix $A_k$ fulfills a particular recurrence relation, since it is an extended continuant. The recurrence relation is given by
\begin{equation*}
\det(A_0)= 0\ ,\  \det(A_1) = 1\ ,\ \text{and}\ \det(A_k) =\lambda_k\det(A_{k-1})-\det(A_{k-2}). 
\end{equation*}
Thus we have 
\begin{eqnarray*}
\det(A_k)&=&\lambda_k\det(A_{k-1})-\det(A_{k-2}) \\
&\overset{\text{IH}}{=}& \lambda_k\det(u_0,u_k)-\det(u_0,u_{k-1}) \\
&=& \det(u_0,\lambda_ku_k-u_{k-1})\\
&=& \det(u_0,u_{k+1})
\end{eqnarray*}
as claimed.
\end{proof}

\begin{proof}[Proof of Theorem~\ref{thm_toriczd}.]
Since $\ds$ is torus-invariant, it is given as ${\ds= \sum_{\ray\in \fan(1)}a_\ray \ds_\ray}$. We can assume, that $\ds$ is effective, i.e., $a_\ray \geq 0$ for all $\ray \in \fan(1)$. This defines the polygon
\begin{equation*}
\divp{\ds} =\left\{ m \in \latm_\R \suchthat \langle m, \rg \rangle \geq -a_\ray \text{ for all } \ray \in \fan (1) \right\}.
\end{equation*} 
Let $\tilde{a}_\ray \in \R$ be the coefficients such that 
\begin{equation*}
\divp{\ds} =\left\{ m \in \latm_\R \suchthat \langle m, \rg \rangle \geq -\tilde{a}_\ray \text{ for all } \ray \in \fan (1) \right\}
\end{equation*} 
and all the inequalities are tight on $\divp{\ds}$, i.e., for every $\ray \in \fan(1)$ there exists some point $m \in \divp{\ds}$ such that $\langle m, \rg \rangle =  -\tilde{a}_\ray $. \\
\\
Set $\zp{\ds}{} \coloneqq \sum_{\ray\in \fan(1)} \tilde{a}_\ray \ds_\ray$ and $\zn{\ds}{} \coloneqq \sum_{\ray\in \fan(1)} (a_\ray- \tilde{a}_\ray) \ds_\ray$. Then
\begin{equation*}
\ds= \sum_{\ray\in \fan(1)}a_\ray \ds_\ray = \zp{\ds}{} +\zn{\ds}{}
\end{equation*} 
and $(a_\ray-\tilde{a}_\ray) \geq 0$ by definition.  We now show that the divisors satisfy 1.-3.
\begin{enumerate}
\item Since $\var$ is a surface, the divisor $\zp{\ds}{}$ is nef by construction. 

\item Let $\ds_\ray$ be a curve with $\ds_\ray.\ds_\ray \geq 0$ for some $\ray \in \fan(1)$. There exists a vector $\dir \in M$ such that $v$ is orthogonal to $u_{\ray'}$ and $\langle v,u_\ray \rangle < 0$, where $\ray'$ is a ray adjacent to $\ray$. Then the inequality  corresponding to $\ray$ is tight on $\divp{\ds}$, i.e., $a_\ray = \tilde{a}_\ray$, because otherwise the polytope $\divp{\ds}$ would be unbounded in the direction of $v$. Thus only negative curves will appear in $\zn{\ds}{}$.\\
\\
The matrix $(\zdn{i} . \zdn{j})_{i,j}$ is negative definite if all leading principal minors of ${(-\zdn{i} . \zdn{j})_{i,j}}$ are positive. Label the negative curves that appear in the negative part $\zn{\ds}{}$ as $\left\{N_1,\ldots,N_\ell \right\}$ in such a way that adjacent rays are given consecutive indices counter-clockwise. Then the intersection matrix $(\zdn{i} . \zdn{j})_{i,j}$ is a block matrix, where each block is of the form~(\ref{eqn_blockmatrix}) as in Lemma~\ref{lem_intnum}. Let ${\left\{N_1,\ldots,N_k \right\} \subseteq \left\{N_1,\ldots,N_\ell \right\}}$ be adjacent negative curves that form a sub block $(-\zdn{i} . \zdn{j})_{1 \leq i,j \leq k}$ of the matrix $(-\zdn{i} . \zdn{j})_{1 \leq i,j \leq \ell}$ and denote by $\cur_0$ and $\cur_{k+1}$ the remaining curves whose rays are adjacent to $\ray_1$ and $\ray_k$ as indicated in Figure~\ref{fig_negcurves}.

\begin{figure}[h]
\begin{center}
    \includegraphics[scale=0.17]{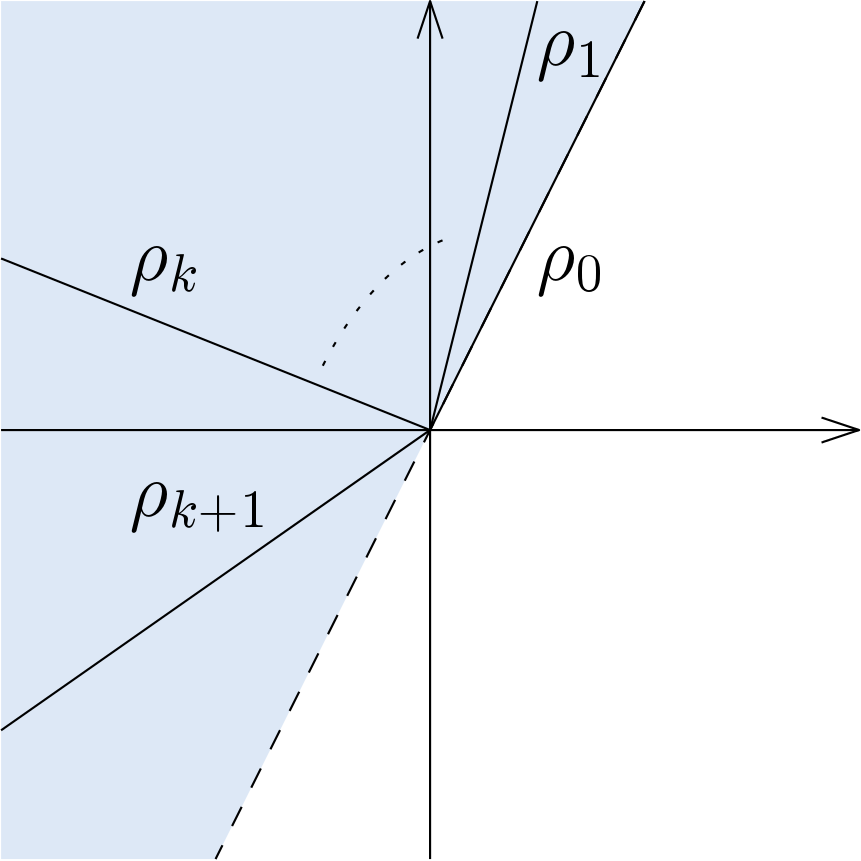}
\caption{Adjacent rays $\ray_0,\ldots,\ray_{k+1}$ of the prime divisors $\cur_0,\ldots,\cur_{k+1}$. }\label{fig_negcurves}
\end{center}
\end{figure}

For the ray generators we have ${u_1,\ldots,u_k \in \cone(u_0,u_{k+1})}$ and that $\cone(u_0,u_{k+1})$ is convex, for otherwise the polytope $\divp{\ds}$ would be unbounded in the direction of $v'$, where $v' \in M$ is chosen to be orthogonal to $u_0$ and fulfill ${\langle v',u_{k+1} \rangle >0}$. \\
\\
Thus it remains to show that the determinant of each such sub block matrix is positive. According to Lemma~\ref{lem_intnum} we have  
\begin{equation*}
\det(-N_i.N_j)_{1 \leq i,j \leq k} = \det(u_0,u_{k+1}).
\end{equation*}
Let $u_0=(m_1,m_2)$ and $u_{k+1}=(m_1',m_2')$, and assume without loss of generality that $m_1>0$. Since $\cone(u_0,u_{k+1})$ is convex and $u_1,\ldots,u_k \in \cone(u_0,u_{k+1})$, we have $m_2' > \frac{m_2}{m_1}m_1'$, because otherwise the polytope $\divp{\ds}$ would be unbounded. It follows that 
\begin{equation*}
\det(u_0,u_{k+1})= \det \begin{pmatrix}
m_1 &m_1' \\ m_2 &m_2' 
\end{pmatrix}
= m_1m_2'-m_1'm_2 > 0.
\end{equation*}

A similar argument works for $m_1 \leq 0$. Thus altogether, we have that $(\zdn{i} . \zdn{j})_{i,j}$ is negative definite, since all sub block matrices of $(-\zdn{i} . \zdn{j})_{i,j}$ have a positive determinant. 
\item Let $\ray \in \fan(1)$ be a ray for which $\ds_\ray$ appears in the negative part $\zn{\ds}{}$ of the decomposition. Then by construction of $\zp{\ds}{}$ its corresponding face $\face_\ray \preceq \divp{\zp{\ds}{}}$ is a vertex. It follows that
\begin{equation*}
\zp{\ds}{}.\ds_\ray=|\face_\ray \cap \latm|-1=0,
\end{equation*}
when $\divp{\ds}$ is a lattice polytope. A similar argument works in the non-integral case using $\llen_\latm(\face_\ray)$. 

\end{enumerate}
The above gives the existence of a Zariski decomposition. It remains to show uniqueness of $\zn{\ds}{}$. Assume we have a decomposition
\begin{equation*}
\ds=\zp{\overline{\ds}}{}+\zn{\overline{\ds}}{}=\sum_{\ray\in \fan(1)}{\overline{a}_\ray\ds_\ray}+\sum_{\ray\in \fan(1)}{(a_\ray-\overline{a}_\ray)\ds_\ray}.
\end{equation*} 
Since $\zn{\overline{\ds}}{}$ is supposed to be effective and $\zp{\overline{\ds}}{}$ is supposed to be nef which translates into only tight inequalities for $\divp{\zp{\overline{\ds}}{}}$, we have $\overline{a}_\ray \leq \tilde{a}_\ray$ for all $\ray \in \fan(1)$. Let $\ray \in \fan(1)$ be the ray of a divisor $\ds_\ray$ that appears in the negative part $\zn{\overline{\ds}}{}$. Then as argued before this has to be a negative curve. But due to 3. the corresponding face $\face_\ray$ of $\divp{\zp{\overline{\ds}}{}}$ has to be a vertex and therefore it follows that ${\overline{a}_\ray=\tilde{a}_\ray}$. This yields uniqueness of  $\zn{\overline{\ds}}{}$. 
\end{proof}

\subsection{The `Tilting-Isomorphism' for Newton-Okounkov Bodies}\label{sec_tilting}

Although, we can always assume the divisor $\ds$ to be torus-invariant, the shape of the Newton--Okounkov body $\nob{\flag}{\ds}$ will heavily depend on the flag $\flag$ which on the other hand is not necessarily torus-invariant. If the curve $Y_1$ in the flag is determined by an equation of the form $x^v-1=0$ for some primitive $v\in \Z^2$, then we can give a  combinatorial way to compute $\nob{\flag}{\ds}$.

\begin{proposition}\label{prop_mv}
Let $\var$ be a smooth projective toric surface, $\ds$ a big divisor, and \linebreak ${\flag \colon \var \supseteq \cur \supseteq \{\point \}}$ an admissible flag on $\var$, where the curve is ${\cur=\overline{\{x \in \tor \suchthat x^v=1\}}}$ for some primitive $\dir \in \Z^2$ and $\point$ a general smooth point on $\cur$. 
Then the associated function $\beta(t)$ in Theorem~\ref{thm_function} is given by
\begin{eqnarray}
\beta( t) &=& \zp{(\ds-t \cur)}{}.\cur \label{eqn_beta1} \\ 
	&=& \zp{(\ds-t \cur')}{}.\cur \label{eqn_beta2} \\ 
	&=& \mv \left(  \divp{\zp{(\ds-t\cur')}{} }, \np(x^v-1) \right) \label{eqn_beta3} \\
	&=&  \mv \left( \divp{\ds }  \cap  (\divp{\ds} +tv ), \np(x^v-1) \right)  \label{eqn_beta4}
\end{eqnarray}
for $0 \leq t \leq \mu $, where $\cur'$ is a torus-invariant curve that is linearly equivalent to $C$. 
\end{proposition}

\begin{proof}
Since the Newton--Okounkov body only depends on the numerical equivalence class, we may assume that the divisor $\ds$ is torus-invariant, i.e., ${\ds= \sum_{\ray \in \fan(1)} \dc_\ray \ds_\ray}$, where $\fan$ is the fan associated to $\var$.\\ 
\\
From Theorem~\ref{thm_function} we know that ${\beta(t) = \zp{(\ds-t \cur)}{}.\cur}$
for ${\nu \leq t \leq \mu}$, where $\zp{(\ds-t \cur)}{}$ is the positive part of the Zariski decomposition of ${\ds-t \cur}$ and for $\cur$ not part of $\zn{(\ds-t \cur)}{}$ we have ${\nu=0}$. Theorem~\ref{thm_zd} states that the decomposition is unique up to the numerical equivalence class of the given divisor. Let 
\begin{equation*}
\cur'= \sum_{\ray \in \fan(1)}{- \min_{m \in \supp(x^v-1) }{  \langle m,\rg \rangle } \ds_\ray}
\end{equation*}
be the torus-invariant curve given in Proposition~\ref{div_koeff}. This means ${\cur' \sim \cur}$ and the curves are in particular numerically equivalent which yields~(\ref{eqn_beta2}).
Due to Sections~5.4/5.5 in~\cite{fulton1993introduction} the intersection product of two curves equals the mixed volume of the associated Newton polytopes and therefore we have~(\ref{eqn_beta3}).\\
\\
To verify the remaining equality, we show that 
\begin{equation*}
\divp{\zp{(\ds-t\cur')}{} } = \divp{\ds }  \cap  (\divp{\ds} +tv )
\end{equation*}
holds  up to translation. For the torus-invariant curve $\ds- t \cur'$ the construction of its Zariski decomposition as in Theorem~\ref{thm_toriczd} guarantees the equality
\begin{equation*}
\divp{ \zp{(\ds - t \cur')}{}} = \divp{(\ds - t \cur')}
\end{equation*}
for the corresponding polytopes. Consider its translation by $tv$, this gives
\begin{align*}
&\divp{(\ds - t \cur')} + tv  \\
&= \left\{ m +tv \in \latm_\R  \suchthat \langle m, \rg \rangle  \geq - \left(\dc_\ray + t\cdot \min{( 0, \langle v,\rg \rangle )  }\right) \text{ for all } \ray \in \fan (1) \right\} \\
	&=  \left\{ m \in \latm_\R  \suchthat \langle m -tv, \rg \rangle  \geq - \left(\dc_\ray + t\cdot \min{( 0, \langle v,\rg \rangle )  }\right) \text{ for all } \ray \in \fan (1)  \right\} \\ 
	&= \left\{ m \in \latm_\R  \suchthat \langle m, \rg \rangle  \geq - \dc_\ray - t \cdot \min{( 0, \langle v,\rg \rangle )  } +t \langle v,\rg \rangle  \text{ for all } \ray \in \fan (1) \right\} \\
&= \left\{ m \in \latm_\R  \suchthat \langle m, \rg \rangle  \geq - \dc_\ray + \max{( 0, t \langle v,\rg \rangle ) } \text{ for all } \ray \in \fan (1) \right\}. \\ \ 
\end{align*}
On the other hand, we have
\begin{equation*}
\divp{\ds} = \left\{ m  \in \latm_\R  \suchthat \langle m, \rg \rangle  \geq - \dc_\ray \text{ for all } \ray \in \fan (1)  \right\} 
\end{equation*}
and 
\begin{equation*}
\divp{\ds} + tv  = \left\{ m  \in \latm_\R  \suchthat \langle m, \rg \rangle  \geq - \dc_\ray +t \langle v,\rg \rangle \text{ for all } \ray \in \fan (1)  \right\}. 
\end{equation*}
Thus their intersection is the set
\begin{align*}
& \divp{\ds} \cap \left( \divp{\ds} + tv \right) \\ 
&=\left\{ m  \in \latm_\R  \suchthat \langle m, \rg \rangle  \geq - \dc_\ray \text{ and } \langle m, \rg \rangle  \geq - \dc_\ray +t \langle v,\rg \rangle \text{ for all } \ray \in \fan (1) \right\} \\
	&= \left\{ m \in \latm_\R  \suchthat \langle m, \rg \rangle  \geq - \dc_\ray + \max{( 0, t \langle v,\rg \rangle )  } \text{ for all } \ray \in \fan (1) \right\}.
\end{align*}
This verifies equality in~(\ref{eqn_beta4}). 
\end{proof}

\begin{example}\label{ex_hir1}
We return to the Hirzebruch surface $\var = \hir_1$ from Example~\ref{ex_build_nob}, and consider the big divisor  ${\ds= \ds_3 +2\ds_4}$ on $\var$. Then for any admissible  torus-invariant flag $\flag'$ the associated Newton--Okounkov body $\nob{\flag'}{\ds}$ coincides with a translate of the polytope $\divp{\ds}$ which can be seen in Figure~\ref{fig_verschieben}.\\
\\
We wish to determine the Newton--Okounkov body $\nob{\flag}{\ds}$ given by a different flag \linebreak ${\flag \colon \var \supseteq \cur \supseteq \{\point\}}$, where ${\cur = \overline{\{(x,y) \in \tor \suchthat y^{-1}-1=0\}}}$ is a non-invariant curve, and $\point$ is a general smooth point on $C$. In local coordinates the curve $\cur$ is given by the binomial ${y^{-1}-1}$ for ${v=(0,-1)}$ and has the line segment ${\np{(\cur)}= \conv((0,0),(0,-1))}$ as its Newton polytope. Using Proposition~\ref{div_koeff} we obtain the torus-invariant curve ${\cur'=\ds_2}$ which is linearly equivalent to $\cur$. \\
\\
To determine the Newton--Okounkov body we use variation of Zariski decomposition for the  divisor ${\ds_t=\ds-t \cur}$. To compute the upper part of the Newton--Okounkov body in terms of the piecewise linear function $\beta$, we move a copy of the polytope $\divp{\ds}$ in the direction of $v$ as indicated in Figure~\ref{fig_verschieben}.

\begin{figure}[h!]
\begin{center}
    \includegraphics[scale=0.11]{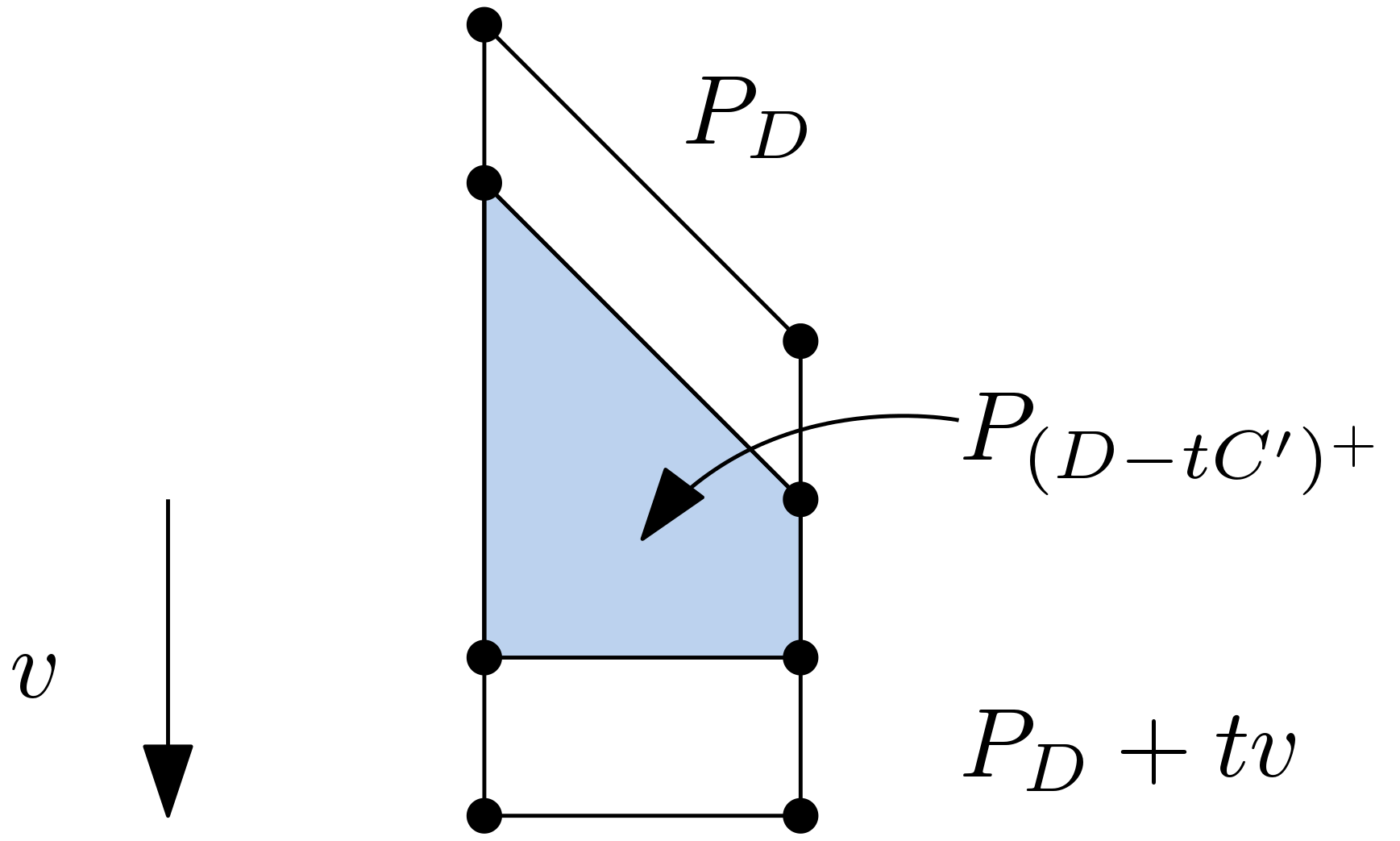}
\caption{Moving a copy of $\divp{\ds}$ in the direction of $v$ to obtain $\divp{\ds} \cap (\divp{\ds}+tv)=\divp{\zp{(\ds-t\cur')}{}}$.} \label{fig_verschieben}
\end{center}
\end{figure}

The intersection $\divp{\ds} \cap (\divp{\ds}+tv)$ gives the polytope associated to $\divp{\zp{(\ds-t\cur')}{}}$. By Proposition~\ref{prop_mv} the function $\beta$ is then given as  

\begin{eqnarray*}
\beta(t) &=& \zp{\ds}{t}. \cur = \mv \left( \divp{\ds }  \cap  (\divp{\ds} +t \cdot (0,-1) ), \np(y^{-1}-1) \right) \\
	&=& \begin{cases} 1 & \text{if } 0 \leq t \leq 1 \\ 2-t & \text{if } 1 \leq t \leq 2,
		\end{cases} 
\end{eqnarray*}

where the mixed volume $\mv \left( \divp{\ds }  \cap  (\divp{\ds} +t \cdot (0,-1) ), \np(y^{-1}-1) \right)$ can be seen as the area of the shaded region in Figure~\ref{fig_mv}.

\begin{figure}[h!]
\begin{center}
    \includegraphics[scale=0.2]{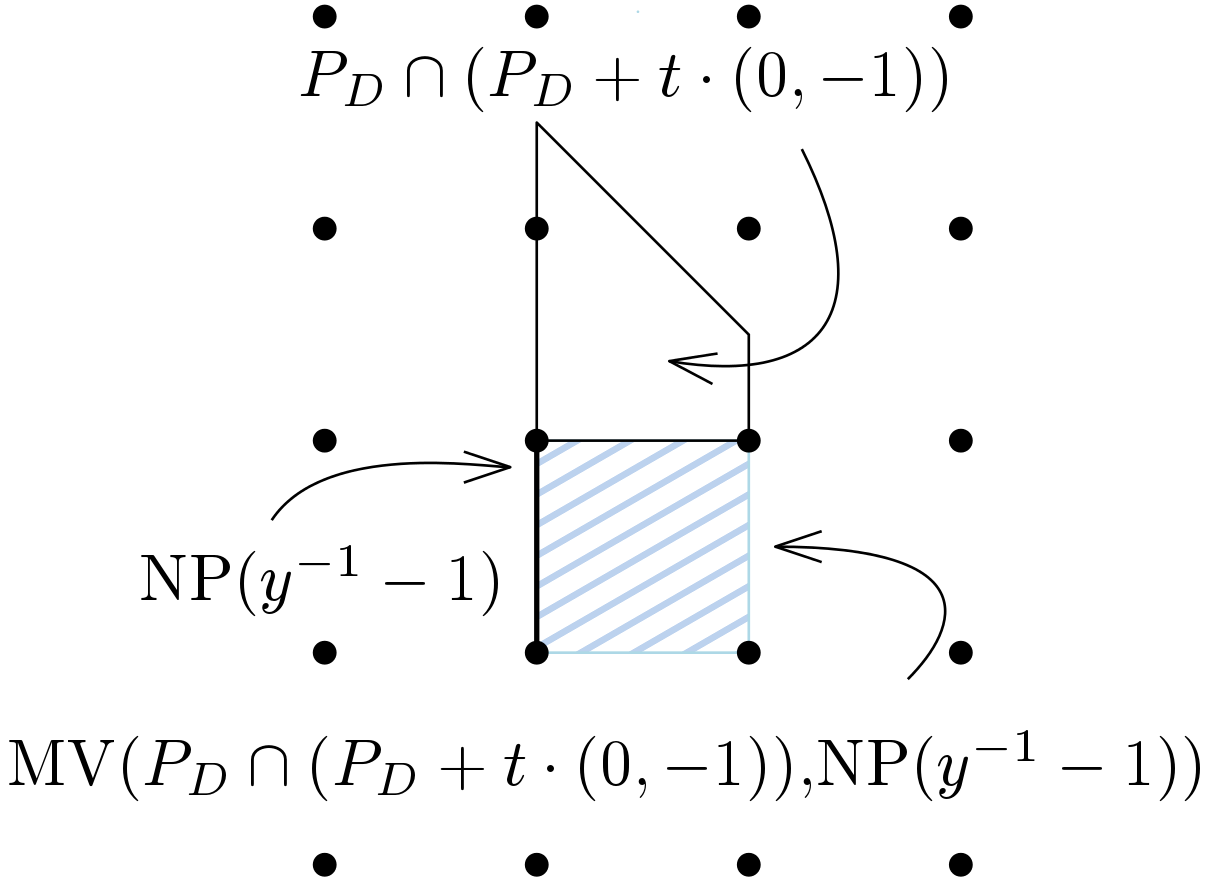}
\caption{The mixed volume $\mv \left( \divp{\ds }  \cap  (\divp{\ds} +t \cdot (0,-1) ), \np(y^{-1}-1) \right)$.} \label{fig_mv}
\end{center}
\end{figure}

Since $\ds$ is nef, we have $\nu= 0$ and since $\point$ can be chosen general enough on $\cur$, we also have $\alpha(t) \equiv 0$. Therefore the Newton--Okounkov body $\nob{\flag}{\ds}$ is the polytope shown in Figure~\ref{fig_nob_hir1}.

\begin{figure}[h!]
\begin{center}
    \includegraphics[scale=0.07]{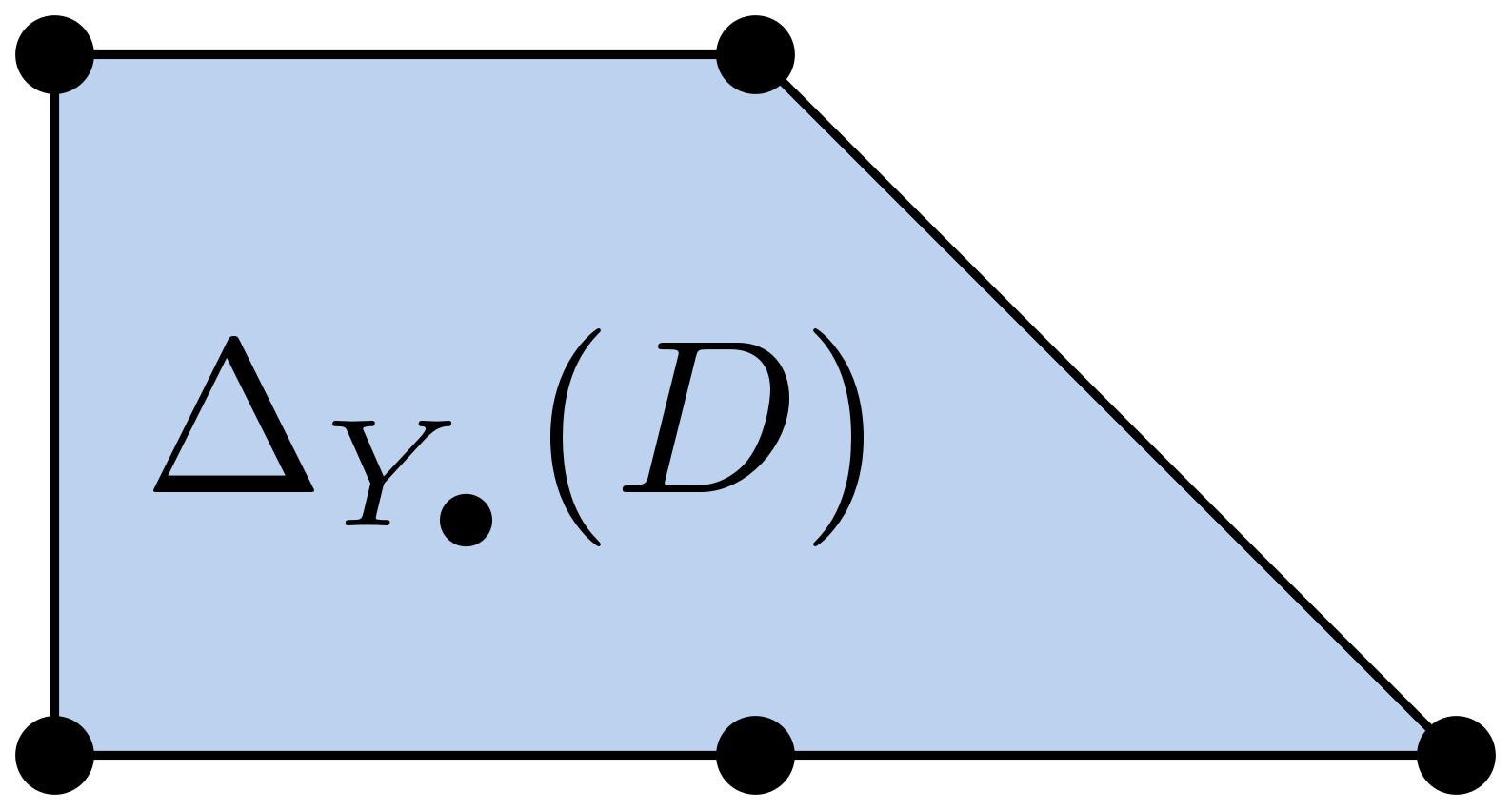}
\caption{The Newton--Okounkov body $\nob{\flag}{\ds}$.} \label{fig_nob_hir1}
\end{center}
\end{figure}

\end{example}

Let us for simplicity assume that $\nu=0$ and that $\alpha \equiv 0$. Then the Newton--Okounkov body $\nob{\flag}{\ds}$ is completely determined by $\beta$.\\
\\
Given the polytope $\divp{\ds}$ and the vector $v$, the procedure described in Proposition~\ref{prop_mv} to compute the function $\beta$ divides the polytope $\divp{\ds}$ into chambers. In the following we consider this process in detail. For that we introduce the following definition. 

\begin{definition}
Let $P \subseteq \R^2$ be a $2$-dimensional polytope and let $\dir \in \R^2$ be a direction. Then we call a facet $\face \preceq P$ \emph{sunny} with respect to $v$ if $\langle v, \nor{\face} \rangle > 0$, where $\nor{\face}$ is the inner facet normal of $\face$. We call the set of all sunny facets of $P$ with respect to $v$ the \emph{sunny side} of $P$ with respect to $v$ and denote it by $\sun{P}{v}$.  
\end{definition}

Let $\sun{\divp{\ds}}{v}$ be the sunny side of $\divp{\ds}$ with respect to $v$. By construction the function $\beta$ is piecewise linear. There is a break point at time $\tilde{t} \geq 0$ if and only if there exists a vertex $\vertex \in \ver(\divp{\ds})$ such that 
\begin{equation*}
\vertex \in \divp{\ds} \cap (\sun{\divp{\ds}}{v}+\tilde{t}v).
\end{equation*}   

Thus we move the sunny side $\sun{\divp{\ds}}{v}$ along the polytope $\divp{\ds}$ in the direction of $v$. We start at time $t_0=0$. Whenever we hit a vertex $\vertex_i \in \ver(\divp{\ds})$ at time $t_i$, we enter a new chamber as indicated in Figure~\ref{fig_movev}. 

\begin{figure}[h!]
\begin{center}
    \includegraphics[scale=0.2]{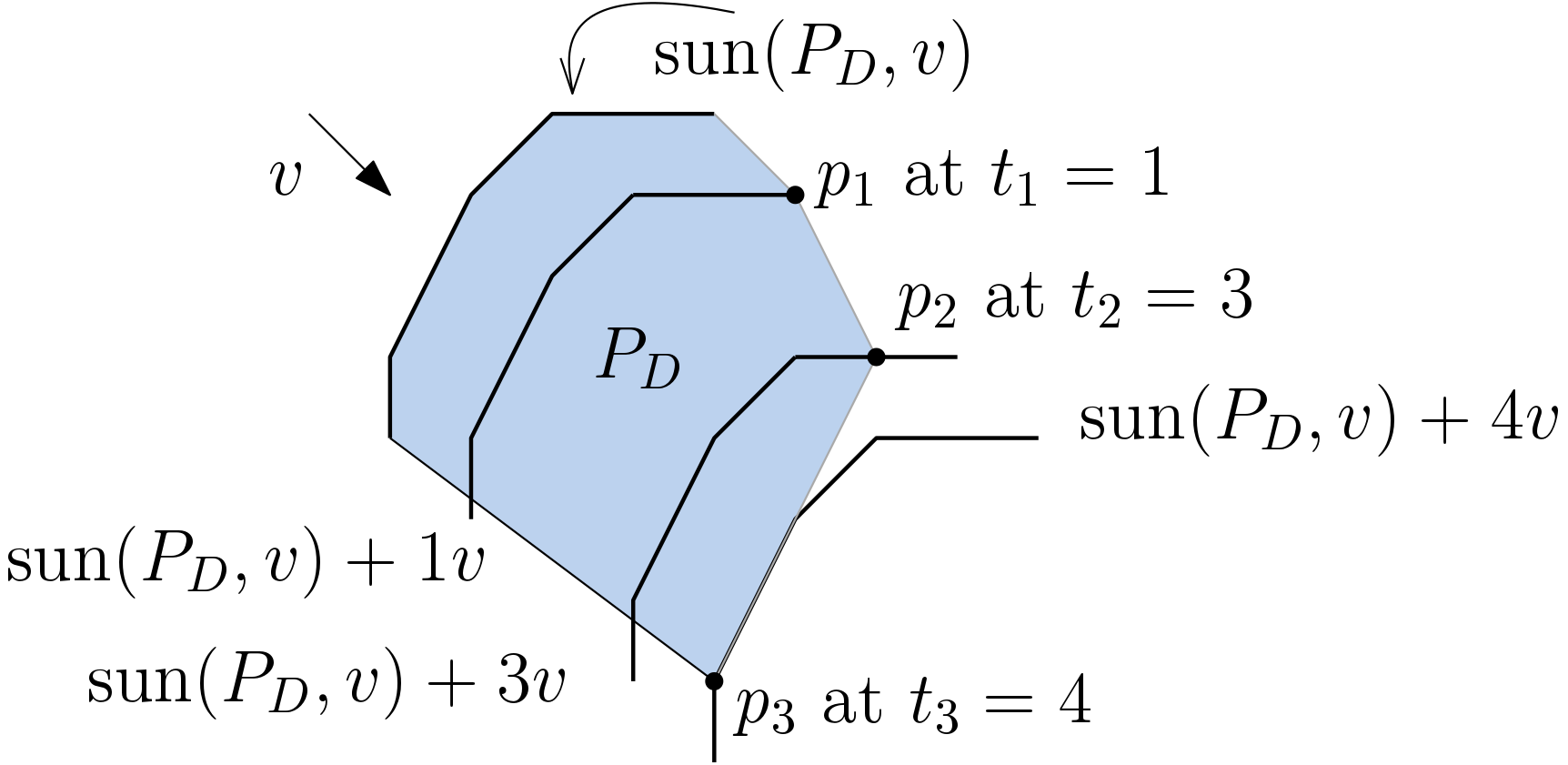}
\caption{Break points $\vertex_1,\vertex_2$, and $\vertex_3$ of shifting the sunny side $\sun{\divp{\ds}}{v}$ through $\divp{\ds}$ in the direction of $v$.} \label{fig_movev}
\end{center}
\end{figure}

Then  $\beta(t)$ is linear in each time interval $\left[ t_i,t_{i+1} \right]$ for $i  \in \Z_{\geq 0}$. \\
\\
The other part of the chamber structure comes from inserting a wall in  the direction of $v$ for each vertex $\vertex \in \sun{\divp{\ds}}{v}$ and in the direction of $-v$ for each vertex $\vertex \in \sun{\divp{\ds}}{-v}$ as it can be seen in Figure~\ref{fig_subdiv}.

\begin{figure}[h!]
\begin{center}
    \includegraphics[scale=0.14]{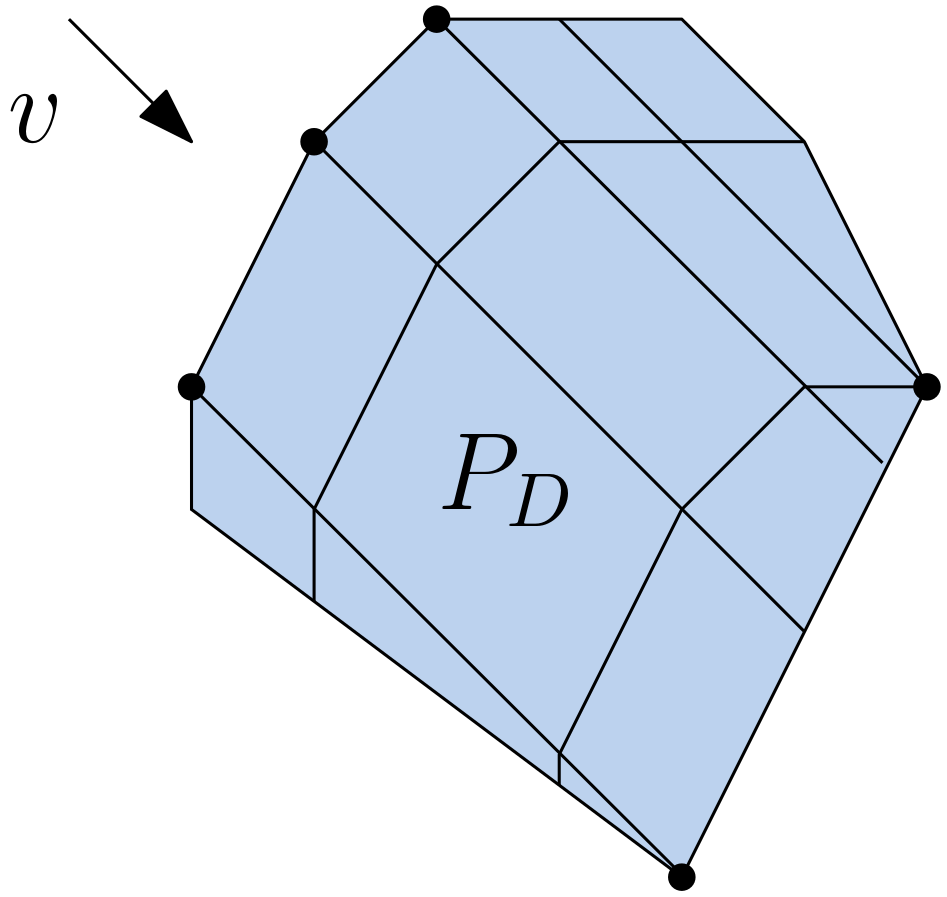}
\caption{The chamber structure on $\divp{\ds}$ induced by the shifting process.} \label{fig_subdiv}
\end{center}
\end{figure}

In the following, we verify that for this particular chamber structure there exists a map between $\divp{\ds}$ and $\nob{\flag}{\ds}$ that is linear on each of the chambers.\\
\\
For that, we choose a coordinate system $m_1,m_2$ for $\latm_\R \cong \R^2$ such that $v=(1,0)$ without loss of generality. Consider the polytope $\divp{\ds} \subseteq \R^2 $ in $(m_1,m_2)$-coordinates, and assume without loss of generality that $\divp{\ds}$ lies in the positive orthant. We can write it as
\begin{equation*}
\divp{\ds}=\left\{ (m_1,m_2) \in \R^2 \suchthat \gamma \leq m_2 \leq \delta, \ \ell(m_2) \leq m_1 \leq r(m_2) \right\},
\end{equation*} 
for some $\gamma,\delta \in \R$ and some piecewise linear functions $\ell$ and $r$ that determine the sunny sides $\sun{\divp{\ds}}{v}$ and $\sun{\divp{\ds}}{-v}$, respectively.
To determine the function $\beta$ using the combinatorial approach from Proposition~\ref{prop_mv}, we shift the sunny side $\sun{\divp{\ds}}{v}$ through the polytope as depicted in Figure~\ref{fig_movev}. Now we want to `tilt the polytope leftwards' such that the $m_1$-coordinate of each point in the image expresses exactly the time at which the point in the original polytope is visited in the shifting process. This is shown in Figure~\ref{fig_left}.
To make this precise, map the polytope $\divp{\ds}$ via
\begin{eqnarray*}
\Psi_\text{left} \colon \divp{\ds}\subseteq \R^2 & \to & \R^2 \\
(m_1,m_2) & \mapsto & (m_1-\ell(m_2),m_2).
\end{eqnarray*}

\begin{figure}[h!]
\begin{center}
    \includegraphics[scale=0.18]{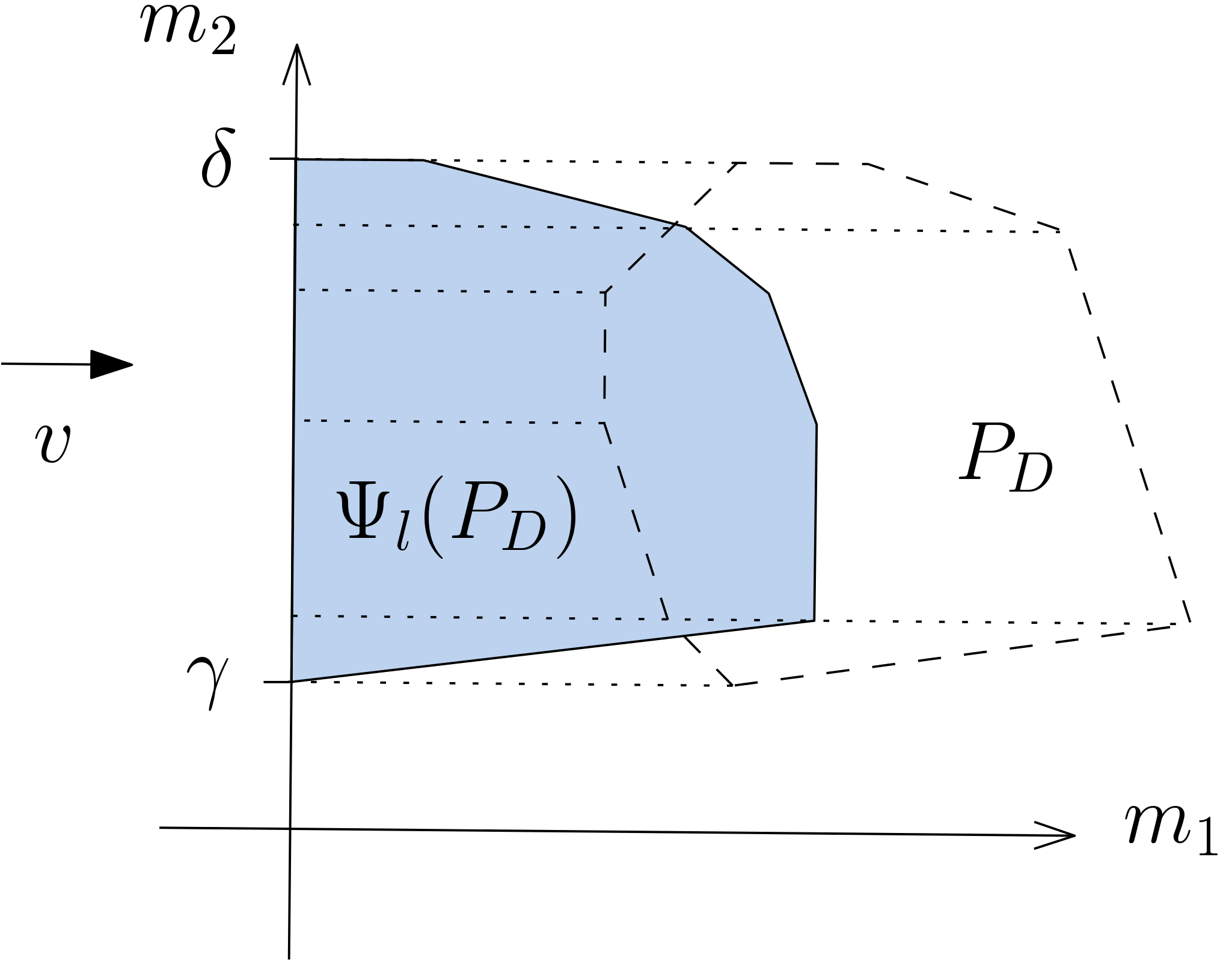}
\caption{Tilting the polytope $\divp{\ds}$ leftwards via the map $\Psi_{\text{left}}$.} \label{fig_left}
\end{center}
\end{figure}

By construction, the map $\Psi_\text{left}$ is a piecewise shearing of the original polytope and therefore volume-preserving. Additionally, $\Psi_\text{left}(\divp{\ds}) \cap \{m_1=t\}$ are exactly the images of the points of $\divp{\ds}$ that are visited at time $t$. 
Given $m_1=t$ we now want to determine $\beta(t)$. According to~(\ref{eqn_beta4}) it is given by 
\begin{eqnarray*}
\beta(t) &=& \mv \left( \divp{\ds }  \cap  (\divp{\ds} +tv ), \np(x^v-1) \right) \\
&=& \mv \left( \divp{\ds }  \cap  (\sun{\divp{\ds}}{v} +tv ), \np(x^v-1) \right) \\
&=& \mv \left( \Psi_\text{left}(\divp{\ds}) \cap \{m_1=t\}, \np(x^v-1) \right) \\
&=& \llen_{\latm}(\Psi_\text{left}(\divp{\ds}) \cap \{m_1=t\}).
\end{eqnarray*}
The last equation holds, since $v$ was chosen to be $(1,0)$. 
In the last step we want to `tilt the polytope downwards' similarly to the previous process as can be seen in Figure~\ref{fig_down}. Therefore we can describe the polytope $\Psi_\text{left} (\divp{\ds})$ as
\begin{eqnarray*}
\Psi_\text{left} (\divp{\ds}) &=& \left\{ (m_1,m_2) \in \R^2 \suchthat \gamma \leq m_2 \leq \delta, \ 0 \leq m_1 \leq r(m_2)-\ell(m_2) \right\} \\
&=& \left\{ (m_1,m_2) \in \R^2 \suchthat 0 \leq m_1 \leq \hat{\delta}, \ \hat{\ell}(m_1) \leq m_2 \leq \hat{r}(m_2) \right\},
\end{eqnarray*}
for some $\hat{\delta} \in \R$ and some piecewise linear functions $\hat{\ell}$ and $\hat{r}$ that determine the bottom and top of the polytope.\\
\\
So set
\begin{eqnarray*}
\Psi_\text{down} \colon \Psi_\text{left} (\divp{\ds})\subseteq \R^2 & \to & \R^2 \\
(m_1,m_2) & \mapsto & (m_1,m_2-\hat{\ell}(m_1)).
\end{eqnarray*}

\begin{figure}[h!]
\begin{center}
\includegraphics[scale=0.18]{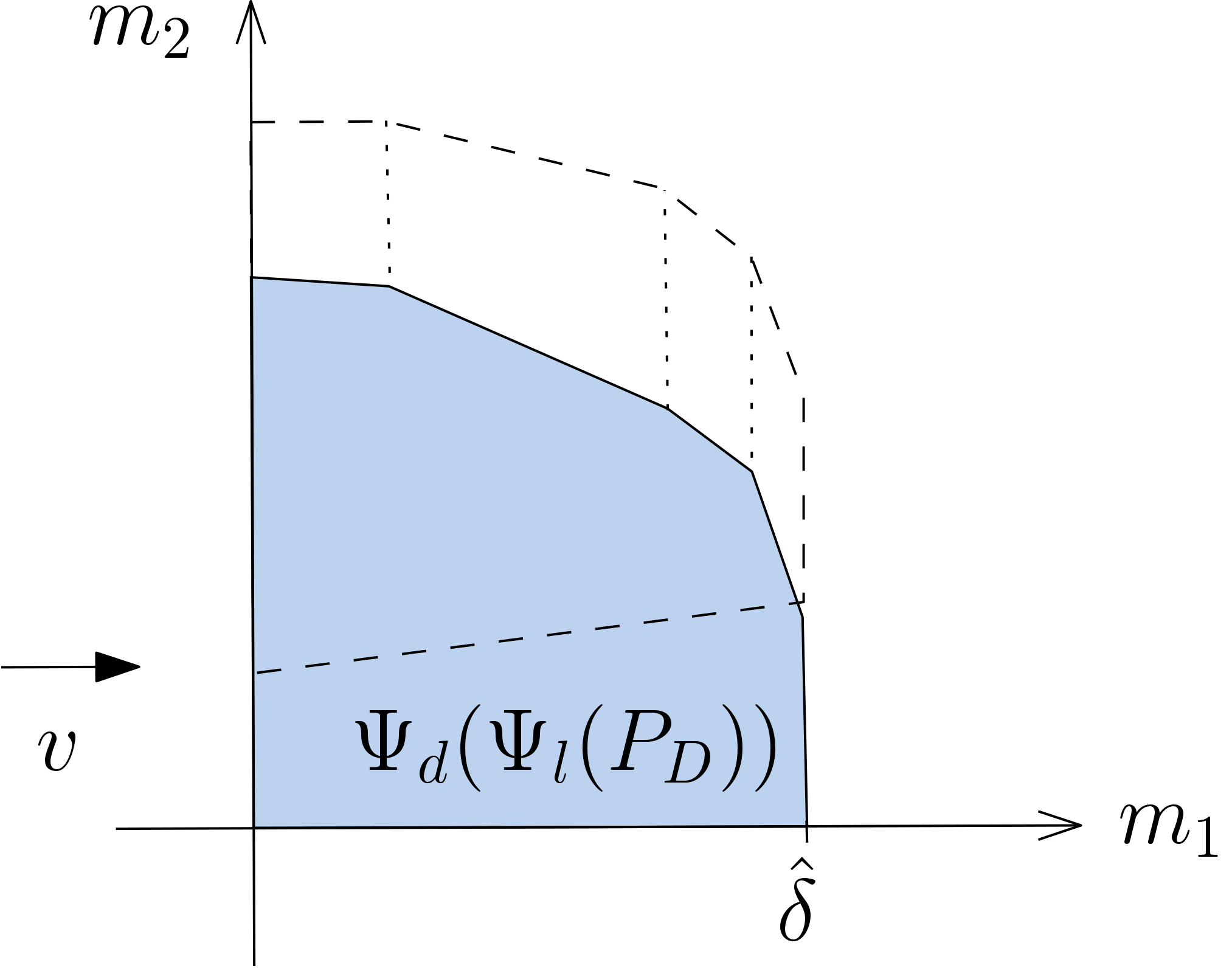}
\caption{Tilting the polytope $\Psi_{\text{left}}(\divp{\ds})$ downwards via the map $\Psi_{\text{down}}$.} \label{fig_down}    
\end{center}
\end{figure}

By construction this is again a piecewise shearing of the polytope and therefore volume-preserving. The image $\Psi_\text{down}(\Psi_\text{left} (\divp{\ds}))$ is the subgraph of $\beta$ and thus it coincides with the Newton--Okounkov body $\nob{\flag}{\ds}$ with respect to the new flag $\flag$.\\
\\
The above shows the following.

\begin{corollary}\label{cor_mv}
Let $\var$ be a smooth projective toric surface, $\ds$ a big divisor,
and \linebreak $\flag \colon \var \supseteq \cur \supseteq \{\point\}$ an
admissible flag on $\var$, where the curve $\cur$ is given by a
binomial $x^v-1$ for a primitive $\dir \in \Z^2$ and $\point$ is a
general smooth point on $\cur$. Then there exists a piecewise linear,
volume-preserving isomorphism $\pwl=\Psi_{\operatorname{down}}\circ \Psi_{\operatorname{left}}$ between the two Newton--Okounkov bodies
$\divp{\ds}$ and $\nob{\flag}{\ds}$.
\end{corollary}

Moreover, the image under $\pwl$ can explicitly be described in terms of measurements of the polytope. For that we need to introduce more terminology. 

\begin{definition}
Let $\pol \subseteq \R^\dimm$ be a polytope, $\dir \in \Z^\dimm$ a primitive vector, and $\linf \in \dir^\perp$ a primitive integral functional.

For a point $\op \in \pol$ we define the
\emph{length of $\pol$ at $\op$ with respect to $\dir$} to be
\begin{equation*}
\relw{\pol,\op}{\dir}\coloneqq \max{ \left\{ t \in \R \suchthat \op -
    t\dir \in \pol \right\} }, 
\end{equation*}

and $\relw{\pol}{\dir}$ is the maximal length over all $\op \in \pol$.


Further, we denote by $\feasible{\pol}{\dir}{\op}$ the intersection of
$$ P \cap \left( P + \relw{\pol,\op}{\dir} \cdot \dir \right) $$
with the half plane given by $\langle \linf, \cdot \rangle \ge \langle
\linf, \op \rangle$.

\begin{figure}[htb]
  \centering
  \includegraphics[scale=0.2]{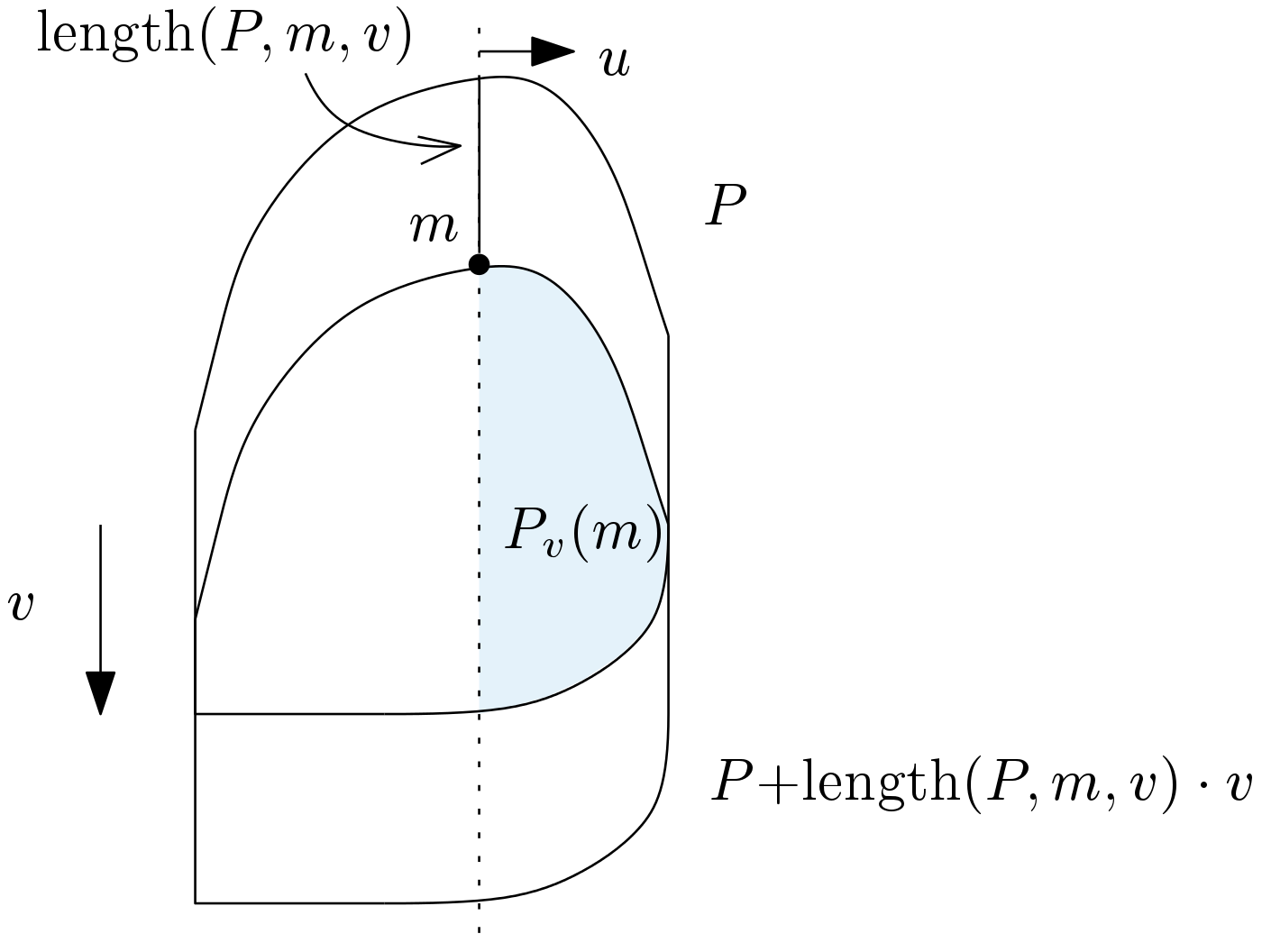}
  \caption{The feasible region in $\pol$ with respect to $\dir$ and $\linf$, given
    $\op \in \pol$.}
  \label{fig:feasible-region}
\end{figure}

\end{definition}

Observe that with the above notation $\pwl(\op) =
(\relw{\pol,\op}{\dir},\wid{\linf}(\feasible{\pol}{\dir}{\op}))$.%

\begin{remark}\label{rem_wallcrossing}
Note, that the piecewise linear, volume-preserving isomorphism $\pwl$ is
reminiscent of the transformation constructed in~\cite{escobar2019wall}. The authors give geometric maps between the Newton--Okounkov bodies corresponding to two adjacent maximal-dimensional prime cones in the tropicalization of the variety $\var$. This can also be studied from the perspective of complexity-one T-varieties. 
\end{remark}


\section{Newton--Okounkov Functions on Toric Varieties}\label{sec_function}

This section examines Newton--Okounkov functions in three settings. To start with, we consider the completely toric case in Subsection~\ref{sec:torictoric} and show that in this case the resulting function will be linear, see Proposition~\ref{prop_distance}. This is related to a result that identifies the subgraph of a Newton--Okounkov function as a certain Newton--Okounkov body. We translate this relation into polyhedral language in Subsection~\ref{sec:subgraph}. Eventually, we return to the surface case in Subsection~\ref{sec:general-point} and examine Newton--Okounkov functions coming from the geometric valuation at a general point and give combinatorial criteria for when we can fully determine the function, see Theorem~\ref{prop_a+b}, Corollary~\ref{cor_par} and Theorem~\ref{thm_fct_zonotopally}. 

\subsection{The Completely Toric Case}
\label{sec:torictoric}
Whenever we determine the value of a Newton--Okounkov function
$\fct(m)$ for a point ${m \in \nob{\flag}{\ds}}$, we will often assume
that $m$ is a valuative point if not mentioned otherwise. 
In the case, when all the given data is toric, we can completely describe the function $\fct_\sv$, and it even has a nice geometric interpretation.

By `all data toric' we mean that $\var$ is a smooth toric variety, $\flag$ is a flag consisting of torus-invariant subvarieties, $\ds$ is a big torus-invariant divisor on $\var$, and ${\sv \subseteq \var}$ a torus-invariant subvariety.\\
\\
In order to formulate and prove Proposition~\ref{prop_distance} below,
we recall the combinatorics of the blow-up ${\pi_\sv \colon \var^* \to
\var}$ of $\sv$.
As $\sv$ is torus-invariant, it corresponds to a cone ${\ldc \in \fan}$
of the fan.
According to~\cite[Definition 3.3.17]{cox2011toric} the fan $\fan^*$
in $\lat_\R$ of the variety $\var^*$ is given by the star subdivision
of $\fan$ relative to $\ldc$. Set ${u_\ldc=\sum_{\ray \in
  \ldc(1)}{u_\ray}}$, ${\ray_\sv = \cone (u_\ldc)}$, and for each cone
${\con \in \fan}$ containing $\ldc$, set
\begin{equation*}
\fan^*_\con(\ldc)=\left\{ \con' + \ray_\sv \suchthat \ldc \nsubseteq \con'
  \subset \con \right\}
\end{equation*}
and the star subdivision of $\fan$ relative to $\ldc$ is the fan
\begin{equation*}
\fan^* = \fan^*(\ldc)=\left\{ \con \in \fan \suchthat \ldc \nsubseteq \con
\right\} \cup \bigcup_{\con \supseteq \ldc}{\fan^*_\con(\ldc)}.
\end{equation*}
Then the exceptional divisor $E$ of the blow-up $\pi_\sv$ corresponds
to the ray $\ray_\sv \in \fan^*$, and the order of vanishing of a
section $s$ along $\sv$ is, by definition, the order of vanishing of
$\pi_\sv^*(s)$ along $E$.
The Cartier data $\{m^*_{\con^*}\}_{\con^* \in \fan^{*}(\dimm)}$ of
$\pi_\sv^*\ds$ is given by $m^*_{\con^*} = m_{\con^*}$ for $\con^* \in
\fan(\dimm)$ (i.e., ${\con^* \not\supseteq \ray_\sv}$), and $m^*_{\con^*} =
m_{\con}$ for $\con^* \in \fan^*_\con(\ldc)(\dimm)$.

\begin{proposition}\label{prop_distance}
Let $\var$ be an $n$-dimensional smooth projective toric variety
associated to the unimodular fan $\fan$ in $\lat_\R$. Furthermore, let
$\flag$ be an admissible torus-invariant flag and $\ds$
a big
torus-invariant divisor on $\var$ with resulting Newton--Okounkov body
$\nob{\flag}{\ds}$.\\
\\
Let $\sv \subseteq \var$ be an irreducible torus-invariant subvariety. Then the geometric valuation $\ord_Z$ yields a linear function $\fct_\sv$ on $\nob{\flag}{\ds}$. More explicitly, it is given by 
\begin{align*}
\fct_\sv  \colon \nob{\flag}{\ds} & \to \R \\
m & \mapsto  \langle m-m_\ldc,u_\ldc \rangle,
\end{align*}   
where $m_\ldc\coloneqq m_\con$ is part of the Cartier data
$\{m_\con\}_{\con \in \fan(\dimm)}$ of $\ds$ for any cone $\con \in
\fan$ containing $\ldc$.
\end{proposition}

This function $\fct_Z$ measures the lattice distance of a given point
$m$ in the Newton--Okounkov body to the hyperplane with equation
$\langle m , u_\ldc \rangle = \langle m_\ldc,u_\ldc \rangle$. 
If $\ds$ is ample
this is the lattice distance to a face of $\nob{\flag}{\ds}$.

\begin{proof}
Since the flag $\flag$ and the divisor $\ds$ are torus-invariant, the
resulting Newton--Okounkov body $\nob{\flag}{\ds}$ coincides with a
translate of the polytope $\divp{\ds}$.\\
\\
We consider the blow-up ${\pi_\sv\colon \var^* \to \var}$ of $\sv$. Let $\flag^*$ denote the proper transform of $\flag$ on $\var^*$. The pullback $\pi_\sv^*\ds$ of the given divisor $\ds$ determines a polytope $\divp{\pi_\sv^*\ds}$ and by construction we have ${\divp{\ds}\cong  \divp{\pi_\sv^*\ds}}$. To embed the Newton--Okounkov body ${\nob{\flag^*}{\pi_\sv^*\ds} \cong \divp{\pi_\sv^*\ds}}$ in $\R^n$ we have to fix a trivialization of the line bundle. Fix the origin $\fn$ of $\R^n$ to be $m_\ldc$. If $m_\ldc \in \divp{\pi_\sv^*\ds}$, this means that the corresponding character $\chi^{\fn}$ is identified with a global section $s$ of $\she_{\var^*}(\pi_\sv^*\ds)$ that does not vanish along $\sv$. \\
\\
Then according to~\cite[Proposition 4.1.1]{cox2011toric} the order of vanishing of a character $\chi^m$ along $\sv$ is given as 
\begin{equation*}
\ord_\sv(\chi^m) = \ord_E(\chi^m)= \langle m, u_\ldc \rangle
\end{equation*} 
for $m \in  \nob{\flag^*}{\pi_\sv^{*}\ds}$. \\
\\
For a given point $m \in \nob{\flag^*}{\pi_\sv^*\ds}$ let $s \in \sps(\var^*,\she_{\var^*}(k \pi_\sv^*\ds))$ be an arbitrary global section that gets mapped to $m$ by the flag valuation associated to $\flag^*$ for some suitable $k \in \N$. Write $s$ in local coordinates $x_i$ with respect to
the flag $\flag^*$, that is, $Y^*_i$ is given by $x_1=\ldots=x_i=0$ and in particular
$\fn = Y^*_n$.\\
\\
The change of coordinates is obtained by multiplication by the monomial $\chi^{m_\ldc}$ on the level of functions and by a respective translation by the vector $m_\ldc \in M$ on the level of points. This yields
\begin{equation*}
\ord_\sv(\chi^m)=\langle m-m_\ldc,u_\ldc \rangle.
\end{equation*}  
The section $s$ is identified with a linear combination of characters, in which $\chi^m$ appears with non-zero coefficient. This gives the upper bound
\begin{equation*}
\ord_\sv(s)= \min_{m' \in \supp(s)} \ord_\sv(\chi^{m'})\leq \ord_\sv(\chi^m). 
\end{equation*}
The lower bound is realized by the monomial $\chi^m$ itself. Hence, the function that comes from the geometric valuation along the subvariety $\sv$ is given as 
\begin{equation*}
\fct_Z(m)=\ord_\sv(\chi^m)=\langle m-m_\ldc,u_\ldc \rangle
\end{equation*}
for $m \in \nob{\flag^*}{\pi^*_\sv\ds}=\nob{\flag}{\ds}$.
\end{proof}

We give an example to illustrate the proof.

\begin{example} \label{eg:torictoric:F1}
As in Example~\ref{ex_hir1} we consider the Hirzebruch surface $\var=\hir_1$, an admissible torus-invariant flag $\flag$, and the big divisor ${\ds=\ds_3+2\ds_4}$. As a torus-invariant subvariety ${\sv \subseteq \var}$ consider the torus fixed point associated to the cone ${\ldc=\cone((-1,0),(0,1))}$. 

\begin{figure}[h!]
\begin{center}
\includegraphics[scale=0.12]{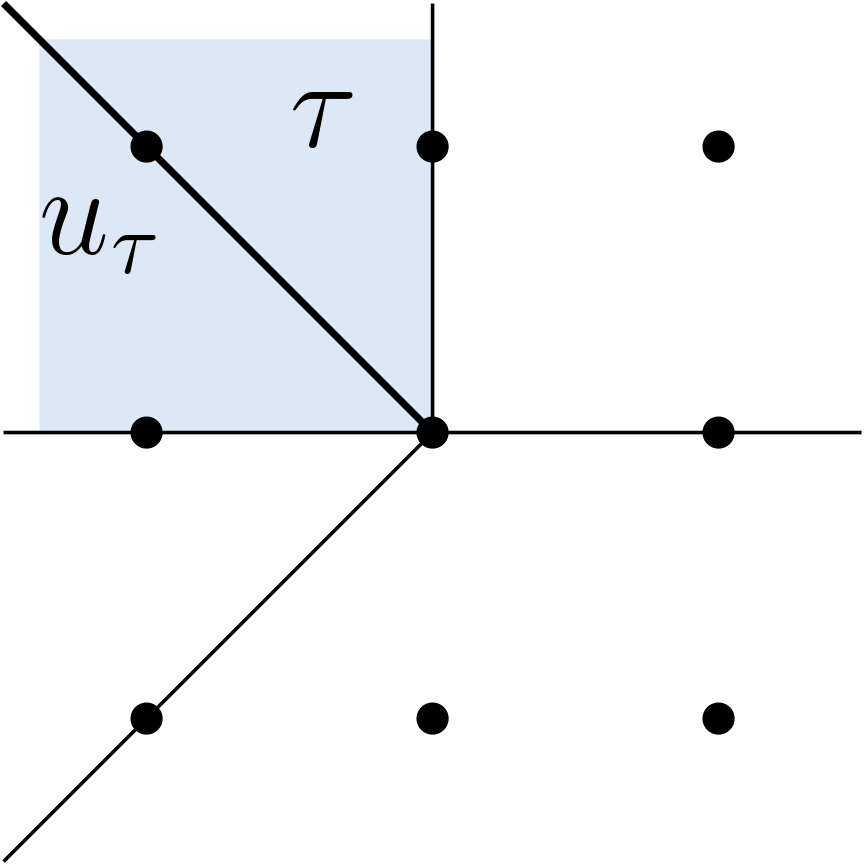}
\caption{Star subdivision of the fan $\fan$ relative to $\ldc$.} \label{fig_star_subd}    
\end{center}
\end{figure}

Then the additional primitive ray generator ${u_\ldc=(-1,1)=(-1,0)+(0,1)}$ for the fan $\fan^*$ comes from the star subdivision of the fan $\fan$ relative to the cone $\ldc$ as indicated in Figure~\ref{fig_star_subd}. The Newton--Okounkov function $\fct_\sv$ on the Newton--Okounkov body $\nob{\flag^*}{\pi^*_\sv\ds}$ is given by 
\begin{equation*}
\fct_\sv(m)= \langle m-m_\ldc,u_\ldc\rangle= \langle m- (1,0),(-1,1)\rangle, 
\end{equation*} 
which gives the values shown in Figure~\ref{fig_latdist}.

\begin{figure}[h!]
\begin{center}
\includegraphics[scale=0.13]{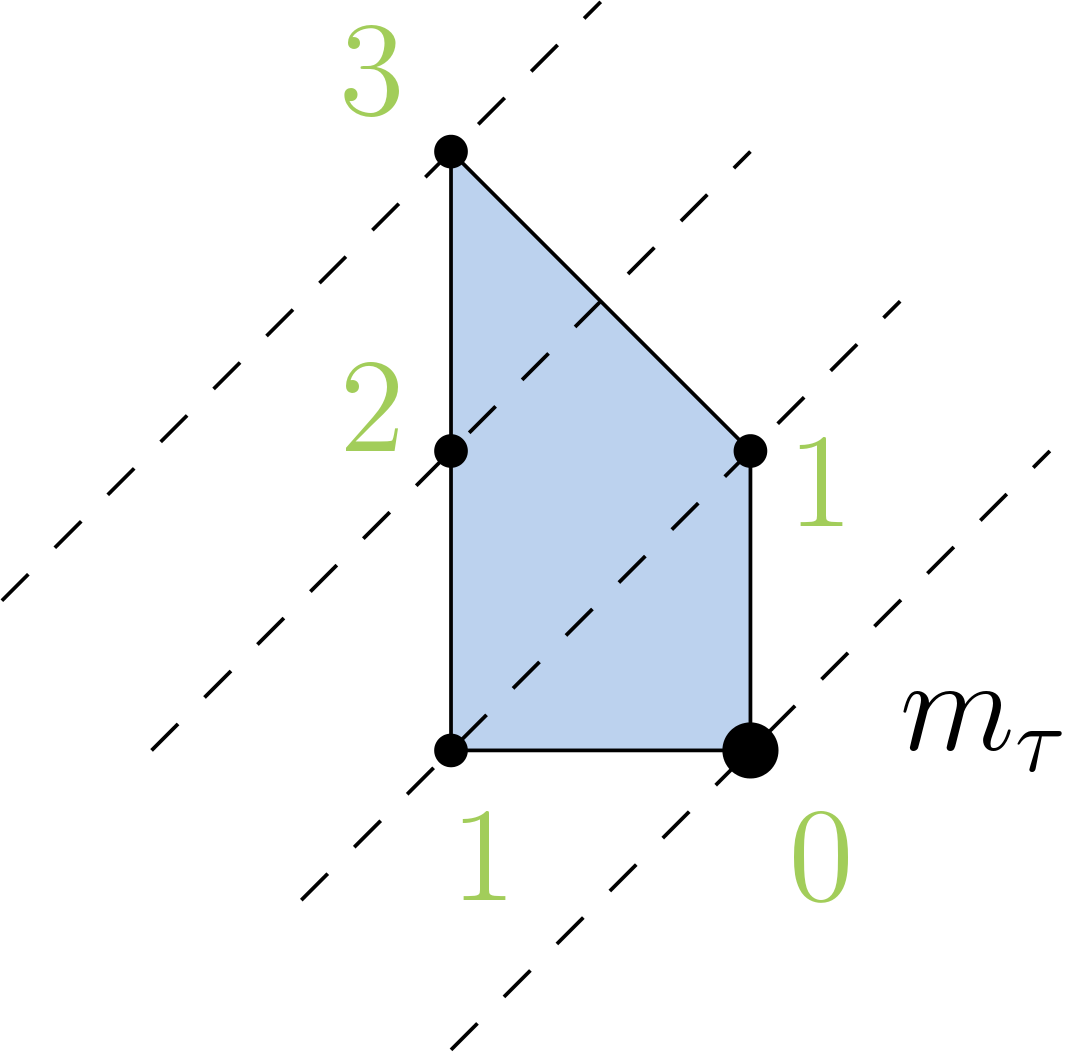}
\caption{Values of the Newton--Okounkov function $\fct_\sv$ associated to the distance to $m_\ldc$.} \label{fig_latdist}    
\end{center}
\end{figure}

\end{example}


\subsection{Interpretation of a Subgraph as a Newton--Okounkov Body}\label{sec:subgraph}

Let $\var$ be a smooth projective variety, $\flag$ an admissible flag
and $\ds$ a big $\Q$-Cartier divisor 
on $\var$. This determines the Newton--Okounkov body
$\nob{\flag}{\ds}$. Given a smooth subvariety $\sv \subseteq \var$, we
consider the function $\fct_{\sv}$ on $\nob{\flag}{\ds}$ that comes
from the geometric valuation $\ord_{\sv}$.

In \cite{KMR} K\"uronya, Maclean and Ro\'{e} construct a
variety $\hat{\var}$, a flag $\hat{Y}_\bullet$ and a
divisor $\hat\ds$
on $\hat{\var}$ so that the resulting Newton--Okounkov body is the
subgraph of $\fct_{\sv}$ over $\nob{\flag}{\ds}$. We translate their
construction into polyhedral language in the toric case.

According to Lemma 4.2 in \cite{KMR} we may assume that the geometric
valuation $\ord_{\sv}$ comes from a smooth effective Cartier divisor
$\dsl$ on $\var$, i.e. $\ord_{\sv}= \ord_{\dsl}$. This can always be
guaranteed by possibly blowing up $\var$
(compare~\S\ref{eg:torictoric:F1}).

Set 
\begin{equation*}
\hat{\var} \coloneqq \PP_{\var}(\she_{\var} \oplus \she_{\var}(\dsl)).  
\end{equation*}
In other words, we consider the total space of the line bundle
$\she_\var(\dsl)$ and compactify each fiber to a $\PP^1$. 
We denote by $\pi$ the projection $\hat\var
\overset{\pi}{\longrightarrow} \var$.
The zero-section of $\she_\var(\dsl)$ is a divisor $\var_0
\lhook\joinrel\xrightarrow{\iota_0} \hat\var$ which is isomorphic
to $\var$. The same is true for the $\infty$-section $\var_\infty
\lhook\joinrel\xrightarrow{\iota_\infty} \hat\var$.

In our toric situation, $\hat\var$ is again toric, and its fan is
described in Proposition 7.3.3 of \cite{cox2011toric} as follows.
The local equation of $\dsl$ as a Cartier divisor along a toric
patch $U_\sigma$ is a torus character which corresponds to a linear
function on $\sigma$. These linear functions glue to the support
function $\spf_\dsl \colon |\fan| \to \R$ of $\dsl$ (see
Definition~4.2.11 \& Theorem~4.2.12 in \cite{cox2011toric}).
Using $\spf_\dsl$, we define an upper and a lower cone in $\lat_\R\times\R$ for every $\con \in \fan$:
\begin{align*}
  \hat\con &\coloneqq \left\{ (u,h) \in \lat_\R\times\R
                  \suchthat u \in \con \ , \ h \ge \spf_\dsl(u) \right\} \\
  \check\con &\coloneqq \left\{ (u,h) \in \lat_\R\times\R
                  \suchthat u \in \con \ , \ h \le \spf_\dsl(u) \right\} \,.
\end{align*}
Together with their faces, these cones form a fan $\hat\fan$ which
determines our $\hat\var$.
The upper and lower cones of the origin $\origin \in \fan$, are rays
$\hat\origin$ and $\check\origin$ whose toric divisors in
$\hat\var$ are $\var_0$ and $\var_\infty$, respectively.
The projection $\hat\var \to \var$ is toric. It comes from the
projection $N \times \Z \to N$ which identifies both
$\sta(\hat\origin)$ and $\sta(\check\origin)$ with $\fan$.

\begin{example} \label{eg:Xhat}
Let $\var_\fan= \PP^1$ be the projective line. Its corresponding fan
$\fan$ in $\R$ is depicted in Figure~\ref{fig_fanP1}, where the
torus-invariant prime divisors $\ds_0$ and $\ds_1$ correspond to
$\con_0 = \R_{\ge 0}$ and $\con_1 = \R_{\le 0}$ with primitive
ray generators $u_{0}=1$ and $u_1=-1$, respectively.

\begin{figure}[htb]
\includegraphics[scale=0.12]{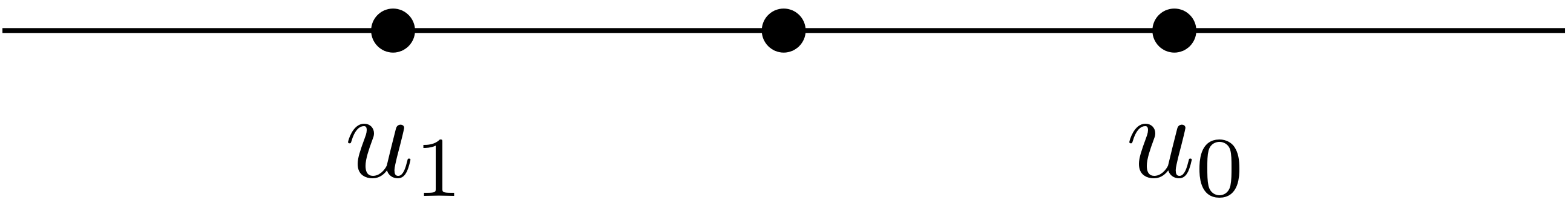}
\caption{The fan $\fan$ of the projective line $\var_\fan=
  \PP^1$} \label{fig_fanP1} 
\end{figure} 

Consider the divisor $\dsl = \ds_0$. Then $\hat\fan$ is a fan in $\R^2$
and its top-dimensional cones are
\begin{align*}
\hat{\con}_0 & = \cone((0,1),(1,-1)) \,, \\
\check{\con}_0 & = \cone((0,-1),(1,-1)) \,,\\
\hat{\con}_1 & = \cone((0,1),(-1,0)) \, \text{ and} \\
\check{\con}_1 & = \cone((0,-1),(-1,0)) \,.
\end{align*}
We obtain the fan of the Hirzebruch surface $\hir_1$ as depicted in Figure~\ref{fig_fanhat}.

\begin{figure}[htb]
\includegraphics[scale=0.15]{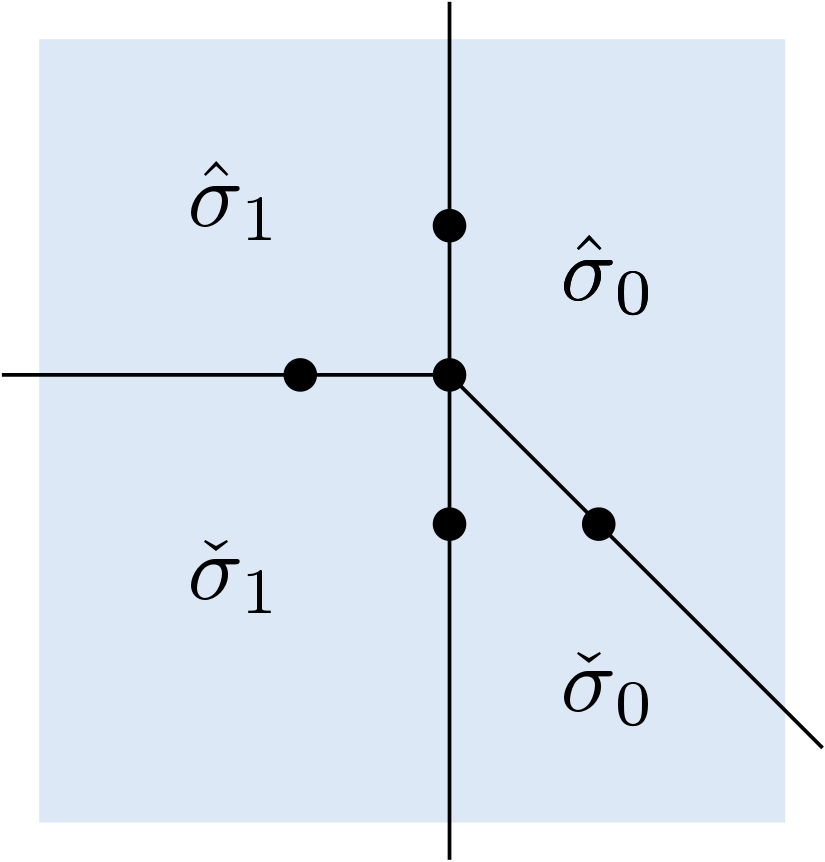}
\caption{The fan $\hat\fan$ of $\PP_{\PP^1}(\she_{\PP^1} \oplus
  \she_{\PP^1}(1))$}\label{fig_fanhat}
\end{figure}

\end{example}

For the suitable divisor $\hat\ds$
on $\hat{\var}$ we fix some rational number $\rat$ such that 
\begin{equation*}
\rat > \sup \{t > 0 \suchthat \ds -t\dsl \text{ is big }\}
\end{equation*}
and define
$\hat{\ds} \coloneqq \pi^*\ds + \rat \var_\infty$.
As an admissible flag
$\hat{\flag}$ we set
\begin{equation*}
\hat{Y}_1 \coloneqq \var_0, \ \hat{Y}_i \coloneqq \iota_0(Y_{i-1}) \text{ for all } i\geq 2.
\end{equation*} 
K\"uronya, Maclean and Ro\'{e} show that $\hat{\var}$, $\hat{Y}_\bullet$ and
$\hat{\ds}$ are the suitable objects to obtain the desired
identification
\begin{equation*}
  \left\{ (m,h) \in \nob{\flag}{\ds} \times \R \suchthat 0 \le h \le
    \fct_{\dsl}(m) \right\} =  
  \nob{\hat{Y}_\bullet}{\hat{\ds}}.
\end{equation*} 

\begin{example}\label{ex_fanhat}
  We continue with Example~\ref{eg:Xhat}. In addition to the data
  $\var = \PP^1$, $\dsl = D_0$ we choose the toric flag $Y_1 = 
  V(\con_0)$ and the big divisor $\ds = 2 D_0$.

  Then the flag $\hat{Y}_\bullet$ consists of $\hat Y_1 = V(\hat\origin)$
  and $\hat Y_2 = \iota_0(Y_1) = V(\hat\con_0)$.
  
  The support function for $\pi^*\ds$ is the pullback of the support
  function for $\ds$ along the linear projection $\hat\fan \to
  \fan$. Its values at the ray generators are indicated in
  Figure~\ref{fig_sp_hir}.
  
\begin{figure}[htb]
  \includegraphics[scale=0.1]{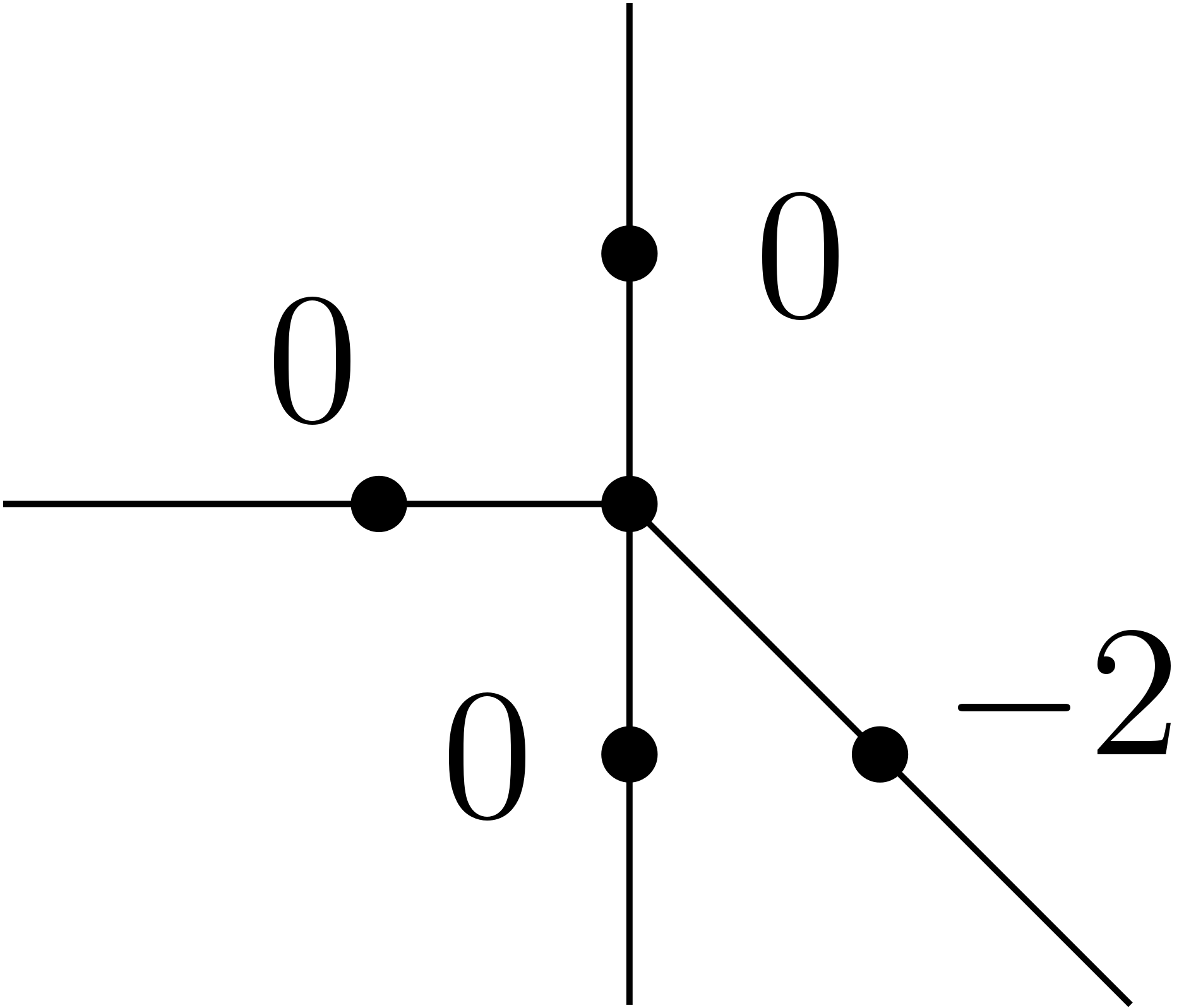}
  \caption{Values of the support function $\spf_{\pi^*\ds}$ for the ray generators $u_\ray$ for $\ray \in \hat\fan(1)$}\label{fig_sp_hir}
\end{figure}

  The resulting polyhedron in $\latm_\R\times\R$ is
  \begin{equation*}
    \left\{ (m,h) \in \latm_\R\times\R \suchthat m \in \nob{\flag}{\ds}
    \ , \ h \le 0 \ , \ h \ge 0 \right\}
  \end{equation*}
  Adding $t \var_\infty = t V(\check\origin)$ to $\pi^*\ds$
  for $t > 0$, relaxes the corresponding inequality $h \le 0$ to $h +
  t \le 0$. The effect on the polyhedron is depicted in
  Figure~\ref{fig_wedge}. 
\begin{figure}[htb]
  \includegraphics[scale=0.27]{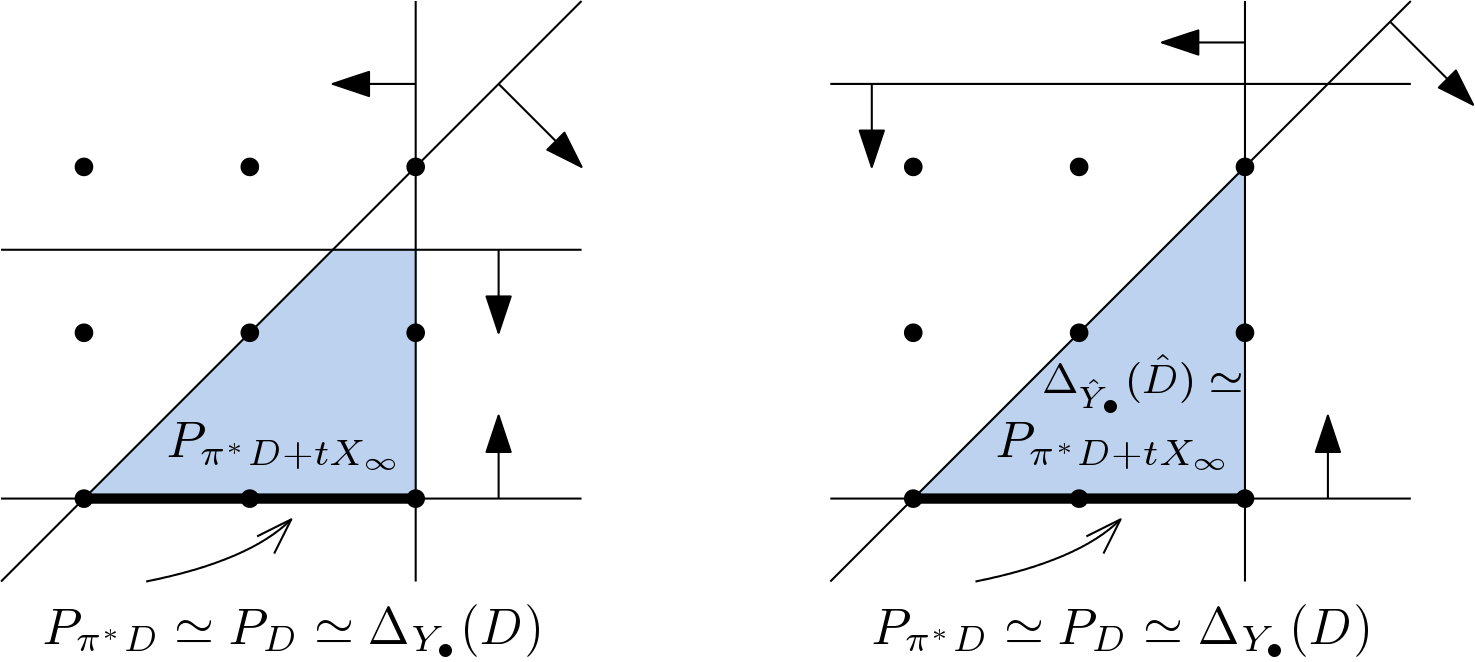}
  \caption{The polytope $\divp{\pi^*\ds +t X_\infty}$ for small $t$ on the
    left and for $t=b$ on the right}\label{fig_wedge}
\end{figure}

\end{example}

\subsection{Geometric Valuation Coming from a General Point}
\label{sec:general-point}
Let $\var$ be a smooth projective toric surface and $\ds$ an ample divisor on $\var$. In this section we relax the requirements in the sense that the function $\fct_R$ now comes from the geometric valuation $\ord_R$ at a general point $R$, not necessarily torus-invariant. Here we can determine the values of $\fct_R$ on parts of $\nob{\flag}{\ds}$ and give an upper bound on the entire Newton--Okounkov body.\\
\\
In order to do so, we need to introduce some more terminology. We are given an admissible torus-invariant flag $\flag \colon \var \supseteq Y_1 \supseteq Y_2$ on $\var$. Since $\flag$ is toric, the Newton--Okounkov body $\nob{\flag}{\ds} \subseteq \R^2$ is isomorphic to $\divp{\ds}$ and one of its facets corresponds to $Y_1$. Let $\linf \in \left(\R^2\right)^*$ denote the defining linear functional that selects this face $Y_1$, when minimized over the polytope $\nob{\flag}{\ds} $. We denote by $F \preceq \nob{\flag}{\ds} $ the face that is selected, when maximizing $\linf$ over $\nob{\flag}{\ds}$. Either this already is a vertex or if not, we maximize $\linf'$ over $F$, where $\linf' \in \left(\R^2\right)^* $ is a linear functional selecting $Y_2$, when minimized over $\nob{\flag}{\ds}$. Denote the resulting vertex in $\ver(\nob{\flag}{\ds})$ by $\vertex_{\flag}$. We say that the vertex $\vertex_{\flag}$ lies at the \emph{opposite side} of the polytope $\nob{\flag}{\ds}$ with respect to the flag $\flag$.

\begin{theorem}\label{prop_a+b}
Let $\var$ be a smooth projective toric surface, $\ds$ an ample
divisor, and $\flag$ an admissible torus-invariant flag on $\var$.
Denote by $\nob{\flag}{\ds}$ the corresponding
Newton--Okounkov body and by $\vertex=\vertex_{\flag}$ the vertex at the opposite
side of $\nob{\flag}{\ds}$ with respect to $\flag$. Moreover let $R\in
\tor$ be a general point. Then for the Newton--Okounkov function
$\fct_R$ coming from the geometric valuation $\ord_R$ we have
\begin{enumerate} 
\item \begin{equation*}
\fct_R(\a,\b) \leq \a+\b
\end{equation*}
for all $(\a,\b) \in \nob{\flag}{\ds} $, where $(\a,\b)$ are the coordinates in the coordinate system associated to $\vertex$. 
\item Furthermore, we have 
\begin{equation*}
\fct_R(\a,\b) = \a+\b
\end{equation*}
for all 
\begin{equation*}
(\a,\b) \in \{ (\a',\b') \in \nob{\flag}{\ds}  \suchthat \np((x-1)^{\a'}(y-1)^{\b'}) \subseteq \nob{\flag}{\ds} \}.
\end{equation*}

\end{enumerate}
\end{theorem}

\begin{proof}
\begin{enumerate}
\item Let $(\a,\b) \in\nob{\flag}{\ds}$ be a valuative point in the Newton--Okounkov body. We want to determine $\fct_\gen(\a,\b)$, where $\fct_R$ is the function coming from the geometric valuation $\ord_R$. Consider an arbitrary section $s \in \sps(\var,\she_{\var}(k\ds))$ that is mapped to $(\a,\b)=\frac{1}{k}\val_{\flag}(s)$ for some $k \in \N$. Let $u,u' \in \left(\R^2\right)^*$ be as above. Then, by construction, the rescaled exponent vectors of all monomials that can occur in $s$ have to be an element of the set
\begin{eqnarray*}
\hyp^+ \coloneqq \left\{ m  \in \nob{\flag}{\ds} \right. & \suchthat &\linf(m) > \linf(\a,\b) \text{ or }  \\
& & \left. ( \linf(m) = \linf(\a,\b) \text{ and } \linf'(m) \geq \linf'(\a,\b))\right\}.
\end{eqnarray*}

As indicated in Figure~\ref{figure_region}, this region is obtained by intersecting $\nob{\flag}{\ds}$ with the positive halfspace associated to the hyperplane $\hyp=\{m \suchthat \linf(m)=\linf(\a,\b)\}$. 

\begin{figure}[h!]
\begin{center}
    \includegraphics[scale=0.2]{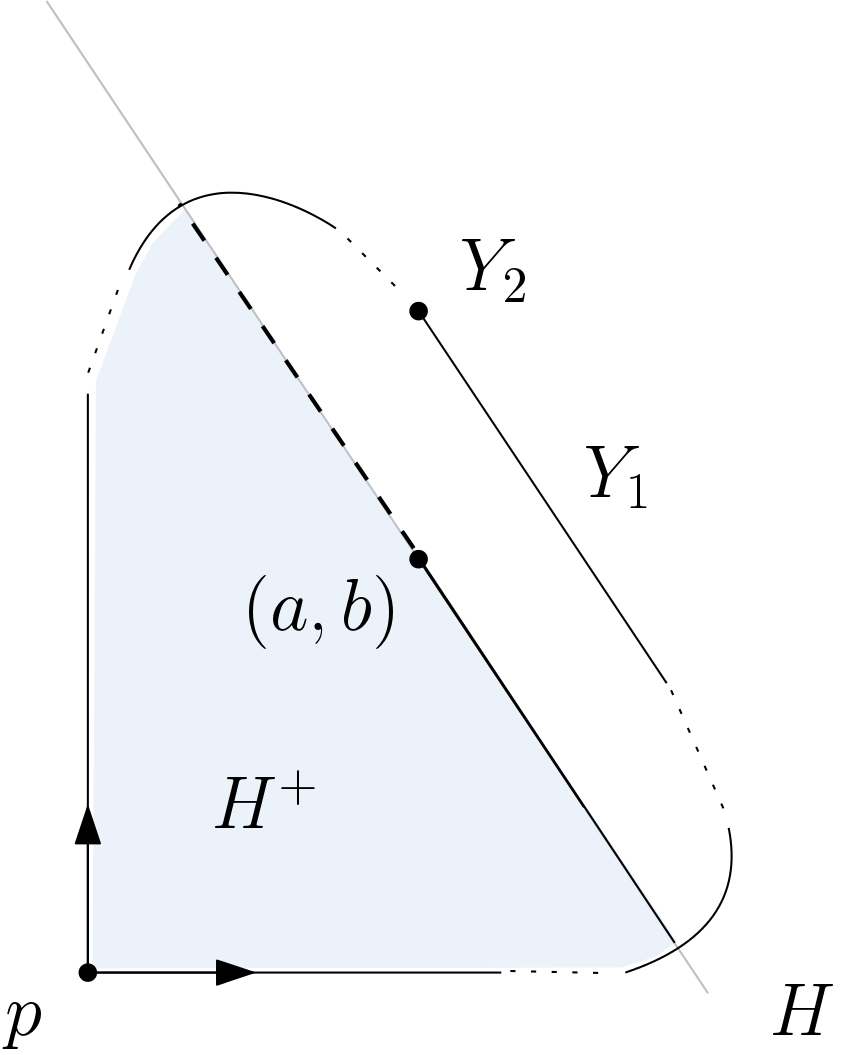}
\caption{Admissible region $\hyp^+$ of rescaled exponent vectors associated to monomials of $s$ inside the Newton--Okounkov body $\nob{\flag}{\ds}$.}\label{figure_region}
\end{center}
\end{figure}

Moreover, we can assume, without loss of generality that the general point $\gen$ is given as $\gen=(1,1)$. To determine the order of vanishing of $s$ at $R$ we substitute $x$ by $x'+1$ and $y$ by $y'+1$ and bound the order of vanishing of ${s'(x',y')=s(x'+1,y'+1)}$ at $(0,0)$. Assuming, without loss of generality that the monomial $x^{k\a} y^{k\b}$ itself occurs in $s$ with coefficient $1$, multiplying out gives
\begin{eqnarray*}
s'(x',y')&=& s(x'+1,y'+1)=(x'+1)^{k\a}(y'+1)^{k\b}+ \ *** \\
			&=& (x')^{k\a}(y')^{k\b}+ \text{ lower order terms } + \ ***.
\end{eqnarray*}
\textsc{Claim}: The monomial $(x')^{k\a}(y')^{k\b}$ cannot be canceled out by terms coming from $***$.

Aiming at a contradiction, assume that \mbox{$***$} contains a monomial ${(x'+1)^{kc}(y'+1)^{kd}}$ for some ${kc},{kd} \in \N$ that produces $(x')^{k\a}(y')^{k\b}$ when multiplied out. Observe that  multiplying out ${(x'+1)^{kc}(y'+1)^{kd}}$ produces all monomials in $${\left\{ (x')^e(y')^f \suchthat e \leq {kc} \text{ and } f \leq {kd} \right\}}.$$ Thus ${kc} \geq {k\a}$ and ${kd} \geq {k\b}$. In addition, as an exponent vector of a monomial in $s$, the point $(c,d)$ is required to be an element of the set $\hyp^+$, which  forces the hyperplane $\hyp$ to have positive slope as indicated in Figure~\ref{fig_slope}. 

\begin{figure}[h!]
\begin{center}
\includegraphics[scale=0.21]{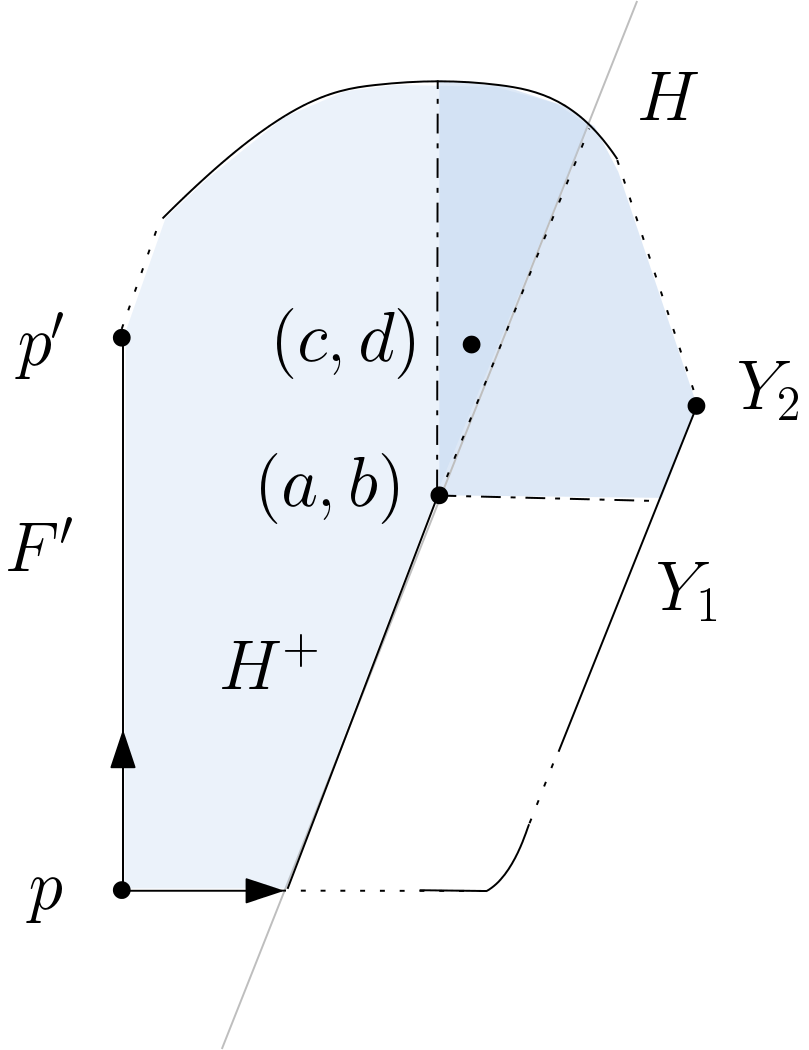}
\caption{A non-empty region of points $(c,d)\in \hyp^+$ that satisfy $c\geq a$ and $d \geq b$ forcing $\hyp$ to have positive slope.}\label{fig_slope}    
\end{center}
\end{figure} 

Let $F' \preceq \nob{\flag}{\ds}$ denote the face that corresponds to $\{x=0\}$ and let $\vertex'$ denote its second vertex. Then $\linf(\vertex')  > \linf( \vertex )$ which contradicts the fact that $\linf$ is maximized at $\vertex$ over $\nob{\flag}{\ds}$. Thus such a monomial $(x'+1)^{kc}(y'+1)^{kd}$ cannot exist and $(x')^{k\a}(y')^{k\b}$ does not cancel out.  \\  
\\
Consequently, $k(\a+\b)$ is an upper bound for the order of vanishing of $s'$ at $(0,0)$ and thus for $s$ at $R$. Since this is true for all sections $s$ that get mapped to $(\a,\b)$, this yields $\fct_R(\a,\b) \leq \a+\b$.\\

\item Consider a point 
\begin{equation*}
(\a,\b) \in \text{Par} \coloneqq \{ (\a',\b') \in \nob{\flag}{\ds} \suchthat \np((x-1)^{\a'}(y-1)^{\b'}) \subseteq \nob{\flag}{\ds} \}\ ,
 \end{equation*}
and set $s(x,y)= (x-1)^{k\a}(y-1)^{k\b}$, for a $k \in \N$ such that $s$ is a global section of $k\ds$, whose Newton polytope can be seen in Figure~\ref{fig_lexmax}. Then by construction, $s$ is a section associated to the point $(\a,\b)$ and its Newton polytope fits inside $k\nob{\flag}{\ds}$. We have $\ord_\gen(s)=k(\a+\b)$ which gives the lower bound $\fct_\gen(a,b) \geq \frac{1}{k}\ord_\gen(s)$. Combined with 1. we obtain $\fct_R(\a,\b)=\a+\b$.  

\begin{figure}[h!]
\begin{center}
\includegraphics[scale=0.17]{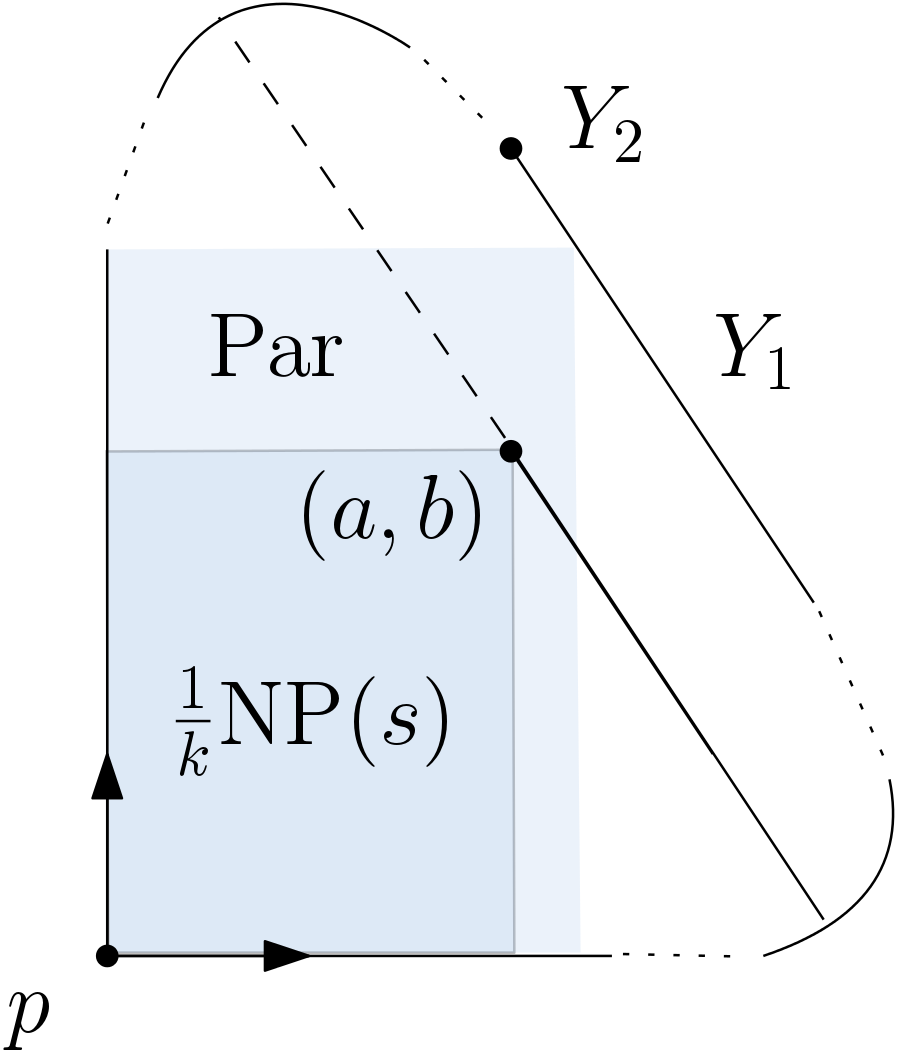}
\caption{The scaled Newton polytope $\frac{1}{k}\np(s)$ of the section $s(x,y) =(x-1)^{k\a}(y-1)^{k\b}$.}\label{fig_lexmax}    
\end{center}
\end{figure}   

\end{enumerate}
\end{proof}

We illustrate the use of Theorem~\ref{prop_a+b} by  the following example. 

\begin{example}\label{ex_par}

We continue our running example of the Hirzebruch surface $\var=\hir_1$ and the ample divisor $\ds= \ds_3 +2\ds_4$ as in Example~\ref{ex_hir1}. Furthermore, fix the torus-invariant flag $\flag \colon \var \supseteq Y_1 \supseteq Y_2$, where $Y_1=\ds_1$ and $Y_2=\ds_1 \cap \ds_2$. Then the vertex $\vertex=\vertex_{\flag}$ of the Newton--Okounkov body $\nob{\flag}{\ds} \cong \divp{\ds}$ that lies at the opposite side of the polytope $\divp{\ds}$ with respect to the flag $\flag$ is the  one indicated in Figure~\ref{fig_lexmax_hir1}. The associated coordinate system specifies coordinates $a,b$ for the plane $\R^2$ and local toric coordinates $x,y$. 

\begin{figure}[h!]
\begin{center}
\includegraphics[scale=0.1]{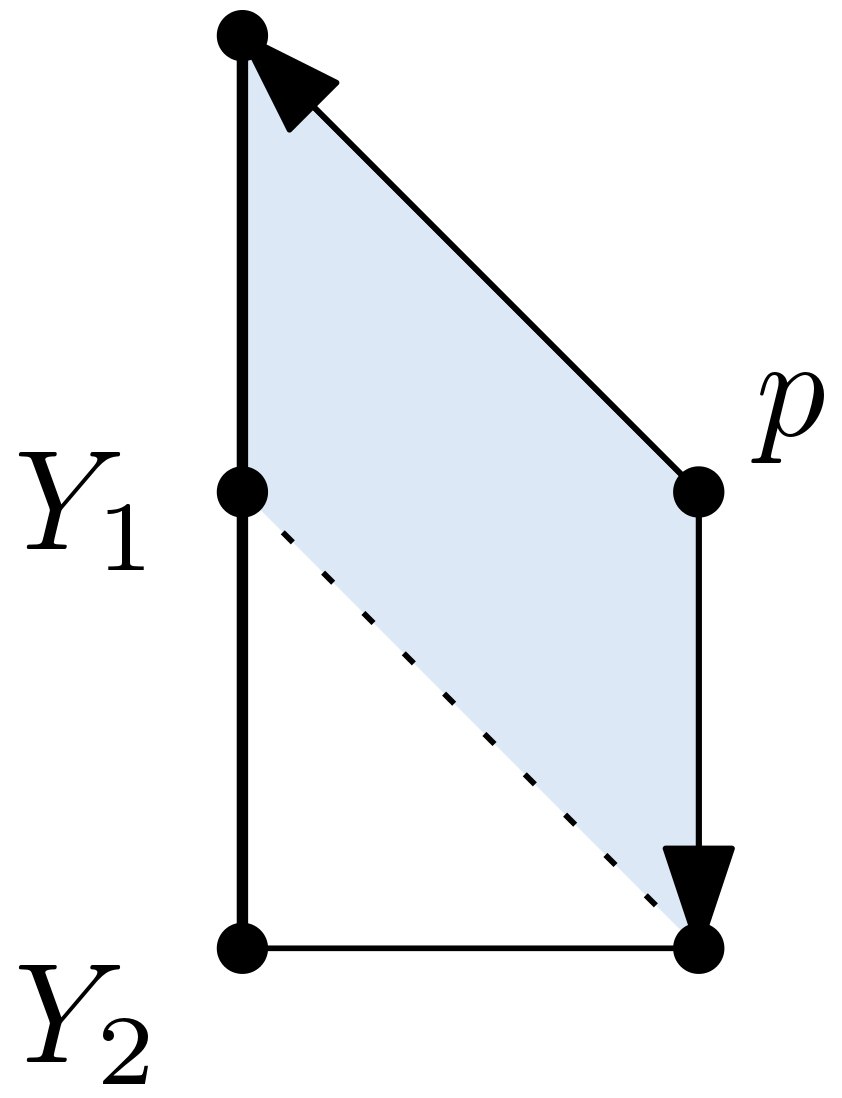}
\caption{The coordinate system associated to the vertex $\vertex$ which lies at the opposite side of $\divp{\ds}$ with respect to the flag $\flag$.}\label{fig_lexmax_hir1}    
\end{center}
\end{figure}   

We want to determine the Newton--Okounkov function coming from the
geometric valuation at the point $\gen=(1,1)$. Theorem~\ref{prop_a+b}
yields the upper bound ${\fct_\gen(\a,\b) \leq \a+\b}$ on the entire
Newton--Okounkov body $\nob{\flag}{\ds}$, and ${\fct_\gen(\a,\b)= \a+\b}$
for $(\a,\b)$ satisfying ${\a,\b \leq 1}$, as indicated by the shaded
region in Figure~\ref{fig_lexmax_hir1}. It will turn out in
Example~\ref{ex_1} that there exist points ${(\a,\b)\in
\nob{\flag}{\ds}}$ for which we have ${\fct_R(\a,\b)<\a+\b}$.
\end{example}

For a particularly nice class of polygons, Theorem~\ref{prop_a+b}
alone is enough to determine the function $\fct_R$. A polytope
$P \subseteq \R^n_{\ge0}$ is called anti-blocking if ${P = (P + \R^n_{\le0}) \cap \R^n_{\ge0}}$
(compare~\cite{FulkersonBlockingAntiBlocking,FulkersonAntiBlocking}).
Observe that this coordinate dependent property implies (and for $n=2$
is equivalent to) the fact that the parallelepiped spanned by the
edges at the origin covers $P$.

\begin{corollary}\label{cor_par}
Let $\var, \flag, \ds$ and $R$ be as in
Theorem~\ref{prop_a+b}. Suppose $\nob{\flag}{\ds}$ is anti-blocking.
Let $\flag'$ be a torus-invariant flag opposite to the origin. Then
the Newton--Okounkov function $\fct_R$ on $\nob{\flag'}{\ds}$
is given by
\begin{equation*}
\fct_R(\a,\b) = \a+\b
\end{equation*}
on the entire Newton--Okounkov body $\nob{\flag'}{\ds} \cong \divp{\ds}$ in
the coordinate system associated to $\flag$.
\end{corollary}

Using the tools from Section~\ref{sec:seshadri},
Corollary~\ref{cor_par} implies that the Seshadri constant of $\ds$ at
$R$ is rational. This can also be seen from Sano's
Theorem~\cite{sano2014seshadri} as $\dim |-K_X| \ge 3$ in the
anti-blocking case.\\
\\
If we are not in the lucky situation of Corollary~\ref{cor_par}, then
things are getting more complicated and more interesting. We give an
approach that works in numerous cases.

\fbox{\textsc{The general strategy}}\\

For the remainder of the paper we will consider the following situation.\\
\\
\begin{notation}\label{notation}
$ $\\
\\
\begin{tabular}{ l p{.8\linewidth} }
$\var$ & a smooth projective toric surface, \\
$\ds$ & an ample torus-invariant divisor on $\var$, \\
$\flag$ & an admissible torus-invariant flag, \\
$\dir$ & the primitive direction of the edge of $\divp{\ds} \cong \nob{\flag}{\ds}$
         corresponding to $Y_1$, towards the vertex corresponding to
         $Y_2$, \\
$\linf$ & the primitive ray generator corresponding to $Y_1$, \\
$\cur$ & the curve in $\var$ given by the binomial $x^\dir-1$, \\
$R$ & a general point on $\cur$, \\
$\flag'$ & the admissible flag $\var \supseteq \cur \supseteq
           \{\gen\}$, \\
$\pwl$ & \mbox{the  piecewise linear,
volume-preserving isomorphism} ${\nob{\flag}{\ds} \to \nob{\flag'}{\ds}}$
         from Corollary~\ref{cor_mv}, \\
$\fct_\gen$ & the function $\nob{\flag}{\ds} \to \R$ coming from the
              geometric valuation $\ord_R$, \\
$\fct'_\gen$ & the function $\nob{\flag'}{\ds} \to \R$ coming from the
              geometric valuation $\ord_R$.
\end{tabular}

\end{notation}
$ $ \\
\\
\textsc{Goal:}
Determine the function 
\begin{equation*}
\fct_R \colon \nob{\flag}{\ds} \cong \divp{\ds} \to \R.
\end{equation*}

\textsc{Approach:}

\begin{enumerate}
\item\label{enum_step1} For each valuative point $(\a,\b) \in
  \nob{\flag}{\ds}$ `guess' a Newton polytope $\frac{1}{k}\np(s)
  \subseteq \divp{\ds} $ of a global section $s \in
  \sps(\var,\she_\var(k\ds))$ for some $k \in \N$ to maximize the
  order of vanishing $\ord_R(s)$ according to the following rules:
  \begin{itemize}
  \item The section $s$ has to correspond to the point $(\a,\b)$.
  \item Choose a Newton polytope $\np(s)$ that is a zonotope whose
    edge directions  all come from edges in $\divp{\ds}$. 
  \item Try to maximize the perimeter of the Newton polytope $\frac{1}{k}\np(s)$
    among the above.
  \end{itemize}
\item Determine the values of the function $\fct \colon
  \nob{\flag}{\ds} \to \R$ that takes $\frac{1}{k}\ord_R(s)$ as a
  value with respect to the chosen sections $s$ for a point $(\a,\b)
  \in \nob{\flag}{\ds}$ and compute the integral
  $\int_{\nob{\flag}{\ds}}{\fct}$ .
\item Compute the Newton--Okounkov body $\nob{\flag'}{\ds}$ with
  respect to the new flag $\flag'$ using  variation of Zariski
  decomposition or the combinatorial methods from
  Section~\ref{sec_nob_toric}.
\item Compute the integral $\int_{\nob{\flag'}{\ds}}{\fct'}$, where we
  assume the function to be given by
  \begin{eqnarray*}
    \fct' \colon {\nob{\flag'}{\ds}} &\to& \R \\
    (\a',\b') & \mapsto& \a'+\b'.
  \end{eqnarray*}	 
\item Compare the value of the integrals
  $\int_{\nob{\flag}{\ds}}{\fct}$ and
  $\int_{\nob{\flag'}{\ds}}{\fct'}$. It holds that
  \begin{equation}\label{eqn_ints}
    \int_{\nob{\flag}{\ds}}{\fct} \leq
    \int_{\nob{\flag}{\ds}}{\fct_R}\quad =
    \int_{\nob{\flag'}{\ds}}{\fct'_R}
    \leq\int_{\nob{\flag'}{\ds}}{\fct'}.
  \end{equation}
  \begin{itemize}
  \item If
    $\int_{\nob{\flag}{\ds}}{\fct}=\int_{\nob{\flag'}{\ds}}{\fct'}$,
    then we have equality in (\ref{eqn_ints}) and therefore a
    certificate that the choices that we have made were valid and we are
    done.
  \item If
    $\int_{\nob{\flag}{\ds}}{\fct}<\int_{\nob{\flag'}{\ds}}{\fct'}$,
    then either we have chosen sections with non-maximal orders of
    vanishing at $\gen$ in step~\ref{enum_step1} or for the chosen
    vector $v$ 
    the function $\fct'_\gen$ takes values smaller than $\a'+\b'$ somewhere on $\nob{\flag'}{\ds}$.
\end{itemize}	 
\end{enumerate}

\begin{example}\label{ex_1}

We return to Example~\ref{ex_par}, and again consider the Hirzebruch surface  \linebreak $\var=\hir_1$ equipped with the torus-invariant flag $\flag \colon \var \supseteq Y_1 \supseteq Y_2$, where $Y_1=\ds_1$ and $Y_2=\ds_1 \cap \ds_2$ and the ample divisor $\ds=\ds_3+2\ds_4$ on $\var$.
We want to determine the values of a function on $\nob{\flag}{\ds}$ coming from a geometric valuation at a general point, so let $R=(1,1) \in \var$ in local coordinates. More precisely, for the rational points in $\nob{\flag}{\ds}$ in the coordinate system associated to the flag $\flag$ we study  
\begin{eqnarray*}
 \fct_R \colon \nob{\flag}{\ds} &\to & \R \\
  (\a,\b) &\mapsto & \lim_{k \to \infty} \frac{1}{k} \sup \left\{ t \in \R \suchthat  \text{ there exists } s \in \sps(\var,\she_\var(k\ds)) \suchthat\right. \\
  & & \left.\val_{\flag}(s)=k(\a,\b), \ord_R(s) \geq t  \right\}.
\end{eqnarray*}
We claim that $\fct_R$ coincides with the function $\fct$ given by 
\begin{equation*}
\fct(\a,\b)= \begin{cases}
2-\a & \text{ if } 0 \leq \a+\b \leq 1 \\
3-2a-\b& \text{ if } 1 \leq \a+\b \leq 2
\end{cases}
\end{equation*}
at a point $(\a,\b) \in \nob{\flag}{\ds} $.

To verify this claim, we will give explicit respective sections and argue that the maximal value of $\ord_R$ is achieved for these particular sections. We treat the two cases individually. 

\begin{figure}[h!]
\begin{center}
     \includegraphics[scale=0.16]{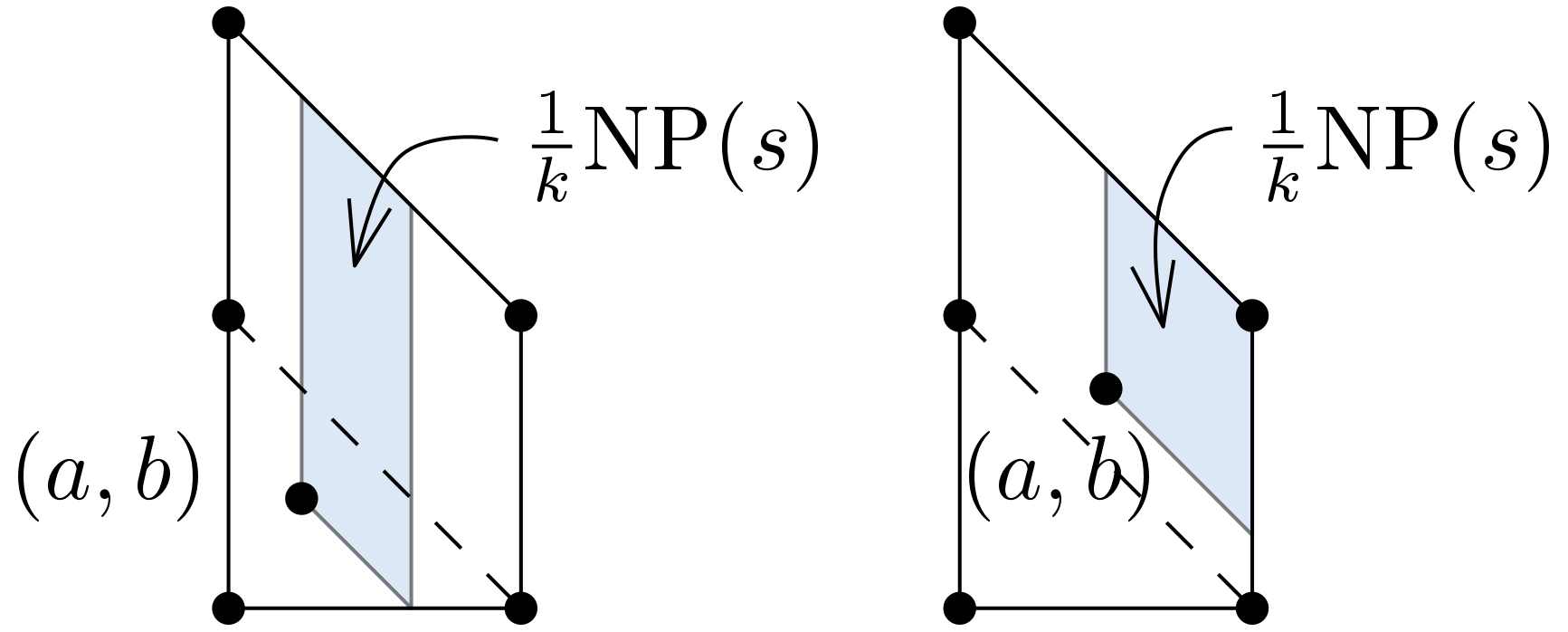}
\caption{Newton polytopes $\frac{1}{k}\new(s)$ of the respective sections $s$.}\label{newtonpoly}
\end{center}
\end{figure}

\pagebreak
\fbox{$0 \leq \a+\b \leq 1 $:}\\
\\
Set 
\begin{equation*}
s(x,y)=(x^\a(y-1)^{2-\a-\b}(x-y)^{\b})^k
\end{equation*}
in local coordinates $x,y$ for suitable $k \in \N$. The corresponding Newton polytope $\frac{1}{k}\new(s)$ is depicted in Figure~\ref{newtonpoly}. Since the leftmost part of it has coordinates $(\a,\cdot)$, we have $\ord_{Y_1}(s)=k\a$. If we restrict to the line $(\a,\cdot)$, then the lowest point of the Newton polytope is $(\a,\b)$ and thus $\ord_{Y_2}(s_1)=k\b$. Together with the fact that the Newton polytope $\frac{1}{k}\np(s)$ fits inside the Newton--Okounkov body $\nob{\flag}{\ds}$, this guarantees that the section $s$ is actually mapped to the point $(\a,\b)$ when computing $\nob{\flag}{\ds}$. \\
\\
For the order of vanishing of interest we obtain
\begin{equation*}
\ord_R(s)= k((2-\a-\b)+\b)=k(2-\a)\ .
\end{equation*}
\\
\fbox{$1 \leq \a+\b \leq 2$:}\\
\\
Set 
\begin{equation*}
s(x,y)=(x^ay^{\a+\b-1}(y-1)^{2-\a-\b}(x-y)^{1-\a})^k.
\end{equation*}
That all the requirements are fulfilled by $s$ follows by using the same arguments as in the previous case. For the order of vanishing of interest we obtain
\begin{equation*}
\ord_R(s)= k((2-\a-\b)+(1-\a))=k(3-2\a-\b)\ .
\end{equation*}

The values of the resulting piecewise linear function are depicted in Figure~\ref{fig_valueslam}.

\begin{figure}[h!]
\begin{center}
\includegraphics[scale=0.11]{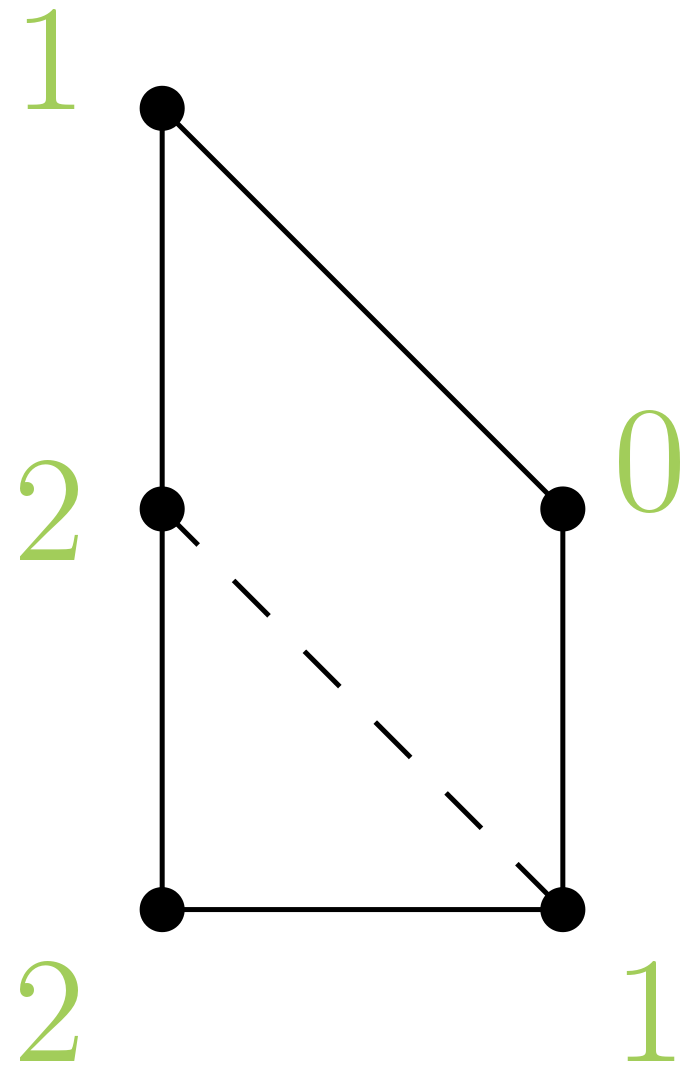}
\caption{The values of the function $\fct$ on $\nob{\flag}{\ds}$.}\label{fig_valueslam}    
\end{center}
\end{figure}

If we integrate $\fct$ over $\nob{\flag}{\ds}$, we obtain  

\begin{equation*}
\int_{\nob{\flag}{\ds}}{\fct}=\frac{11}{6}.
\end{equation*}

Now it remains to show that these values are actually the maximal ones that can be realized. In order to do this, we make use of the fact that the integral of our function $\fct$ over the Newton--Okounkov body $\nob{\flag}{\ds}$ is independent of the flag $\flag$.  \\
\\
We keep the underlying variety $\var$ and the ample divisor $\ds$. Choose a new admissible flag ${\flag' \colon \var \supseteq Y_1' \supseteq Y_2'}$, where $Y_1'$ is the curve defined by the local equation ${y^{-1}-1=0}$ and ${Y_2'=R=(1,1)}$ is the point of the geometric valuation. Since this flag is no longer torus-invariant, the corresponding Newton--Okounkov body $\nob{\flag'}{\ds}$ will differ from the polytope $\divp{\ds}$. As shown in Example~\ref{ex_hir1}, we obtain the new Newton--Okounkov body $\nob{\flag'}{\ds}$ depicted in Figure~\ref{fig_nob_hir1}.  \\
\\
For the function $\fct'_R$ on $ \nob{\flag'}{\ds}$, we are still working with the geometric valuation associated to $\ord_R$. 
Thus, set
\begin{eqnarray*}
\fct' \colon \nob{\flag'}{\ds} &\to& \R \\
(\a',\b') &\mapsto&  \a'+\b'.
\end{eqnarray*} 

\begin{figure}[h!]
\begin{center}
    \includegraphics[scale=0.07]{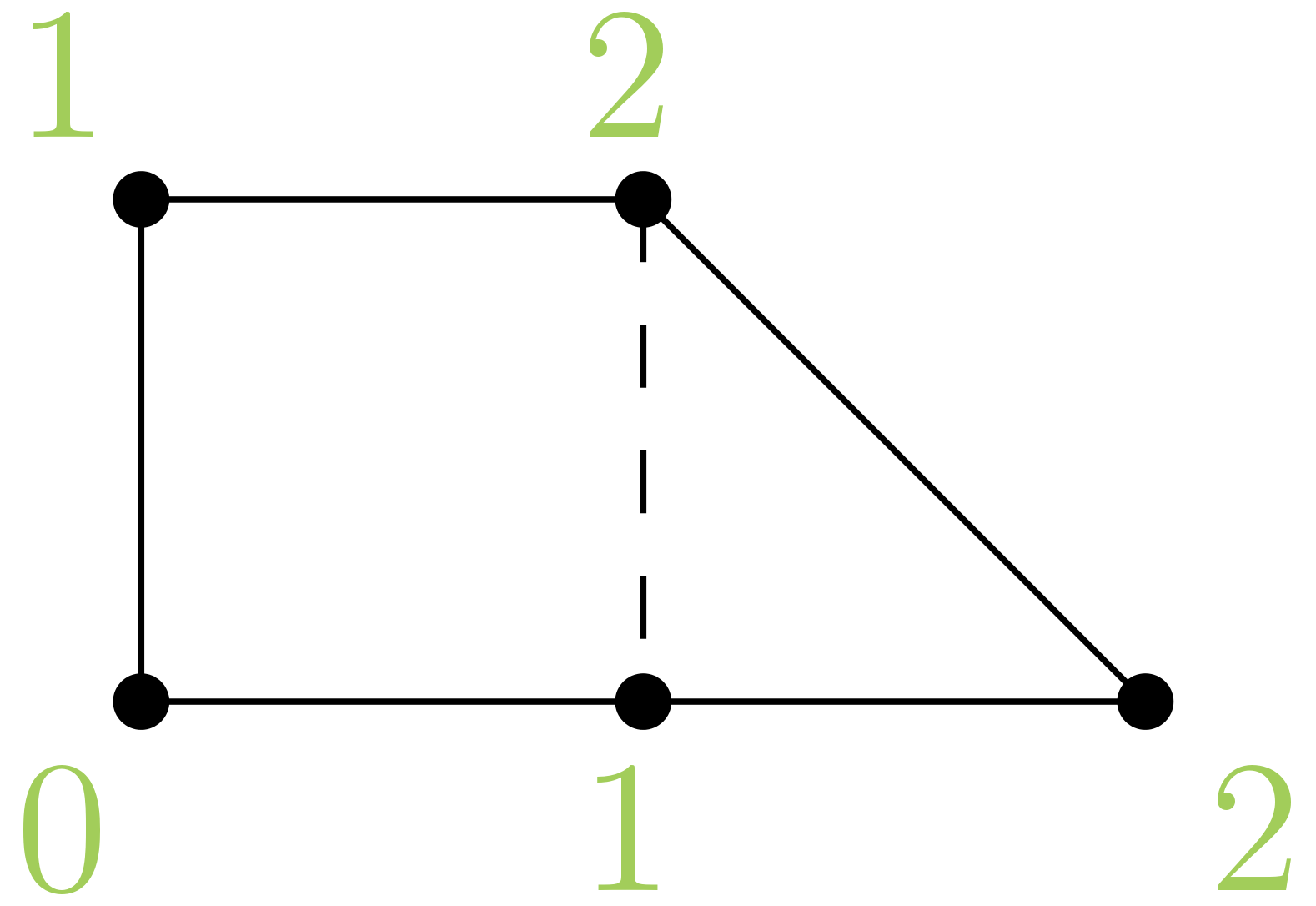}
\caption{The values of the function $\fct'$ on $\nob{\flag'}{\ds}$.}\label{fig_valueslam2}
\end{center}
\end{figure}

The values of $\fct'$ are depicted in Figure~\ref{fig_valueslam2}. If we integrate $\fct'$ over $\nob{\flag'}{\ds}$, we obtain      

\begin{equation*}
\int_{\nob{\flag'}{\ds}}{\fct'}=\frac{11}{6}.
\end{equation*}

Overall, we have $\int_{\nob{\flag}{\ds}}{\fct }= \int_{\nob{\flag'}{\ds}}{\fct'}$. This shows that our choice for the section $s$ was indeed maximal with respect to $\ord_R(s)$ and thus determines the value of $\fct_R=\fct$. 
\end{example}

\begin{remark}\label{rem:equidecomposable}
In the previous example the integrals $\int_{\nob{\flag}{\ds}}{\fct}$ and $\int_{\nob{\flag'}{\ds}}{ \fct'}$ coincide. Observe that even more is true. Let 
\begin{equation*}
G(\fct)=\{(\a,\b,\fct(\a,\b)) \suchthat (\a,\b)\in \nob{\flag}{\ds}\}
\end{equation*}
denote the graph of $\fct$. Since $\fct$ is a concave and piecewise linear function, the set 
\begin{equation*}
\nob{\flag}{\ds}_{\fct} \coloneqq \conv \left( ( \nob{\flag}{\ds}\times \{0\}) \cup G(\fct) \right) \subseteq \R^3 
\end{equation*}
is a $3$-dimensional polytope. If we compare $\nob{\flag}{\ds}_{\fct}$ and $\nob{\flag'}{\ds}_{\fct'}$, it turns out that they are $\text{SL}_3(\Q)$-equidecomposable, where the respective maps are volume-preserving.  

To see this, we give the explicit maps, where $\psi_1$ and $\psi_2$ come from the piecewise linear pieces of $\Psi$ on the respective domains of linearity. Use 
\begin{align*}
\psi_1\colon \R^3& \to \R^3 \\
\begin{pmatrix}
\a \\ \b\\c\\ 
\end{pmatrix} & \mapsto
\begin{pmatrix}
-1 & -1 & 0 \\
-1 & 0 & 0 \\
0 & 0 & 1\\
\end{pmatrix}
\cdot \begin{pmatrix}
\a \\ \b \\ c\\
\end{pmatrix}
+ \begin{pmatrix}
2\\1\\0\\
\end{pmatrix} \\
\end{align*}
to map the parallelogram in $\nob{\flag}{\ds}$ with its corresponding heights, and use 

\begin{align*}
\psi_2\colon \R^3& \to \R^3 \\
\begin{pmatrix}
\a \\ \b\\c\\ 
\end{pmatrix} & \mapsto
\begin{pmatrix}
-1 & -1 & 0 \\
0 & 1 & 0 \\
 0&0&1 \\
\end{pmatrix}
\cdot \begin{pmatrix}
\a \\ \b \\ c\\
\end{pmatrix}
+ \begin{pmatrix}
2 \\0\\0\\
\end{pmatrix} \\
\end{align*}
for mapping the triangle in $\nob{\flag }{\ds}$ with its corresponding heights. This is illustrated in Figure~\ref{equi}, where the respective heights are given in green. 

\begin{figure}[h!]
\begin{center}
\includegraphics[scale=0.28]{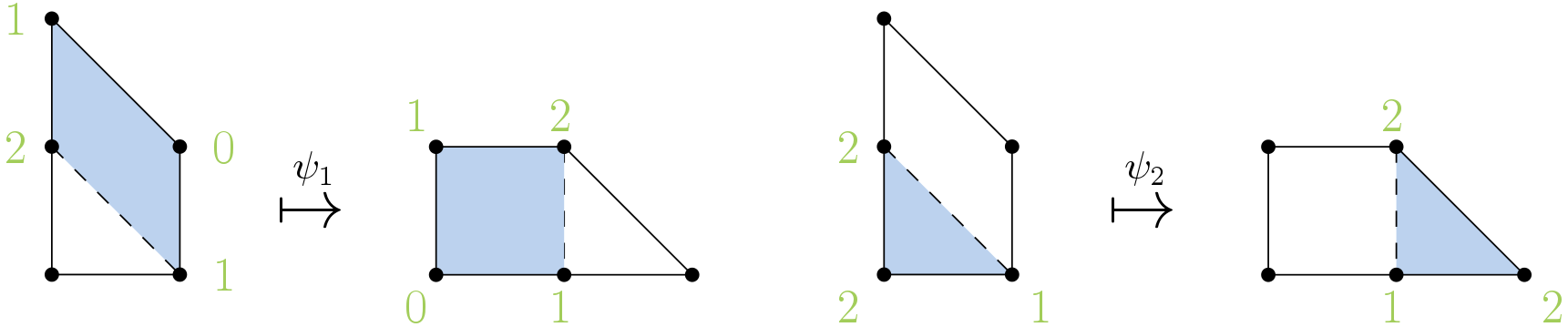}
\caption{$\text{SL}_3(\Q)$-equidecomposable pieces of the Newton--Okounkov bodies $\nob{\flag}{\ds}$ on the left and $\nob{\flag'}{\ds}$ on the right.}\label{equi}    
\end{center}
\end{figure}

\end{remark}
      
We conjecture that this in not a coincidence but holds in our general set-up.
       
\begin{conjecture}\label{conj:pl-transform}
In the situation of our general set-up~\ref{notation},
  $\fct_\gen = \fct'_\gen \circ \pwl$.
\end{conjecture}

The approach for determining the Newton--Okounkov function via $\pwl$
applies to a certain class of polytopes. To describe this class we
need to introduce the following term.

\begin{definition}\label{def_zonotopally}
Let $\pol \subseteq \R^2$ be a polygon.
Let $\dir \in \Z^2$ be a primitive vector, and $\linf \in
\dir^\perp$ a primitive integral functional.

We call $\pol$ \emph{zonotopally well-covered with respect to
  $(\dir,\linf)$} if for all points $\op \in \pol$ the set
$\feasible{\pol}{\dir}{\op}$ contains a zonotope
$\ls_1+\ldots+\ls_\ell$ with none of the $\ls_i$ parallel to $\dir$,
such that 
\begin{equation*}
\sum_{i=1}^{\ell}{\llen_{\Z^2}(\ls_i)} = \wid{\linf}(\feasible{\pol}{\dir}{\op}) \,.
\end{equation*}
The polygon $\pol$ is \emph{zonotopally well-covered} if it is so with
respect to some $(\dir,\linf)$. 
\end{definition}

In fact, it is enough to check the condition for the finitely many
vertices of domains of linearity of $\pwl$.

\begin{theorem}\label{thm_fct_zonotopally}
In the situation of our general set-up~\ref{notation},
if the polytope $\nob{\flag}{\ds}$ is zonotopally well-covered with
respect to $(\dir,\linf)$, then
${\fct_\gen = \fct'_\gen \circ \Psi}$
and ${\fct'_\gen(a',b') = a'+b'}$.
\end{theorem}

In particular, Conjecture~\ref{conj:pl-transform} holds in this case.

\begin{proof}
According to our general strategy, it is sufficient to certify, for
every valuative point $\op \in \nob{\flag}{\ds} \cap \Q^2$, the
existence of a section $s \in \sps(\var,\she_\var(k\ds))$ for some ${k
\in \N}$ with ${\frac1k\val_{\flag}(s)=\op}$ and with order of vanishing
${\ord_R(s) = k(\a'+\b')}$ where ${\pwl(\op) \eqqcolon (\a',\b')}$.

To this end, let $\ls_1+\ldots+\ls_\ell$ be the zonotope inside 
$\feasible{\pol}{\dir}{\op}$ which must exist because
$\nob{\flag}{\ds}$ is zonotopally well-covered. Add the segment
$\ls_0$ from $\origin$ to ${\relw{\pol,\op}{\dir} \cdot \dir}$ to obtain
a rational zonotope
\begin{equation*}
\ls_0+ \ls_1+\ldots+\ls_\ell \quad
\subseteq \quad \ls_0 + \feasible{\pol}{\dir}{\op} \quad
\subseteq \quad \nob{\flag}{\ds}
\end{equation*}
inside $\nob{\flag}{\ds}$ with valuation vertex $\op$. If $k$ is a
common denominator of its vertices, the $k$-th dilate is the Newton
polytope of a product of binomials which vanishes to order 
\begin{equation*}
  k \cdot \sum_{i=0}^{\ell}{\llen_\latm(\ls_i)}
  \ = \
  k \cdot \left(
    \relw{\pol,\op}{\dir} + \wid{\linf}(\feasible{\pol}{\dir}{\op})
  \right)
  \ = \ 
  k \cdot (\a' + \b')
\end{equation*}
as required.
\end{proof}

\begin{remark}
The property of being centrally-symmetric is not sufficient for being zonotopally well-covered. Consider for instance the polytope 
\begin{equation*}
\pol= \conv((0,0),(2,1),(1,3),(-1,2)) \subseteq \R^2
\end{equation*}
in Figure~\ref{fig_bad_square} and the direction $\dir=(-1,0)$ with $u=(0,1)$. Then for the point ${\op=(-\frac{1}{2},1)}$ the polygon has length $\relw{\pol,\op}{\dir}=\frac{5}{2}$ at $\op$ with respect to $\dir$. The intersection \linebreak $\feasible{\pol}{\dir}{\op}=\pol \cap ( \pol+ \frac{5}{2}\cdot (-1,0))$ is just a line segment $\ls$ whose lattice length is $\llen_{\Z^2}(\ls)=\frac{1}{2}$. But on the other hand, we have $\wid{\linf}(\feasible{\pol}{\dir}{\op})=1 > \frac{1}{2}$. 

\begin{figure}[h!]
\centering
\includegraphics[scale=0.2]{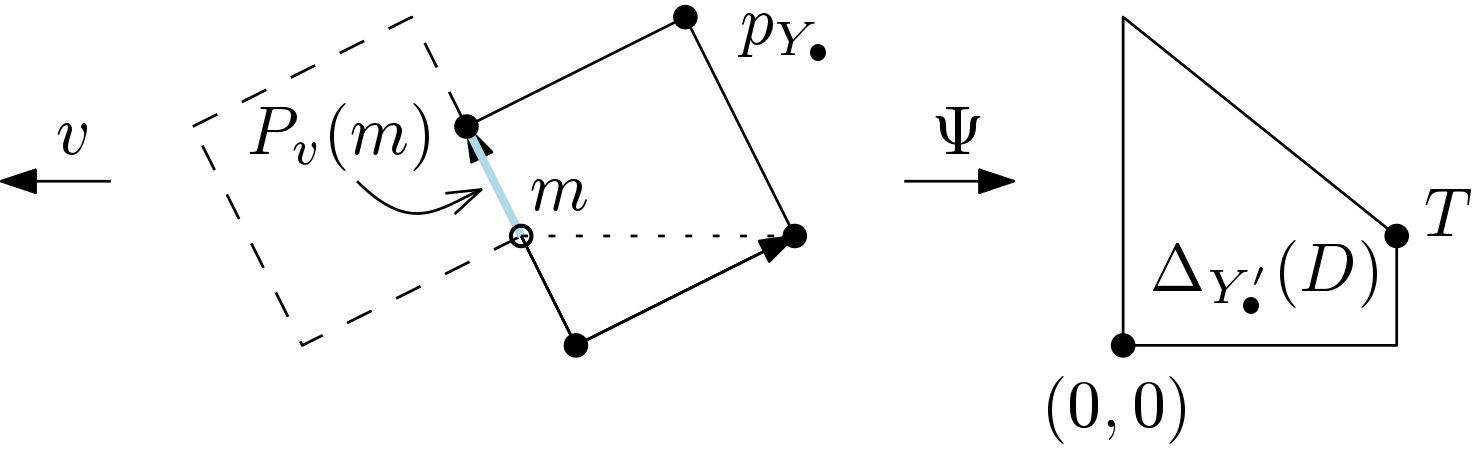}
\caption{An instance of a centrally-symmetric polytope $\pol$ that is not zonotopally well-covered.  }\label{fig_bad_square}
\end{figure}

It remains to argue, why any other direction $\dir \in \Z^2$ will also
fail. If we interpret $\pol$ as the Newton--Okounkov body
$\nob{\flag}{\ds}$ for some completely toric situation
$\var,\ds,\flag$ then the shifting process by the vector $\dir=(-1,0)$
yields the polytope on right in Figure~\ref{fig_bad_square} as the
Newton--Okounkov body $\nob{\flag'}{\ds}$ for the adjusted flag
$\flag'$, where $Y'_1$ is the curve determined by $\dir$ and
$Y_2'=\gen=(1,1)$.  Consider the Newton--Okounkov function $\fct'_\gen
\colon \nob{\flag'}{\ds} \to \R$. Since $\fct'_\gen (\a',b') \leq
\a'+b'$ and $\max_{(\a',\b') \in
  \nob{\flag'}{\ds}}{\fct'_\gen(\a',\b') }$ is independent of the
flag, this yields that $\max{\fct_\gen'} \leq \frac{7}{2}$. A straight
forward computation shows that any primitive direction $\dir \in \Z^2$
with $\|\dir\| > 1$ results in a vertex $(0,\b') \in
\nob{\flag'}{\ds}$ with $\b' > \frac{7}{2}$ which is a contradiction
to the above.\\
\\
Although $P$ is not the polytope of an ample divisor on a smooth
surface, it can be used as a starting point to construct such an
example: The minimal resolution $\pi \colon X^*_P \to X_P$ has a
centrally-symmetric fan. There is a `centrally-symmetric' ample
$\Q$-divisor on $X^*_P$ near the nef divisor $\pi^*D$. Now scale up
the resulting rational polygon to a lattice polygon. \\
\\ 
A similar argument applies to the polygon from Example 4.6 in \cite{castravet2020blownup}, which is depicted in Figure~\ref{fig:7eck}. The authors construct examples of projective toric surfaces whose blow-up at a general point has a non-polyhedral pseudo-effective cone. In this context they introduce, what they call \emph{good} polytopes. For our particular instance of a good polytope the authors argue, that all sections $s \in \sps ( \var,\she_\var(k\ds))$ will have order of vanishing at most $7k$ at the general point. However, all primitive directions $\dir \in \Z^2$ will produce a vertex of $\pwl(\pol)$ with coordinate sum $>7$. 
\end{remark}

\begin{figure}
    \centering
    \includegraphics[scale=0.12]{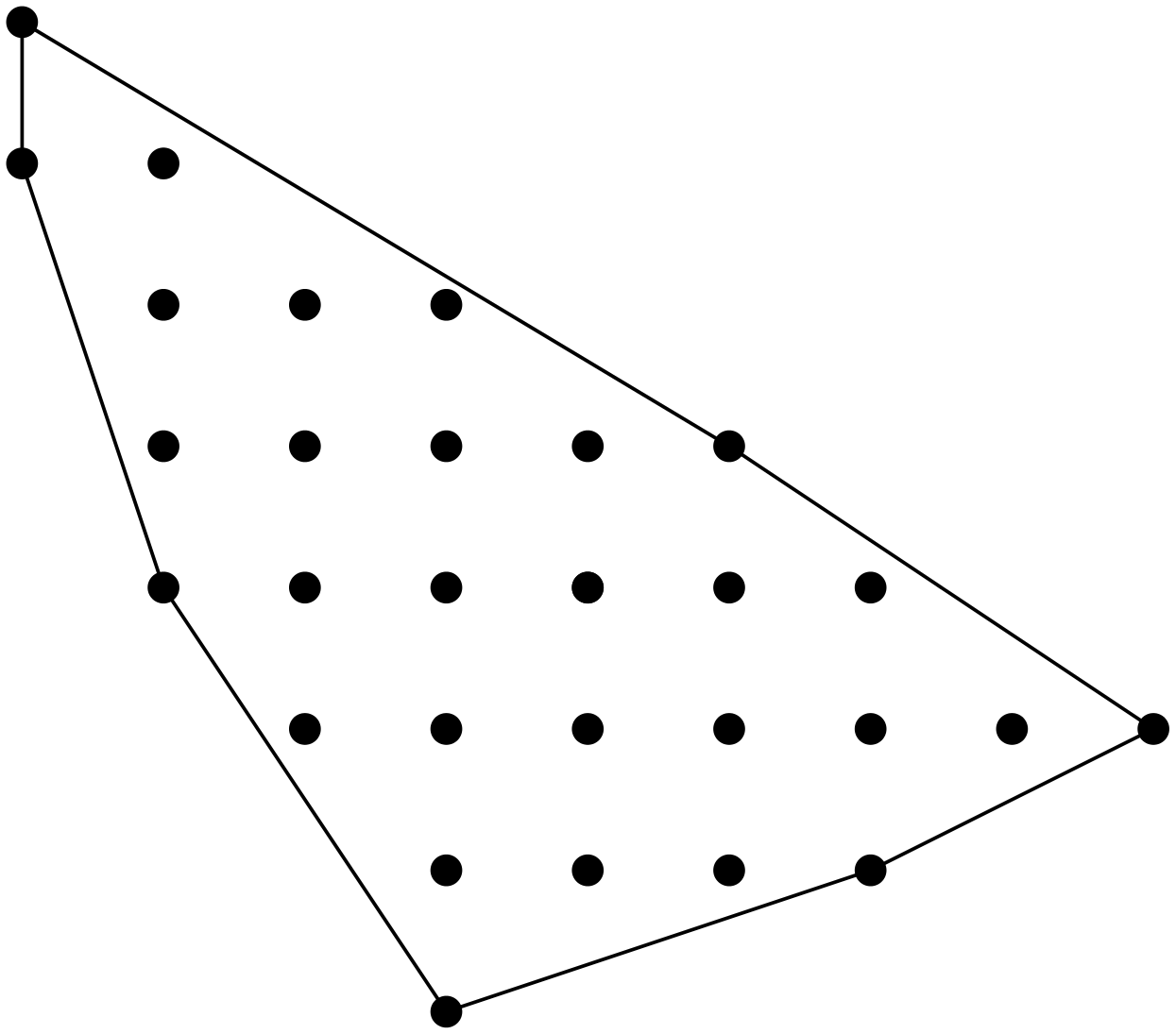}
    \caption{A polygon for which our approach does not work.}
    \label{fig:7eck}
\end{figure}


\section[Rationality of Certain Seshadri Constants]{Rationality of Certain Seshadri Constants on Toric Surfaces}
\label{sec:seshadri}

There is a direct link between the rationality of Seshadri constants on surfaces and that of  integrals of Newton--Okounkov functions. Let $\var$ be a smooth projective surface, $\ds$ an ample divisor, and $\point \in \var$ a point. We denote the blow-up of $\point$ with exceptional divisor $E$ by $\pi \colon \var' \to \var$. The \emph{Seshadri constant} is the invariant
\begin{equation}
\seh(\var,\ds;\point) \coloneqq \sup\{ t > 0 \suchthat \pi^*\ds - t E \text{ is nef}\}.
\end{equation} 
It measures the local positivity of $\ds$ at the point $\point$. Seshadri constants provide information on the shape of the nef and effective  cones of the surface $\var'$ in the direction of $-E$. Although they have been studied for over thirty years, several basic questions about them remain unanswered. One of the main questions is the rationality of $\seh(\var,\ds;\point)$. It is expected that there will be instances (even in dimension two) when irrational Seshadri constants occur (in fact, this would be consistent with Nagata's conjecture \cite{DKMS}), at the same time, no irrational example has been found so far.  In particular, it is known that Seshadri constants on del Pezzo, Enriques~\cite{SzembergEnriques}, abelian~\cite{BauerAbelian} and certain K3 surfaces~\cite{BauerQuartic,GalatiKnutsen,Knutsen} are rational. Certainly, if the blow-up of $\var$ at $\point$ has a finite rational polyhedral effective cone, then $\seh(\var,\ds;\point)$ is forced to be rational.  \\
\\
Our lack of knowledge about the rationality of Seshadri constants on surfaces is  all the more mysterious, since in dimension two there is one way in which $\seh(\var,\ds;\point)$ can be irrational: if it is equal to $\sqrt{(\ds^2)}$ \cite[Section 5.1]{lazarsfeld2004positivity} and $(\ds^2)$ is not a square. If the latter arithmetic condition does not hold, then the rationality of $\seh(\var,\ds;\point)$ is equivalent to the existence of a (necessarily negative) curve $C$ on the blow-up of $\var$ at $\point$  orthogonal to $\pi^*\ds-tE$ for some $t<\sqrt{(\ds^2)}$. In this sense the irrationality of $\seh(\var,\ds;\point)$ is evidence for the non-existence of certain irreducible curves of negative self-intersection on the blow-up. 

\begin{remark}
By the duality between the nef and effective cones on a surface, if \linebreak $\seh(\var,\ds;\point) < \mu_{E}(\pi^*\ds)$ then both numbers are rational. 
As a consequence, if one can find an effective divisor of the form $\pi^*\ds-tE$  with  $\sqrt{(\ds^2)}< t$ then $\seh(\var,\ds;\point)\in\Q$. 
\end{remark}

As the first known rationality criterion we adjust Theorem~3.6 and
Remark~3.7 from \cite{ito14} to our situation.

\begin{theorem} \label{thm:ito-width}
  In the situation of our general set-up~\ref{notation},
  if $\wid{\linf}(\pol) \le \relw{\pol}{\dir}$ then $\seh(\var,\ds;\gen)
  = \wid{\linf}(\pol)$.

  In particular, in this case, $\seh(\var,\ds;\gen)$ is rational.
\end{theorem}

In~\cite{sano2014seshadri} Sano studies Seshadri constants on rational surfaces with anti-canonical pencils.  More precisely, he considers a smooth rational surface $\var$ that is either a composition of blow-ups of $\PP^2$ or of a Hirzebruch surface $\hir_d$ such that $\dim|-\cand{\var}| \geq 1$. In terms of the corresponding polytope this means that $\divp{-\cand{\var}}$ contains at least two lattice points. In these cases, he gives explicit formulas for the Seshadri constant $\seh(\var,\ds;\gen)$ of an ample divisor $\ds$ at a general point $\gen \in \var$ in~\cite[Theorem 3.3 and Corollary 4.12]{sano2014seshadri}. As a consequence he obtains rationality in the cases above as observed in Remark~4.2.\\
\\
In~\cite{lundman2015computing} Lundman computes Seshadri constants at a general point $\gen$ for some classes of smooth projective toric surfaces. It follows in particular that the Seshadri constant is rational in these cases. The characterization of the classes involve the following definitions.

\begin{definition}
Let $\lb$ be a line bundle on a smooth variety $\var$ and $\point \in \var$ a smooth point with maximal ideal $\maxi_\point \subseteq \she_\var$. For a $k \in \N$ consider the map
\begin{eqnarray*}
j_\point^k \colon \sps(\var,\lb) & \to & \sps(\var,\lb \otimes \she_\var / \maxi^{k+1}_\point) \\
s & \mapsto & \left(s(\point),\ldots,\frac{\partial ^t s}{\partial\underline{\point}^t}(\point),\ldots \right)_{t \leq k},
\end{eqnarray*}
where $\underline{\point}=(\point_1,\ldots,\point_\dimm)$ is a local system of coordinates around $\point$. We say that $\lb$ is \emph{$k$-jet spanned} at $\point$ if the map $j_\point^k$ is surjective. We denote by $\djs{\lb}{\point}$ the largest $k$ such that $\var$ is $k$-jet spanned at $\point$ and call it the \emph{degree of jet separation} of $\lb$ at $\point$.  
\end{definition}

So the map $j_\point^k$ takes $s$ to the terms of degree at most $k$ in the Taylor expansion of $s$ around $\point$. For $\var$ a projective toric variety let $s_0,\ldots,s_d$ be a basis for $\sps(\var,\lb)$. Then $\lb$ is $k$-jet spanned at $\point\in \var$ if and only if the \emph{matrix of $k$-jets}
\begin{equation*}
J_k(\lb) \coloneqq (J_k(\lb))_{i,j} \coloneqq \left(\frac{\partial ^{|t|}}{ \partial_{\point_{t_1}}\partial_{\point_{t_2}} \cdots \partial_{\point_{t_n}}}(s_i) \right)_{0 \leq i \leq d, 0\leq |t| \leq k}
\end{equation*}
has maximal rank when evaluated at the point $\point$, where ${t=(t_1,\ldots,t_n)\in \N^\dimm}$ and \linebreak ${|t|=|t_1+\cdots +t_\dimm|}$.

\begin{definition}[{\cite[compare Definition 1.15]{polyhedralAdjunction}}]
Let $\var$ be a smooth projective toric variety and $\ds$ a torus-invariant divisor on $\var$. We define the \emph{codegree} $\qco(\ds)$ as
\begin{equation*}
\qco(\ds) \coloneqq (\sup \{t > 0 \suchthat  \divp{t \cand{\var}+\ds} \text{ is non-empty }\})^{-1}
\end{equation*} and call the polytope $\core(\divp{\ds}) \coloneqq \divp{\qco(\ds)^{-1}\cand{\var}+\ds}$ the \emph{core} of $\divp{\ds}$.
\end{definition}

\begin{theorem}[{\cite[Theorem 1]{lundman2015computing}}]
Let $\var$ be a smooth toric surface and $\lb$ an ample line bundle. If $\var$ is a projective bundle or $\djs{\lb}{\gen} \leq 2$, then $\seh(\var,\lb;\gen)=\djs{\lb}{\gen}$. 
\end{theorem}

The other Theorems in~\cite{lundman2015computing} that yield rationality of Seshadri constants both require $\core(\divp{\ds})$ to be a line segment.\\
\\
Just as the rationality of Seshadri constants follows from that of the corresponding pseudo-effective thresholds, it can also be deduced from the rationality of the associated integral in the following way.

\begin{corollary}[{\cite[Corollary 4.5]{KMR}}]
Let $\var$ be a smooth projective surface, $\point \in \var$, and $\ds$ an ample Cartier divisor on $\var$. Then $\seh(\var, \ds;\point)$ is rational if $\int_{\nob{\flag}{\ds}}{\fct_\point}$ is rational, where $\fct_\point$ is the Newton--Okounkov function coming from the geometric valuation associated to $\point$ and $\flag$ any admissible flag. 
\end{corollary}

We apply this criterion to an example for which Lundman's and Sano's criteria do not apply. 

\begin{example}\label{ex_16eck}
We consider a blow-up $\pi \colon \var \to \PP^2$ of the projective plane in 13 points, namely the toric variety $\var$ whose associated fan $\fan$ is depicted in Figure~\ref{fig_fan_16eck}. 

\begin{figure}[h!]
\centering
\includegraphics[scale=0.2]{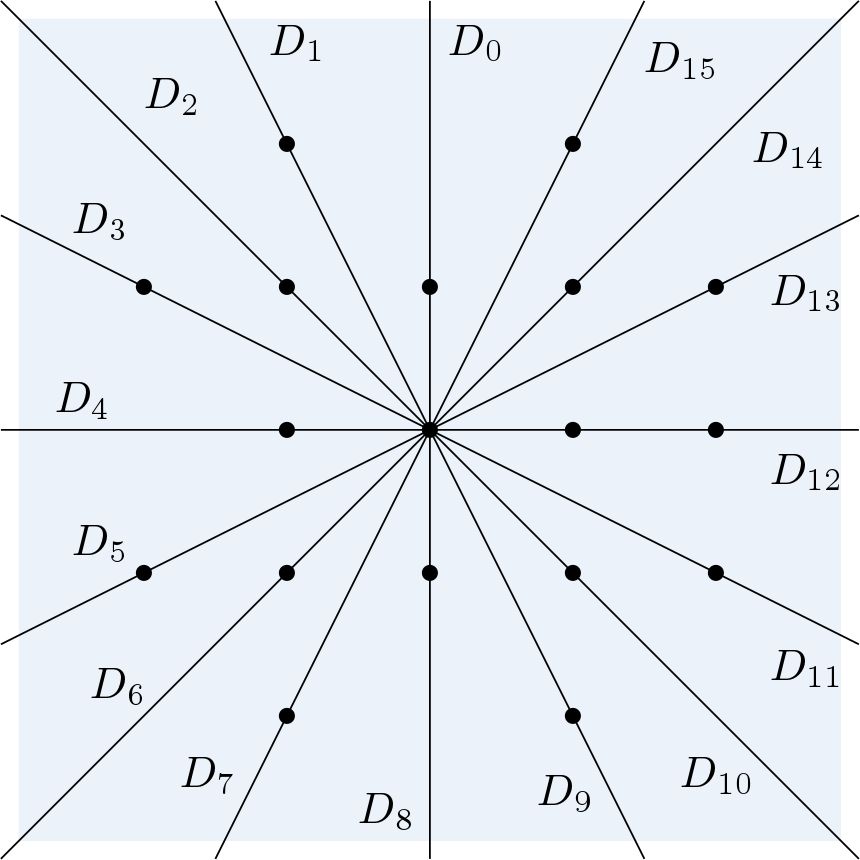}
\caption{The fan $\fan$ with associated torus-invariant prime divisors $\ds_0,\ldots,\ds_{15}$.} \label{fig_fan_16eck}
\end{figure}

The torus-invariant prime divisors are denoted by $\ds_0,\ldots,\ds_{15}$ and choose  $\ds=\ds_1+2\ds_2+6\ds_3+5\ds_4+15\ds_5+11\ds_6+19\ds_7+9\ds_8
+18\ds_9+10\ds_{10}+13\ds_{11}+4\ds_{12}+4\ds_{13}+\ds_{14}$ as an ample divisor on $\var$. For the torus-invariant flag ${\flag \colon \var \supseteq Y_1 \supseteq Y_2}$ with ${Y_1=\overline{ \{y=0\}}}$ and $Y_2=(0,0)$ this gives the polytope $\divp{\ds}$ in Figure~\ref{fig_16eck} as the Newton--Okounkov body $\nob{\flag}{\ds}$. We have $\dim|-\cand{\var}|=0$, $\core(\divp{\ds})$ is a point and the degree of jet separation is $\djs{\lb}{\gen}=9$. Thus this example does not fall in any of the classes covered by Sano or Lundman.  

\begin{figure}[h!]
\centering
\includegraphics[scale=0.2]{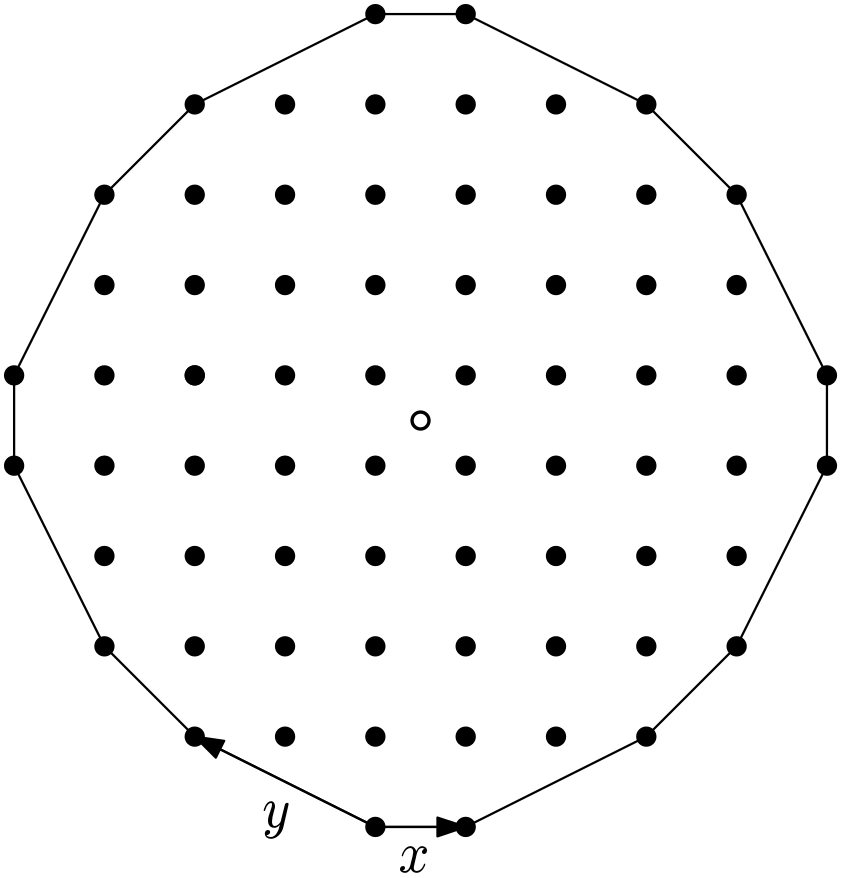}
\caption{The Newton--Okounkov body $\nob{\flag}{\ds} \cong \divp{\ds}$ with $\core(\divp{\ds})$ highlighted.}\label{fig_16eck}
\end{figure}

We claim that the Seshadri constant $\seh(\var,\ds;\gen)$ is rational. To verify this claim we consider the Newton--Okounkov function $\fct_\gen'$ on $\nob{\flag'}{\ds}$ coming from the geometric valuation $\ord_\gen$ at the general point $\gen=(1,1)$ and argue that its integral takes a rational value. In order to do this, consider the flag $\flag' \colon \var \supseteq Y_1' \supseteq Y_2'$, where $Y_1'$ is the curve given by the local equation $x^{-1}-1=0$ and $Y_2'=\gen$. Thus we obtain the Newton--Okounkov body  $\nob{\flag'}{\ds}$ with respect to this flag by the shifting process via the vector $v=(-1,0)$ as explained in Section~\ref{sec_nob_toric}. This gives the polytope $\nob{\flag'}{\ds}$ shown in Figure~\ref{fig_nob_16eck}.

\begin{figure}[h!]
\centering
\includegraphics[scale=0.18]{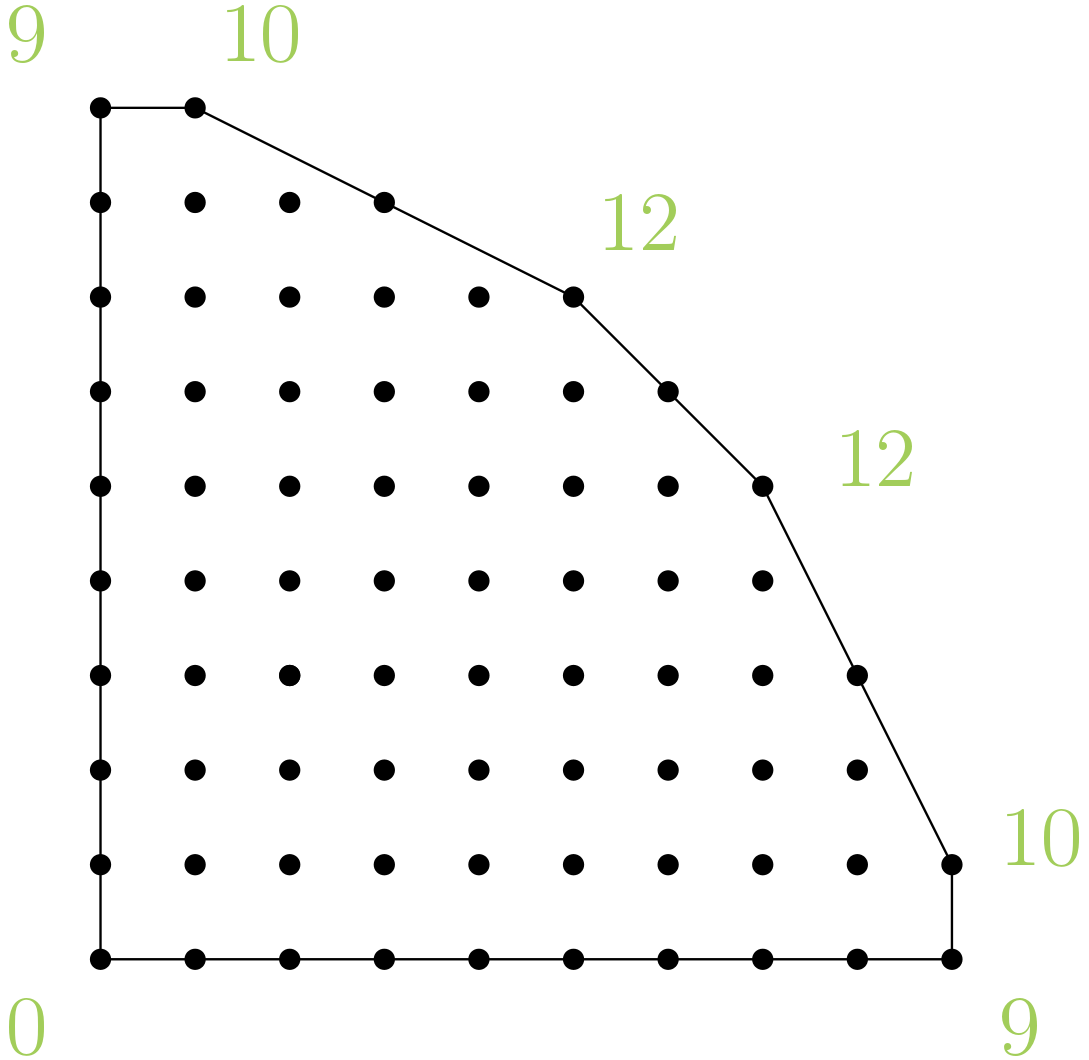}
\caption{The Newton--Okounkov body $\nob{\flag'}{\ds}$ with respective values of $\fct_\gen'$.}\label{fig_nob_16eck}
\end{figure}

We claim that the Newton-Okounkov function $\fct_\gen'$ on $\nob{\flag'}{\ds}$ that comes from the geometric valuation $\ord_\gen$ is given as $\fct'_\gen(\a',\b')=\a'+\b'$ for all $(\a',\b') \in \nob{\flag'}{\ds}$. To prove this, we consider the following global sections of $\sps(\var,\she_\var(\ds))$ as in Table~\ref{tab_specific_sec}. 

The sections are chosen in a way such that they get mapped to the vertices, when building the new Newton--Okounkov body $\nob{\flag'}{\ds}$ and a such that the order of vanishing  is ${\ord_\gen(s)=\a'+\b'}$ for a section $s$ that gets mapped to the point $(\a',\b') \in \nob{\flag'}{\ds}$. For the vertices in $\ver(\nob{\flag'}{\ds})$ these values realize a lower bound for the function $\fct_\gen'$.

\begin{center}
\begin{table}[h!]
 \begin{tabular}{|l|| c c c |} 
 \hline
 global section $s$ & image in $\nob{\flag}{\ds}$ & image in $\nob{\flag'}{\ds}$ & $\ord_\gen(s)$ \\ [0.5ex] 
 \hline\hline
 $s_1(x,y)=(x-1)(x^2y-1)^9$ & $(0,0)$ & $(1,9)$ & $10$ \\ 
 \hline
 $s_2(x,y)=x(x^2y-1)^9$ & $(1,0)$ & $(0,9)$ & $9$ \\
 \hline
 $s_3(x,y)=x^4y^4(x-1)^9(x^2y-1)$ & $(4,4)$ & $(9,1)$ & $10$ \\
 \hline
 $s_4(x,y)=x^6y^5(x-1)^9$ & $(6,5)$ & $(9,0)$ & $9$ \\
 \hline
 $s_5(x,y)=y(x-1)^5(x^2y-1)^7$ & $(0,1)$ & $(5,7)$ & $12$ \\  
 \hline
 $s_6(x,y)=xy^2(x-1)^7(x^2y-1)^5$ & $(1,2)$ & $(7,5)$ & $12$ \\  
 \hline
 $s_7(x,y)=x^{19}y^9$ & $(19,9)$ & $(0,0)$ & $0$ \\ 
 \hline
\end{tabular}
\caption{\label{tab_specific_sec}Global sections of $\she_\var(\ds)$ that realize lower bounds for the order of vanishing $\ord_\gen$.}
\end{table}
\end{center}

 Since the function $\fct_\gen'$ has to be concave, this yields $\fct_\gen'(\a',\b')=\a'+\b'$ on the entire Newton--Okounkov body. For the integral we obtain
\begin{equation*}
\int_{\nob{\flag'}{\ds}}{\fct_\gen'}=\frac{1295}{3},
\end{equation*}
which is rational and therefore the Seshadri constant $\seh(\var,\ds;\gen)$ is rational.

Although proving rationality of the Seshadri constant did not require knowing the values of the function $\fct_\gen$ on the Newton--Okounkov body $\nob{\flag}{\ds}$, determining these values in this particular example is of independent interest. It turns out that the approach of choosing sections whose Newton polytopes are zonotopes with prescribed edge directions is not always sufficient to maximize the order of vanishing at the general point $\gen$. For the function $\fct_\gen$ we expect 22 domains of linearity as shown in Figure~\ref{fig_regions_16eck} that arise from the shifting process in the direction of $v=(-1,0)$.

\begin{figure}[h!]
\centering
\includegraphics[scale=0.2]{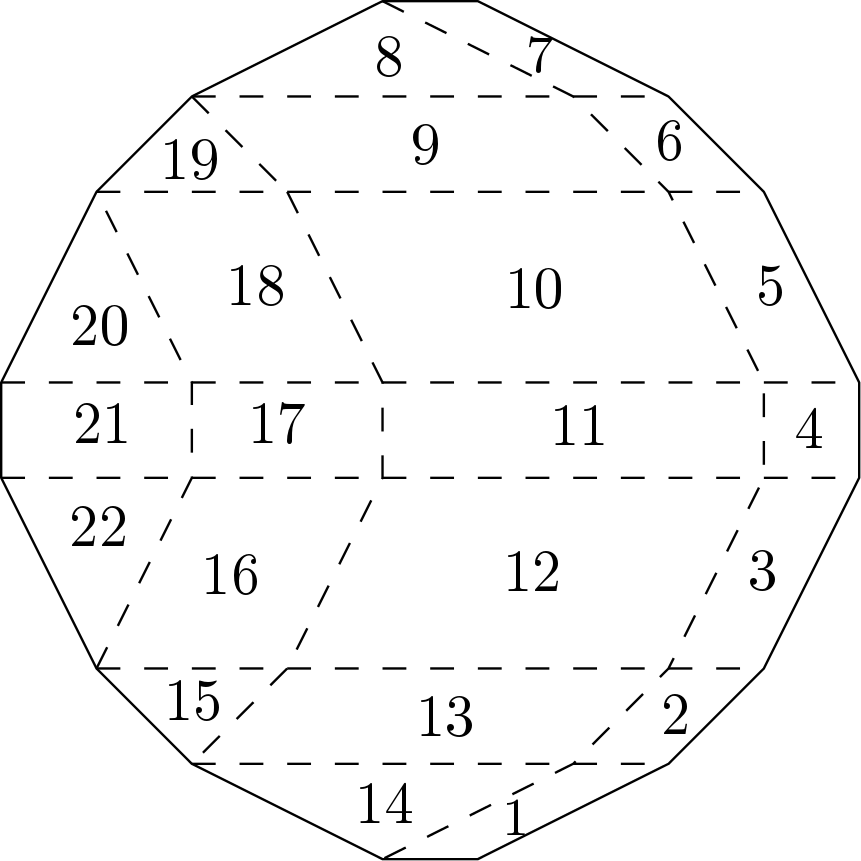}
\caption{Expected domains of linearity of the function $\fct_\gen$ on the Newton--Okounkov body $\nob{\flag}{\ds}$.}\label{fig_regions_16eck}
\end{figure}

As seen in the Appendix in Table~\ref{table_2} for the domains $1,\ldots,9,13,14,15$, and $19$ zonotopes using only edge directions of $\nob{\flag}{\ds}$ are sufficient. For the domains $10,11$, and $12$ we need a Minkowski sum of those edge directions and `small' triangles that have a high order of vanishing at $(1,1)$. The section $s(x,y)=x^3y^2-3xy+y+1$ for instance has order of vanishing $\ord_\gen(s)=2$ and its Newton polytope $\np(s)$ is depicted in Figure~\ref{fig_triangle_1}.

\begin{figure}[h!]
\centering
\includegraphics[scale=0.1]{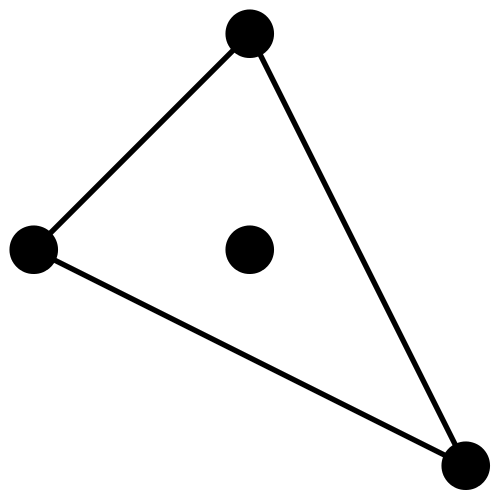}
\caption{The Newton polytope $\np(s)$ of the section $s(x,y)=x^3y^2-3xy+y+1$.}\label{fig_triangle_1}
\end{figure}

For the remaining regions $16,17,18,20,21$, and $22$ global sections with the desired order of vanishing at $\gen$ could not be found via computations up to $k=12$. We expect $\fct$ to take the values shown in Figure~\ref{fig_values16}.

\begin{figure}[h!]
    \centering
    \includegraphics[scale=0.18]{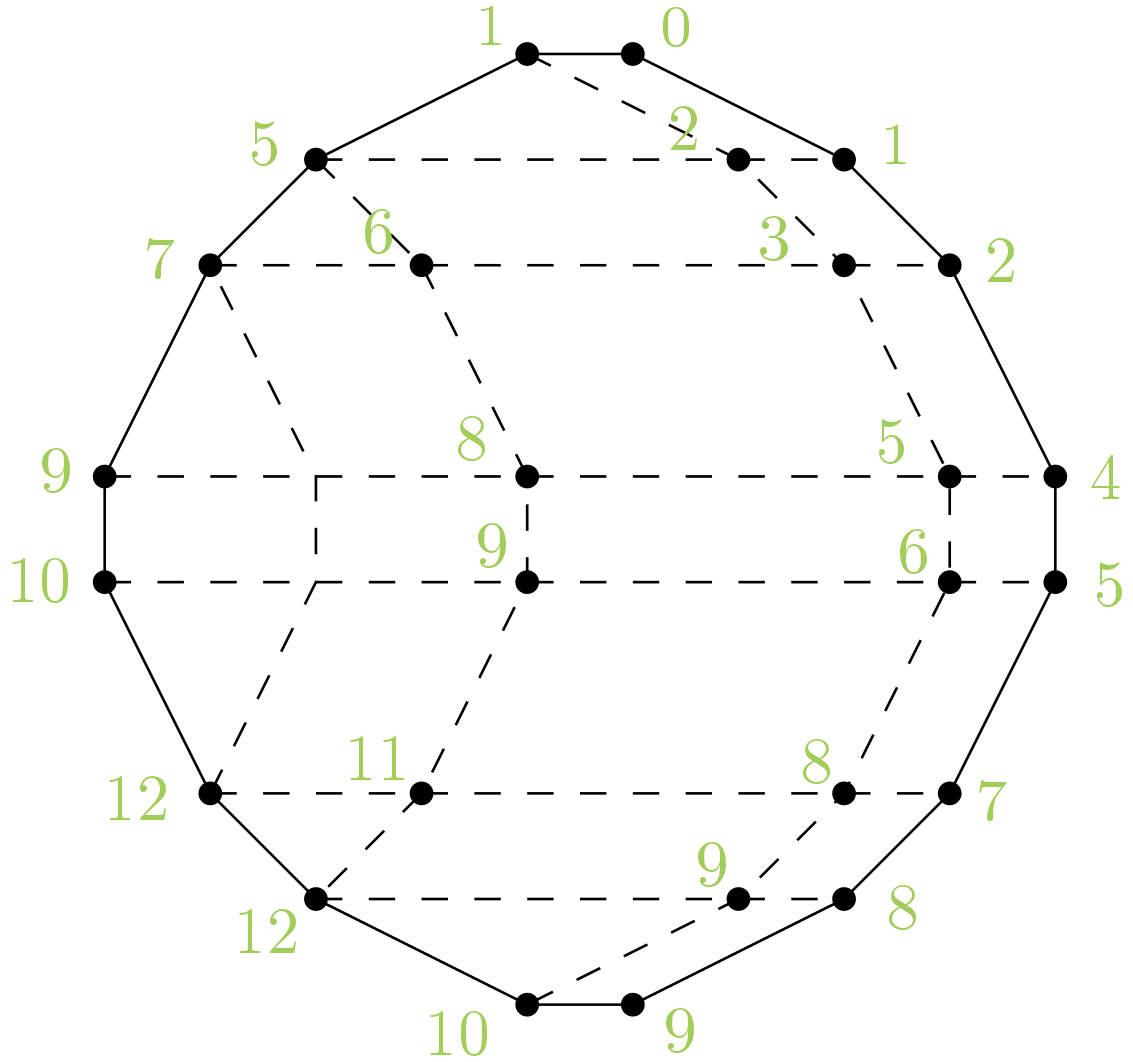}
    \caption{The values of $\fct$ on the Newton--Okounkov body $\nob{\flag}{\ds}$.}
    \label{fig_values16}
\end{figure}

\end{example}

The approach for proving rationality of the Seshadri constant applies to a certain class of polytopes. To describe this class we need to introduce the following terms.

\begin{definition}
Let $\pol \subseteq \R^2$ be a polygon and $\dir \in \Z^2$ a primitive direction. 
Set \linebreak ${\ver(\pol,\dir) \coloneqq \pwl^{-1}(\ver(\pwl(\pol)))}$ for the piecewise linear isomorphism from Section \ref{sec_tilting} and call it the \emph{relevant vertex set of $\pol$ with respect to $\dir$}.
\end{definition}

\begin{definition}\label{def_weakly}
Let $\pol \subseteq \R^2$ be a polygon. Let $\dir \in \Z^2$ be a primitive vector, and $\linf \in \dir^\perp$ a primitive integral functional. We call $\pol$ \emph{weakly zonotopally well-covered with respect to $(\dir,u)$} if for all points $\op \in \ver(\pol,\dir)$ the set $\feasible{\pol}{\dir}{\op}$ 
contains a zonotope ${\ls_1+\ldots+\ls_\ell}$ with none of the $\ls_i$ parallel to $\dir$,
such that 
\begin{equation*}
\sum_{i=1}^{\ell}{\llen_{\Z^2}(\ls_i)}=\wid{\linf}(\feasible{\pol}{\dir}{\op}).
\end{equation*}
The polygon $\pol$ is \emph{weakly zonotopally well-covered} if it is with respect to some $(\dir,\linf)$. 
\end{definition}

\begin{theorem}\label{thm_seshadri}
Let $\var$ be a smooth projective toric surface and $\ds$ an ample torus-invariant divisor on $\var$ with associated Newton--Okounkov body $\nob{\flag}{\ds}$ for an admissible torus-invariant flag $\flag$. If the polytope $\nob{\flag}{\ds}$ is weakly zonotopally well-covered, then
\begin{enumerate}
    \item we can determine $\int_{\nob{\flag}{\ds}}{\fct_\gen}$. 
    \item the Seshadri constant $\seh(\var, \ds;\gen)$ is rational. 
    \item the maximum $\max_{\nob{\flag}{\ds}}{\fct_\gen}$ is attained at the boundary of $\nob{\flag}{\ds}$.
\end{enumerate}
\end{theorem} 

\begin{proof}
Since all input data is torus-invariant, the Newton--Okounkov body $\nob{\flag}{\ds}$ is isomorphic to the polytope $\divp{\ds}$ for any admissible torus-invariant flag $\flag$. By assumption, this polytope is weakly zonotopally well-covered, so let $\dir=(\dir_1,\dir_2) \in \Z^2$ be its associated primitive direction. Consider the flag $\flag' \colon \var \supseteq \cur \supseteq \{\gen\}$, where $\cur$ is the curve given by the local equation $x^{\dir_1}y^{\dir_2}-1=0$ and $\gen=(1,1)$ is a general point on $\cur$. Then the shifting process explained in Section~\ref{sec_nob_toric} yields the Newton--Okounkov body $\nob{\flag'}{\ds}$ with respect to this new flag. 
By Corollary~\ref{cor_mv} this process relates the Newton--Okounkov bodies via a piecewise linear isomorphism $\pwl \colon \nob{\flag}{\ds} \overset{\sim}{\longrightarrow} \nob{\flag'}{\ds}$.\\
\\
 We show that the Newton--Okounkov function $\fct'_\gen \colon \nob{\flag'}{\ds} \to \R$, that comes from $\ord_\gen$, satisfies ${\fct'_\gen(\a',\b')=\a'+\b'}$ for all vertices $T=(\a',\b') \in \ver(\nob{\flag'}{\ds})$. In order to do so, apply the arguments of the proof of Theorem~\ref{thm_fct_zonotopally} to all $m \in \ver(\nob{\flag}{\ds},\dir)$. Together with the facts that $\fct'_\gen$ is  concave and has $\a'+\b'$ as an upper bound it follows that $\fct'_\gen(\a',\b')=\a'+\b'$ on the entire Newton--Okounkov body. Rationality of the integral $\int_{\nob{\flag'}{\ds}}{\fct'_\gen}$ yields rationality of the Seshadri constant $\seh(\var,\ds;\gen)$. Since the maximum is independent of the flag, and $\fct'_\gen$ is linear, it is attained at the boundary of $\nob{\flag}{\ds}$.
\end{proof}

Note, that zonotopally well-covered implies weakly zonotopally well-covered.

\begin{example}
To illustrate the proof we stick to Example ~\ref{ex_16eck}. The polytope $\nob{\flag}{\ds}$ is weakly zonotopally well-covered with respect to $(\dir,\linf)=((-1,0),(0,1))$. Consider, for instance, the vertex  $T=(7,5) \in \nob{\flag'}{\ds}$. Its preimage under the piecewise linear isomorphism $\pwl$ is $\op=(1,2) \in \nob{\flag}{\ds}$ and $\relw{\nob{\flag}{\ds},\op}{\dir}=7$. A global section which is mapped to $\op$ and $T$ respectively, is
\begin{equation*}
s(x,y)=xy^2 \cdot(x-1)^{7} \cdot (x^2y-1)^5
\end{equation*}
as seen in Figure~\ref{fig_sec_16eck} with $\ord_\gen(s)=7+5=12$. 

\begin{figure}[h!]
\centering
 \includegraphics[scale=0.32]{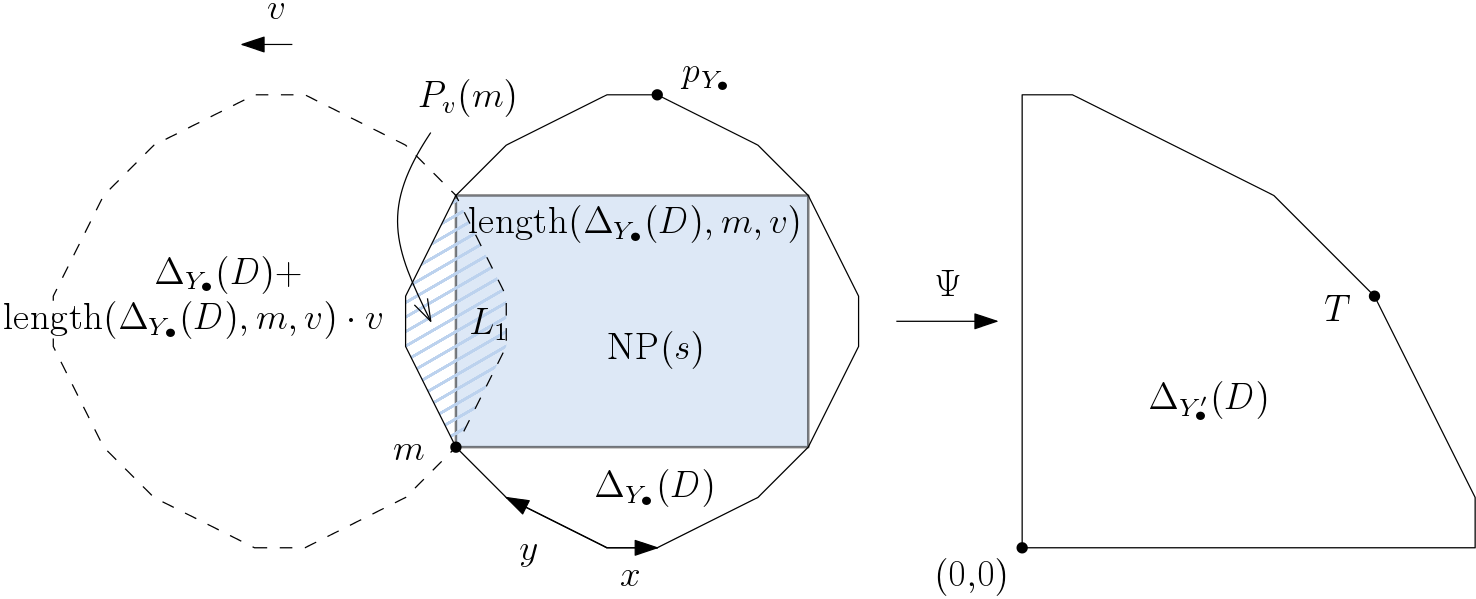}
\caption{The setup of the proof of Theorem~\ref{thm_seshadri} in the context of Example~\ref{ex_16eck}. }\label{fig_sec_16eck}
\end{figure}
\end{example}

We can construct classes of polarized toric surfaces for which our
method of guessing sections using convex geometry yields rationality
of the Seshadri constant while other methods like Ito's width bound,
Lundman's core criterion or Sano's anti-canonical pencil do not
apply. We explain the general method and illustrate it with an
explicit example. (The hard part is to find polygons for which the
above methods do not work.) \\
\\
In order to ensure that the core of our polygons is a point, we need
to introduce some machinery which might be of independent interest.
In analogy with the Fine interior of a rational polytope
(compare~\cite[\S4.2]{Fine}, \cite{BatyrevKasprzykSchaller}) we define the
\emph{Fine adjoint} $\adjoint{\pol}{c}$ of a convex body $\pol \subseteq \latm_\R$ and a
parameter $c > 0$ as follows.
$$
\adjoint{\pol}{c} \coloneqq \bigcap_{\linf \in N \setminus \{\origin\}}
\left\{ m \in M_\R \suchthat \langle \linf, m \rangle \ge \min \langle
  \linf, \pol \rangle + c \right\}
$$
If $P$ is a rational polytope and $c \in \Q_{>0}$, this is again a
rational polytope; $\adjoint{\pol}{1}$ is called \emph{Fine interior}
in~\cite{BatyrevKasprzykSchaller}.
If the toric variety $X_\pol$ associated with
$\pol$ has at most canonical singularities, this agrees with the
standard adjoint $\pol^{(c)}$ where the intersection is taken only
over facet defining $\linf$'s (see~\cite{polyhedralAdjunction}).

The \emph{Fine codegree} $\codeg{\pol}$ of $\pol$ is then the minimal $c$ for
which $\adjoint{\pol}{1/c} \neq \emptyset$. Finally, we call the last
non-empty Fine adjoint $\adjoint{\pol}{1/\codeg{\pol}}$ of $\pol$ its
\emph{Fine core} $\fineCore{\pol}$.

If $m \in \relint( \fineCore{\pol})$, call those $\linf \in N \setminus
\{\origin\}$ for which $\langle \linf, m \rangle = \min \langle \linf,
\pol \rangle + 1/\codeg{\pol}$ \emph{essential} for $\pol$
(compare~\cite[Lemma~2.2]{polyhedralAdjunction}).

For rational $\pol$ the Fine codegree will be a rational number and
hence the Fine core will be a rational polytope of positive
codimension. Figure~\ref{fig:core-examples} illustrates that in the
case of non-canonical singularities, the dimensions of core and Fine
core can differ, and even if both are points, they need not agree. 

\begin{figure}[htb]
  \centering
  \includegraphics[scale=0.15]{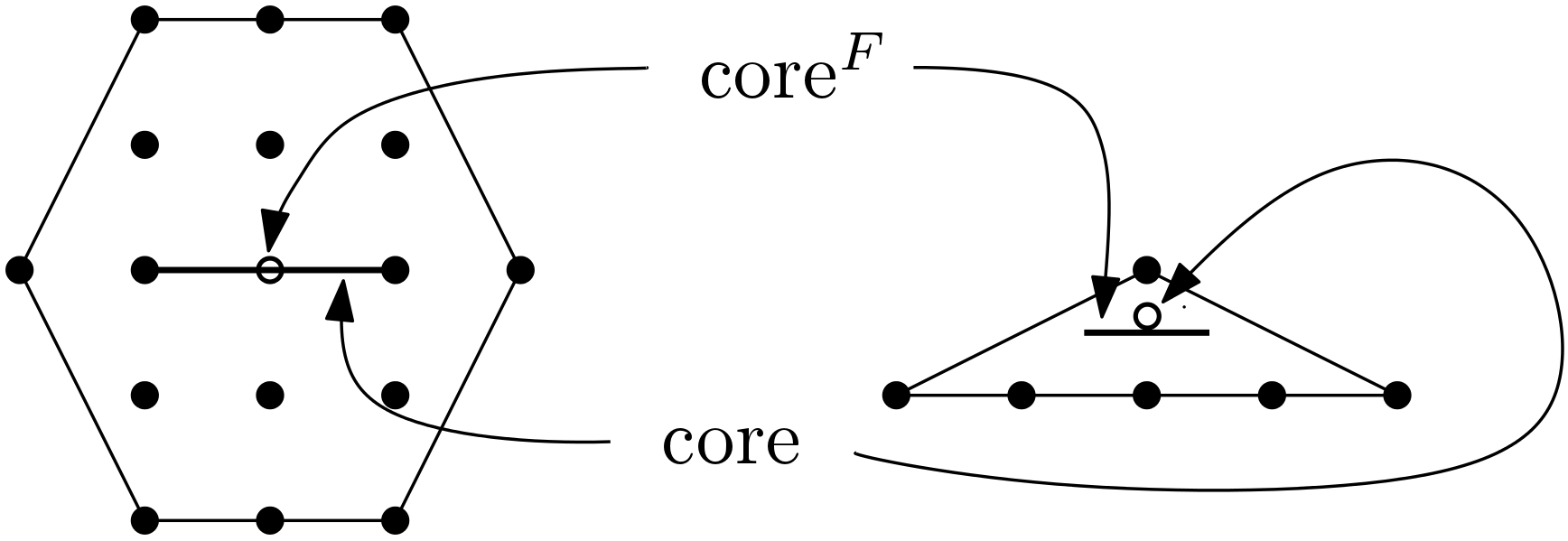}
  \caption{Core versus Fine core.}
  \label{fig:core-examples}
\end{figure}

\begin{lemma} \label{lemma:FineCore}
  Suppose $P$ and $Q$ are polytopes in $\latm_\R$ whose Fine cores are points \linebreak
  ${\fineCore{P} = \{m_P\}}$ and $\fineCore{Q} = \{m_Q\}$,
  respectively. Suppose further that the $\linf \in N \setminus
  \{\origin\}$ which are essential for both $P$ and $Q$ positively
  span $N_\R$.

  Then the Fine core of $kP + Q$ is the point $k m_P + m_Q$ for all $k
  \ge 1$.

  In particular, if $X_{P+Q}$ has at most canonical
  singularities, then the (usual) core of $kP + Q$ is this point.
\end{lemma}

With these preparations, we can describe our construction.
We write $\deg(P) = 2 \operatorname{area}(P)$ for the normalized
volume of $P$.

\begin{theorem} \label{thm:seshadri-construction}
   Suppose $P$ is a lattice polygon whose Fine core is a point
   such that $\wid{}(P) > \sqrt{\deg(P)}$ and $P$ supports a Laurent
   polynomial $s$ which vanishes to order $> \sqrt{\deg(P)}$ at
   $\gen$.

   Then there is a lattice polygon $Q$ so that for $k \gg 0$ the
   polygon $kP+Q$ satisfies
   \begin{enumerate}
   \item $\var_{kP+Q}$ is smooth, \label{thm:seshadri-construction:smooth}
   \item $\core(kP+Q)$ is a point, \label{thm:seshadri-construction:core}
   \item $\djs{D_{kP+Q}}{\gen} \ge k+1$, \label{thm:seshadri-construction:djs}
   \item $h^0(-K_{X_{kP+Q}}) = 1$, \label{thm:seshadri-construction:no-pencil}
   \item $\wid{}(kP+Q) >
     \sqrt{\deg(kP+Q)}$ \label{thm:seshadri-construction:width}
   \item $kP+Q$ supports a Laurent polynomial which vanishes to
     order $> \sqrt{\deg(kP+Q)}$ at
     $\gen$. \label{thm:seshadri-construction:section}
   \end{enumerate}
   In particular, $\seh(\var_{kP+Q},\ds_{kP+Q};\gen) \in \Q$.   
\end{theorem}

\begin{example}
Specific examples of such polygons $P$ are the triangles 
\begin{equation*}
\Delta^{(d)}
= \conv\left(d \cdot \Delta \cup \{(-1,-1)\} \right),
\end{equation*}
where $\Delta$
denotes the standard triangle $\conv\left((0,0),(1,0),(0,1)\right)$
from Figure~\ref{fig:Seshadri-examples}. Their parameters are
$\deg(\Delta^{(d)}) = d^2+2d$, $\wid{}(\Delta^{(d)}) = d+1$, and they
support a section which vanishes to order $\ge d+1$ at $\gen$ simply
because they contain more than $\dim_\C \C[x,y]/\langle x,y \rangle^d
= \binom{d+1}{2}$ lattice points so that the linear map
\begin{align*}
  \left\{ s \in \C[x,y] \suchthat \supp(s) \subseteq 
    P+(1,1) \right\} \quad &\to \quad \C[x,y]/\langle x,y \rangle^d \\
  s(x,y) \quad &\mapsto \quad s(x+1,y+1)
\end{align*}
must have a kernel. Specifically, for $P=\Delta^{(1)}$, we have
$\deg(P)=3$, $\wid{}(P)=2$, and $s = x+y+1/xy-3$ is a section which
vanishes to order $2$ at $\gen$.

\begin{figure}[htb]
  \centering
  \includegraphics[scale=0.17]{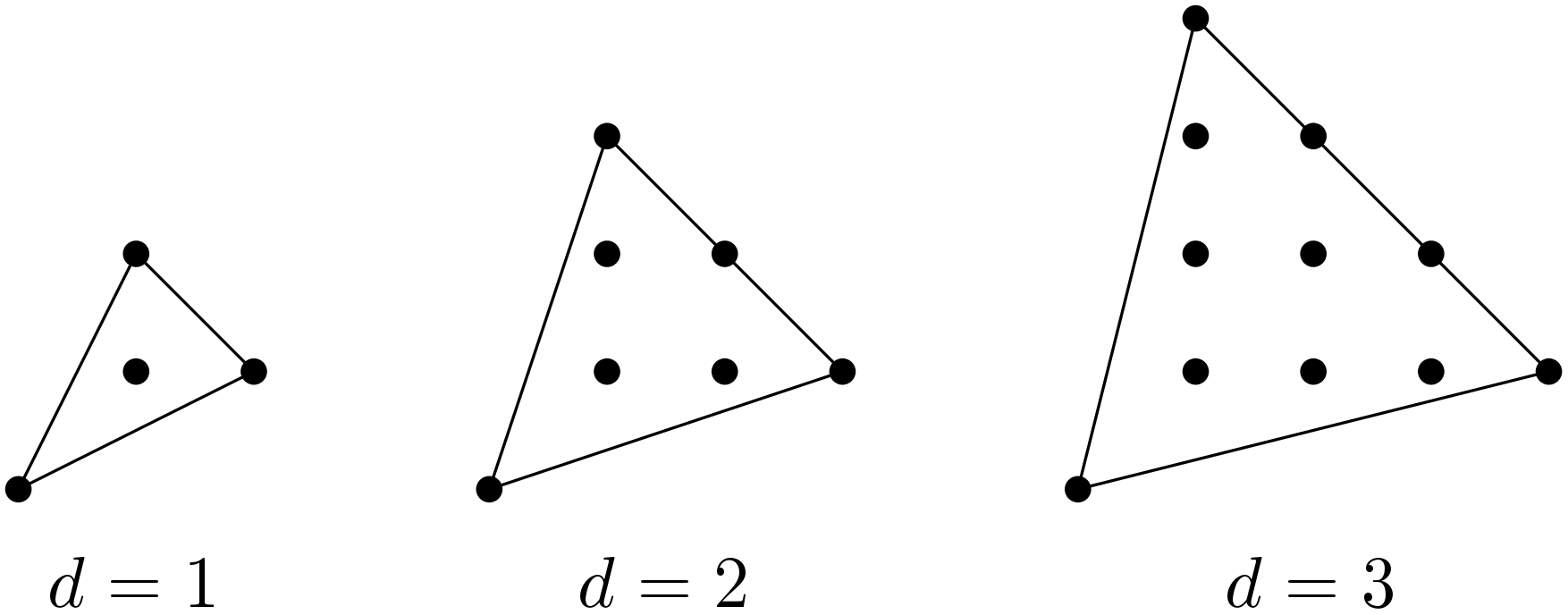}
  \caption{Examples of wide polygons $\Delta^{(d)}$ with small area
    for $d=1,2,3$.}
  \label{fig:Seshadri-examples}
\end{figure}

For our specific $P = \Delta^{(1)}$, the polygon $Q$ in
Figure~\ref{fig:Seshadri-example-Q} does the job.

\begin{figure}[htb]
  \centering
  \includegraphics[scale=0.11]{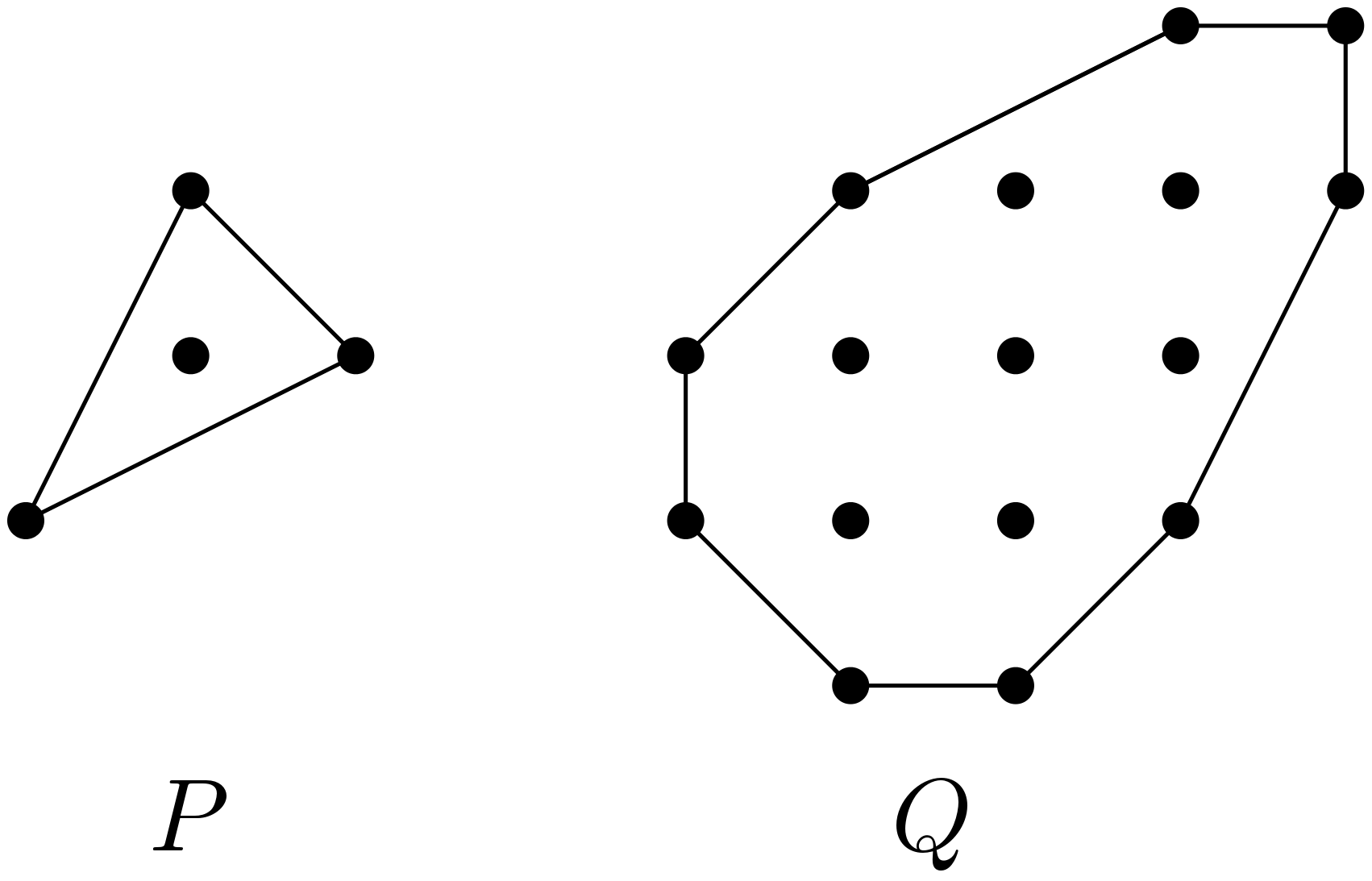}
  \caption{Resolving $\Delta^{(1)}$, eliminating anti-canonical
    sections, and ensuring that the core is a point.}
  \label{fig:Seshadri-example-Q}
\end{figure}
\end{example}

In order to proof the theorem, we need another lemma.

\begin{lemma}\label{lem_originonly}
 Let $\pol$ be a polygon whose Fine core is a point. Then the origin is the only lattice point in the interior of 
 \begin{equation*}
     \operatorname{Ess} \coloneqq \conv(\{u \in \lat \setminus \{0\} \suchthat u \text{ is primitive and essential for } \pol\}).
 \end{equation*}
\end{lemma}

\begin{proof}
Assume there exists a lattice point $u_0 \in \lat \setminus \{0\}$ in the interior of $\operatorname{Ess}$. Then there exist adjacent vertices $u_1,u_2 \in \lat$ of $\operatorname{Ess}$ and coefficients $\lambda_1,\lambda_2 >0$, such that $\linf_0=\lambda_1\linf_1+\lambda_2\linf_2$ with $\lambda_1+\lambda_2 <1$. For the essential vertices it holds that $\langle u_1, \fineCore{\pol} \rangle = \min  \langle \linf_1,\pol\rangle +1 / \codeg{\pol}$ and $\langle u_2, \fineCore{\pol} \rangle= \min  \langle \linf_2,\pol\rangle +1 / \codeg{\pol} $, respectively. Thus for $u_0$ we have
\begin{align*}
    \langle \linf_0,\fineCore{\pol}\rangle  = & \ \langle \lambda_1\linf_1+\lambda_2\linf_2, \fineCore{\pol} \rangle \\
    = & \ (\lambda_1+\lambda_2)\cdot 1 / \codeg{\pol} + \lambda_1 \cdot \min  \langle \linf_1,\pol\rangle  + \lambda_2 \cdot \min  \langle \linf_2,\pol\rangle \\
    \leq & \ (\lambda_1+\lambda_2)\cdot 1 / \codeg{\pol} +  \min  \langle \linf_0,\pol\rangle  \\
     <& \ 1 / \codeg{\pol} +  \min  \langle \linf_0,\pol\rangle.
\end{align*}
This is a contradiction to the definition of $\fineCore{\pol}$. Thus, such a lattice point cannot exist and therefore the origin is the only interior lattice point.  
\end{proof}

\begin{proof}[Proof of Theorem~\ref{thm:seshadri-construction}]
The inequalities \ref{thm:seshadri-construction:width} and
\ref{thm:seshadri-construction:section} hold by assumption for large
$k$, no matter what $Q$ is. Inequality
\ref{thm:seshadri-construction:djs} holds for all $k$ because $kP+Q$
will contain a $k+1$ fold dilate of a unimodular triangle. 

Toric resolution of singularities is a standard procedure,
see~\cite[Chapters~10 \& 11]{cox2011toric}. If necessary, we blow up
further torus fixed points until only one anti-canonical section is
left. This determines the normal fan of $Q$.

It remains to pick $Q$ with the given fan so that the Fine
core is a point and so that we can apply Lemma~\ref{lemma:FineCore}. 
To this end, consider the set of primitive ray generators which
are essential for the given $P$. As $\fineCore{P}$ is a point, the
origin is the only lattice point in the interior of 
their convex hull $Q_1^\vee \subset N_\R$ due to Lemma~\ref{lem_originonly}. Denote $E$ the set of
vertices of $Q_1^\vee$ and denote $Q_1 \subset M_\R$ the polar dual of
$Q_1^\vee$. As $Q_1$ is a simple polytope, we can pick a large $J \in
\Z_{>0}$ so that for every $\linf_0 \in E$ there is a polygon
$Q_{\linf_0}$ with the same normal fan as $Q_1$ such that $\min
\langle \linf_0, Q_{\linf_0} \rangle = -1+1/J$ while $\min \langle
\linf, Q_{\linf} \rangle = -1$ for all other $\linf \in E \setminus
\{\linf_0\}$. By adding appropriate multiples of the $J Q_\linf$ for
$\linf \in E$ to $Q_1$, we can assure that all $\linf \in E$ are
essential for the resulting $Q$.
\end{proof}

\newpage
\section{Appendix}

We give the respective global sections and Newton polytopes that realize a lower bound for the Newton--Okounkov function $\fct_\gen$ on the Newton--Okounkov body $\nob{\flag}{\ds}$ from Example~\ref{ex_16eck}.

\renewcommand{\arraystretch}{1.3}
\begin{table}[h!]
\makebox[\linewidth]{
 \begin{tabular}{|c|| >{\centering\arraybackslash}m{3cm}  >{\centering\arraybackslash}m{4cm}  c l l  c|} 
 \hline
 Region & Inequalities & Newton Polytope & \multicolumn{3}{c}{Section} & $\ord_\gen(s)$ \\ [0.5ex] 
 \hline\hline
 & & & & & &\\
 \multirow{4}{*}{1} &  & \multirow{4}{*}{\includegraphics[width=1in]{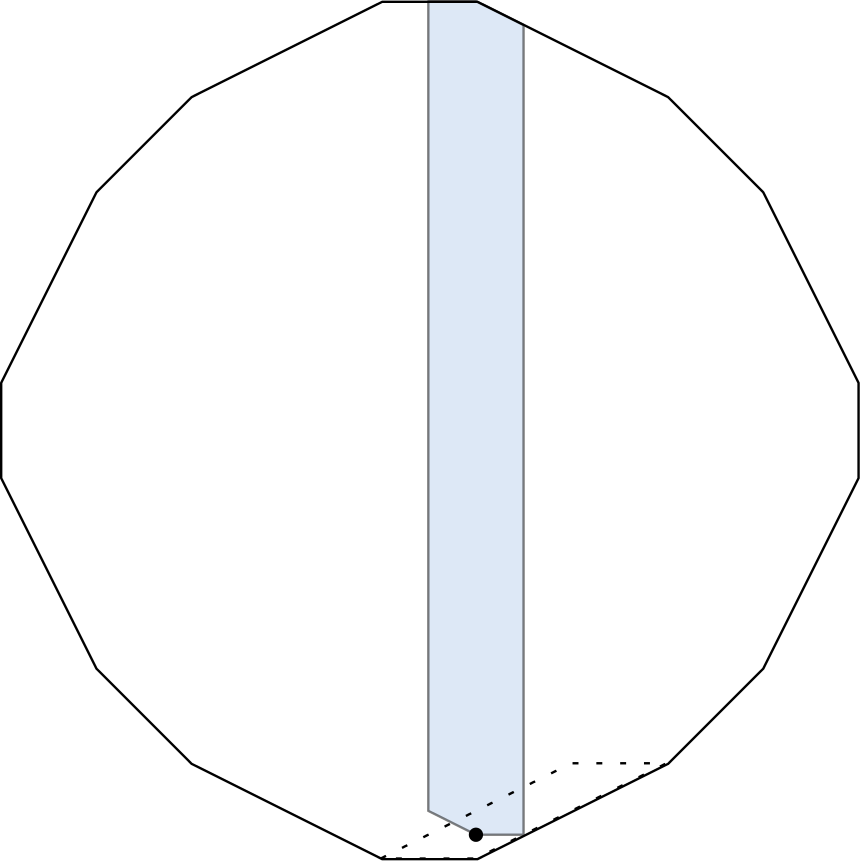}} & 		
 $s(x,y)$ & $=$ & $x^a y^b$ & \multirow{4}{*}{$10-\a+3\b$} \\
 & $0 \leq b \leq 1$, & & & & $(x-1)^{1-a+4b}$ & \\
 & $0\leq a-4b \leq 1 $ & & & & $(y-1)^{b}$ & \\
 & & & & & $(x^2y-1)^{9-2b}$ &\\ 
 & & & & & & \\
 \hline
  & & & & & &\\
 \multirow{5}{*}{2} &  & \multirow{5}{*}{\includegraphics[width=1in]{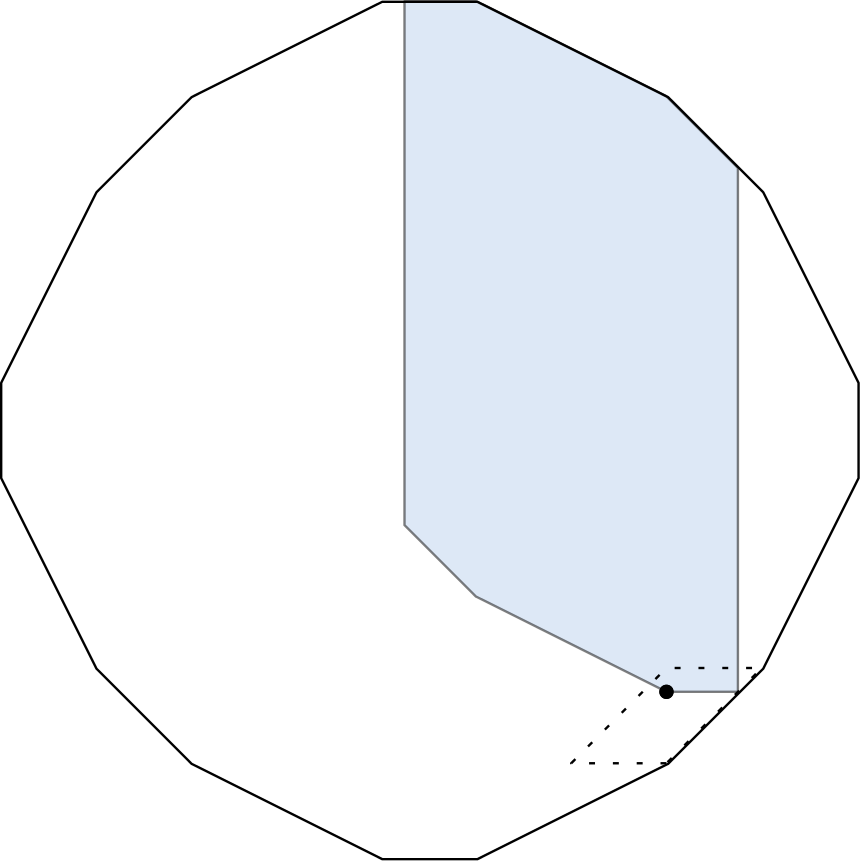}} & 		
 $s(x,y)$ & $=$ & $x^a y^b$ & \multirow{5}{*}{$11-\a+2\b$} \\
 & $1 \leq \b \leq 2$, & & & & $(x-1)^{2-\a+3\b}$ & \\
 & $1 \leq \a-3\b \leq 2 $ & & & & $(y-1)^{2-\b}$ & \\
 & & & & & $(x^2y-1)^{10-3\b}$ &\\ 
 & & & & & $(xy-1)^{3b-3}$ & \\
 & & & & & &\\
 \hline
 & & & & & &\\
 \multirow{4}{*}{3} &  & \multirow{4}{*}{\includegraphics[width=1in]{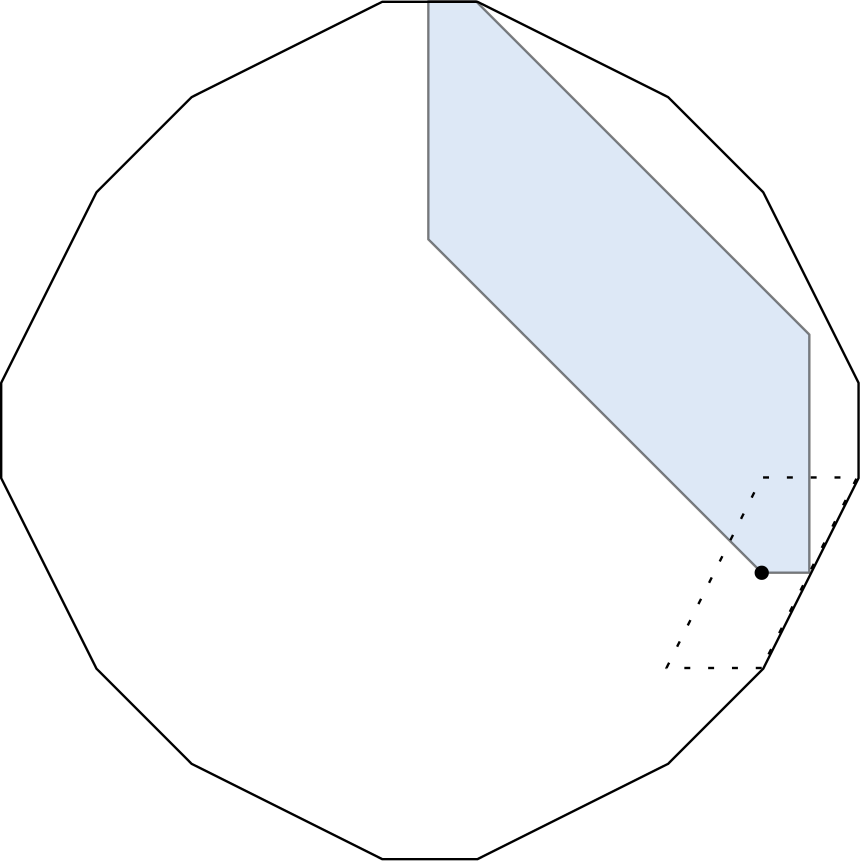}} & 		
 $s(x,y)$ & $=$ & $x^a y^b$ & \multirow{4}{*}{$12-\a+\frac{3}{2}\b$} \\
 & $2 \leq \b \leq 4$, & & & & $(x-1)^{3-\a+ \frac{5}{2}\b}$ & \\
 & $4 \leq 2\a-5\b \leq 6 $ & & & & $(x^2y-1)^{7-\frac{3}{2}\b}$  & \\
 & & & & & $(xy-1)^{2+\frac{1}{2}\b}$ &\\ 
 & & & & &  & \\
 \hline
 & & & & & &\\
 \multirow{4}{*}{4} &  & \multirow{4}{*}{\includegraphics[width=1in]{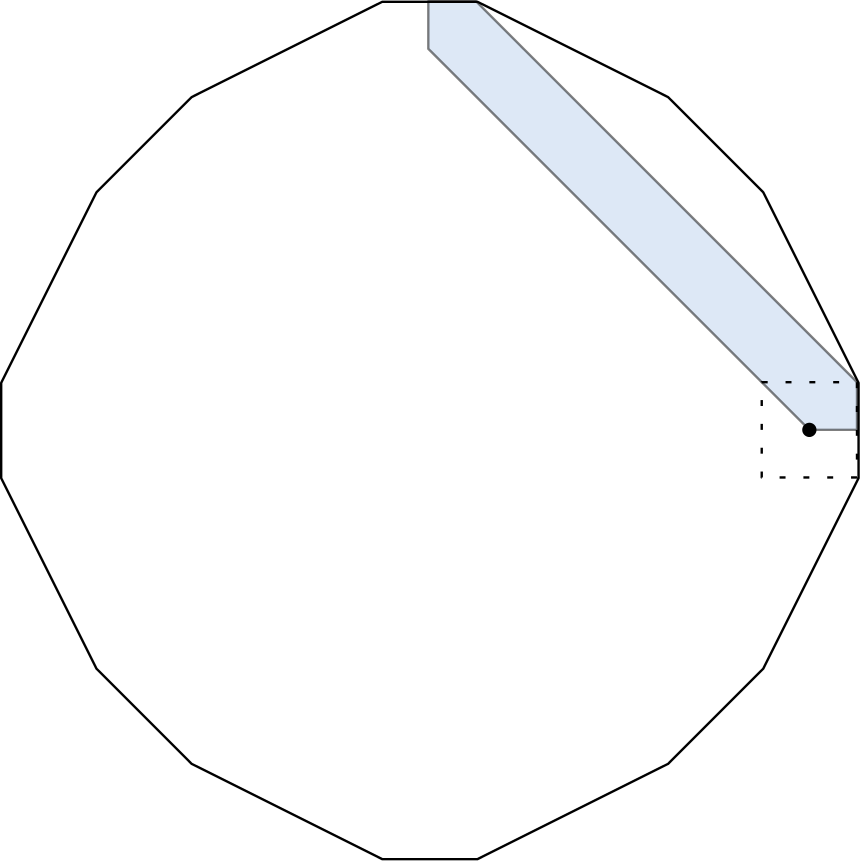}} & 		
 $s(x,y)$ & $=$ & $x^a y^b$ & \multirow{4}{*}{$14-\a+\b$} \\
 & $4 \leq \b \leq 5$, & & & & $(x-1)^{5-\a+2\b}$ & \\
 & $4 \leq \a-2\b \leq 5 $ & & & & $(x^2y-1)^{5-\b}$  & \\
 & & & & & $(xy-1)^{4}$ &\\ 
 & & & & &  & \\
 \hline
\end{tabular}}
\end{table}

\begin{table}
\makebox[\linewidth]{
 \begin{tabular}{|c|| >{\centering\arraybackslash}m{3cm}  >{\centering\arraybackslash}m{4cm} c l l c|} 
 \hline
 Region & Inequalities & Newton Polytope & \multicolumn{3}{c}{Section} & $\ord_\gen(s)$ \\ [0.5ex] 
 \hline\hline
 & & & & & &\\
 \multirow{4}{*}{5} &  & \multirow{4}{*}{\includegraphics[width=1in]{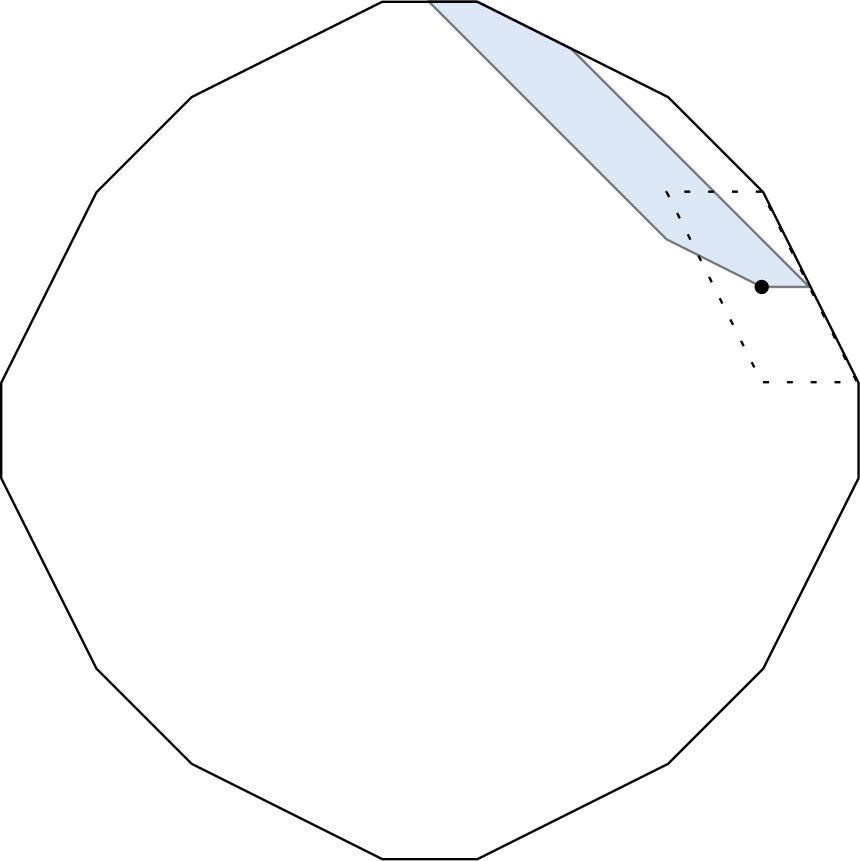}} & 		
 $s(x,y)$ & $=$ & $x^a y^b$ & \multirow{4}{*}{$\frac{33}{2}-\a+\frac{1}{2}\b$} \\
 & $1 \leq \b \leq 2$, & & & & $(x-1)^{\frac{15}{2}-\a+\frac{3}{2}\b}$ & \\
 & $1 \leq \a-3\b \leq 2$ & & & & $(y-1)^{-\frac{5}{2}+\frac{1}{2}\b}$ & \\
 & & & & & $(xy-1)^{\frac{23}{2}-\frac{3}{2}\b}$ & \\
  & & & & & &\\ 
 \hline
  & & & & & &\\
 \multirow{4}{*}{6} &  & \multirow{4}{*}{\includegraphics[width=1in]{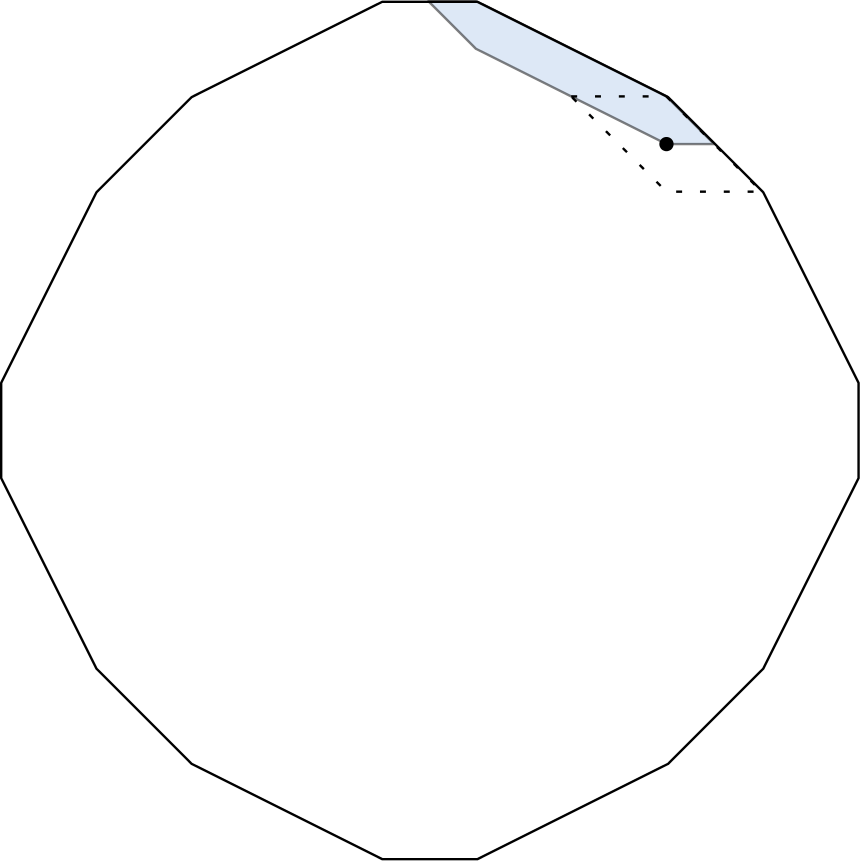}} & 		
 $s(x,y)$ & $=$ & $x^a y^b$ & \multirow{4}{*}{$20 -\a$} \\
 & $7 \leq \b \leq 8$, & & & & $(x-1)^{11-\a+\b}$ & \\
 & $10 \leq \a-\b \leq 11$ & & & & $(y-1)$ & \\
 & & & & & $(xy-1)^{8-\b}$ & \\
  & & & & & &\\ 
 \hline
 & & & & & &\\
 \multirow{4}{*}{7} &  & \multirow{4}{*}{\includegraphics[width=1in]{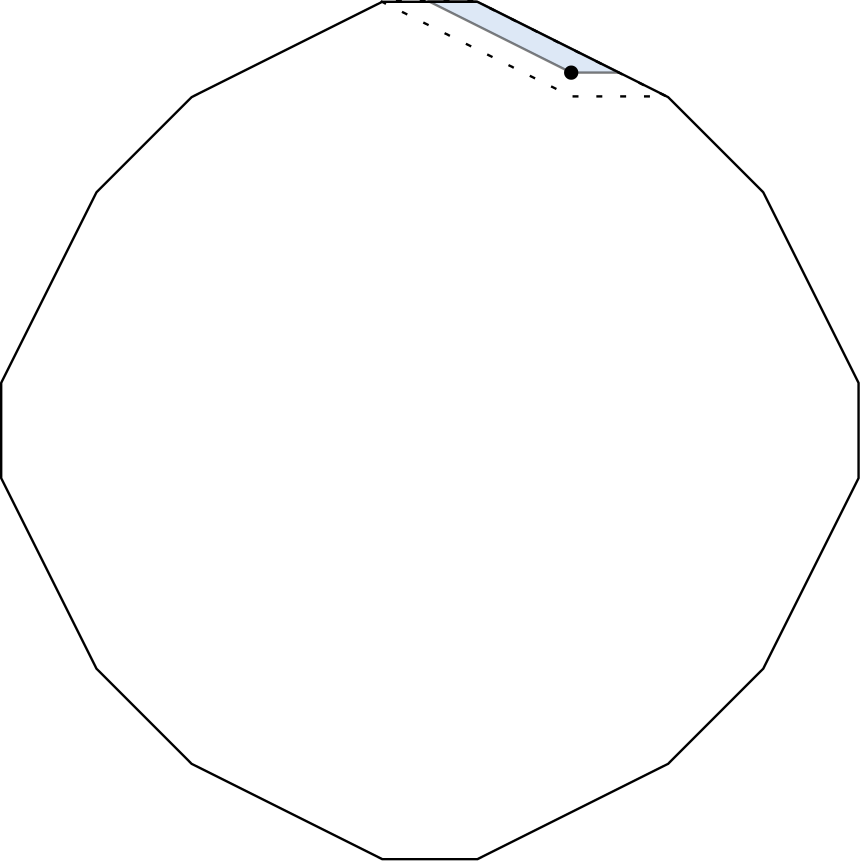}} & 		
 $s(x,y)$ & $=$ & $x^a y^b$ & \multirow{4}{*}{$28-\a-\b$} \\
 & $8 \leq \b \leq 9$, & & & & $(x-1)^{19-\a} $ & \\
 & $18 \leq \a \leq 19 $ & & & & $(y-1)^{9-\b} $ & \\
 & & & & & &\\ 
 & & & & & & \\
 \hline
 & & & & & &\\
 \multirow{4}{*}{8} &  & \multirow{4}{*}{\includegraphics[width=1in]{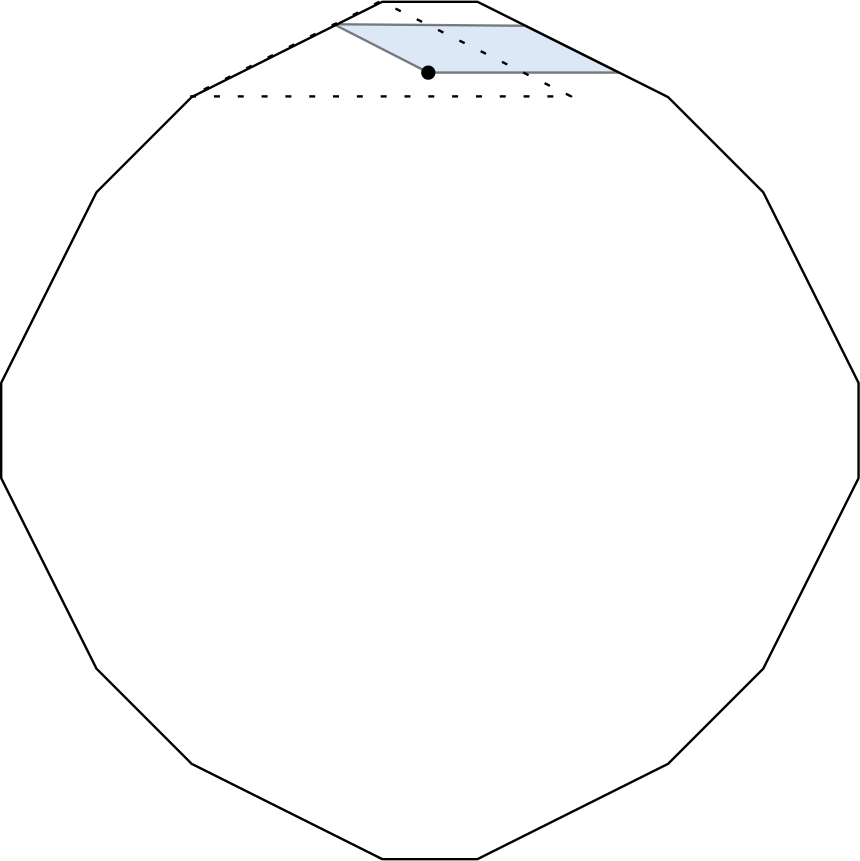}} & 		
 $s(x,y)$ & $=$ & $x^a y^b$ & \multirow{4}{*}{$\frac{47}{2}-\frac{3}{4}\a-\b$} \\
 & $8 \leq \b \leq 9$, & & & & $(x-1)^{19-a}$ & \\
 & $\a \leq 18 $, & & & & $(y-1)^{\frac{9}{2}+\frac{1}{4}\a-\b}$ & \\
 & $-\a+4\b \leq 18$ & & & &  & \\
  & & & & & &\\ 
 \hline
   & & & & & &\\
 \multirow{4}{*}{9} &  & \multirow{4}{*}{\includegraphics[width=1in]{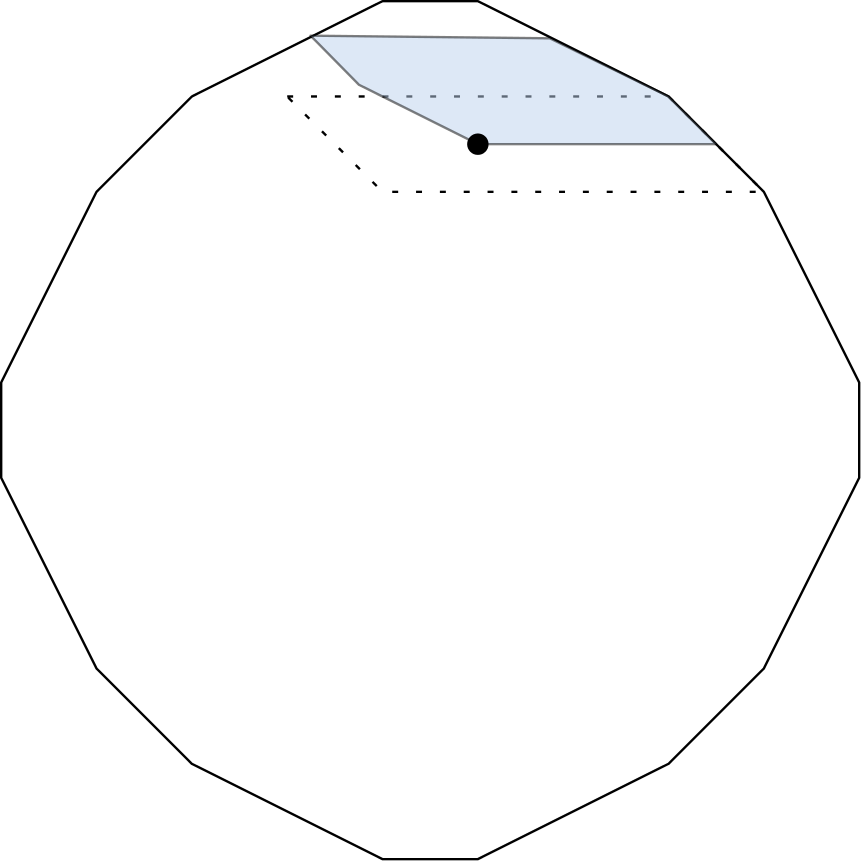}} & 		
 $s(x,y)$ & $=$ & $x^a y^b$ & \multirow{4}{*}{$\frac{35}{2}-\frac{3}{4}\a-\frac{1}{4}\b$} \\
 & $7 \leq \b \leq 8$, & & & & $(x-1)^{11-\a+\b}$ & \\
 & $6 \leq \a-\b \leq 10 $ & & & & $(y-1)^{-3+\frac{1}{4}\a-\frac{1}{4}\b}$ & \\
 & & & & & $(xy-1)^{8-\b}$ & \\
  & & & & & &\\ 
 \hline
 \end{tabular}}
 \end{table}

\begin{table}
\makebox[\linewidth]{
 \begin{tabular}{|c||  >{\centering\arraybackslash}m{3cm}  >{\centering\arraybackslash}m{4cm}  c l l  c|} 
 \hline
 Region & Inequalities & Newton Polytope & \multicolumn{3}{c}{Section} & $\ord_\gen(s)$ \\ [0.5ex] 
 \hline\hline
    & & & & & &\\
 \multirow{6}{*}{10} &  & \multirow{6}{*}{\includegraphics[width=1in]{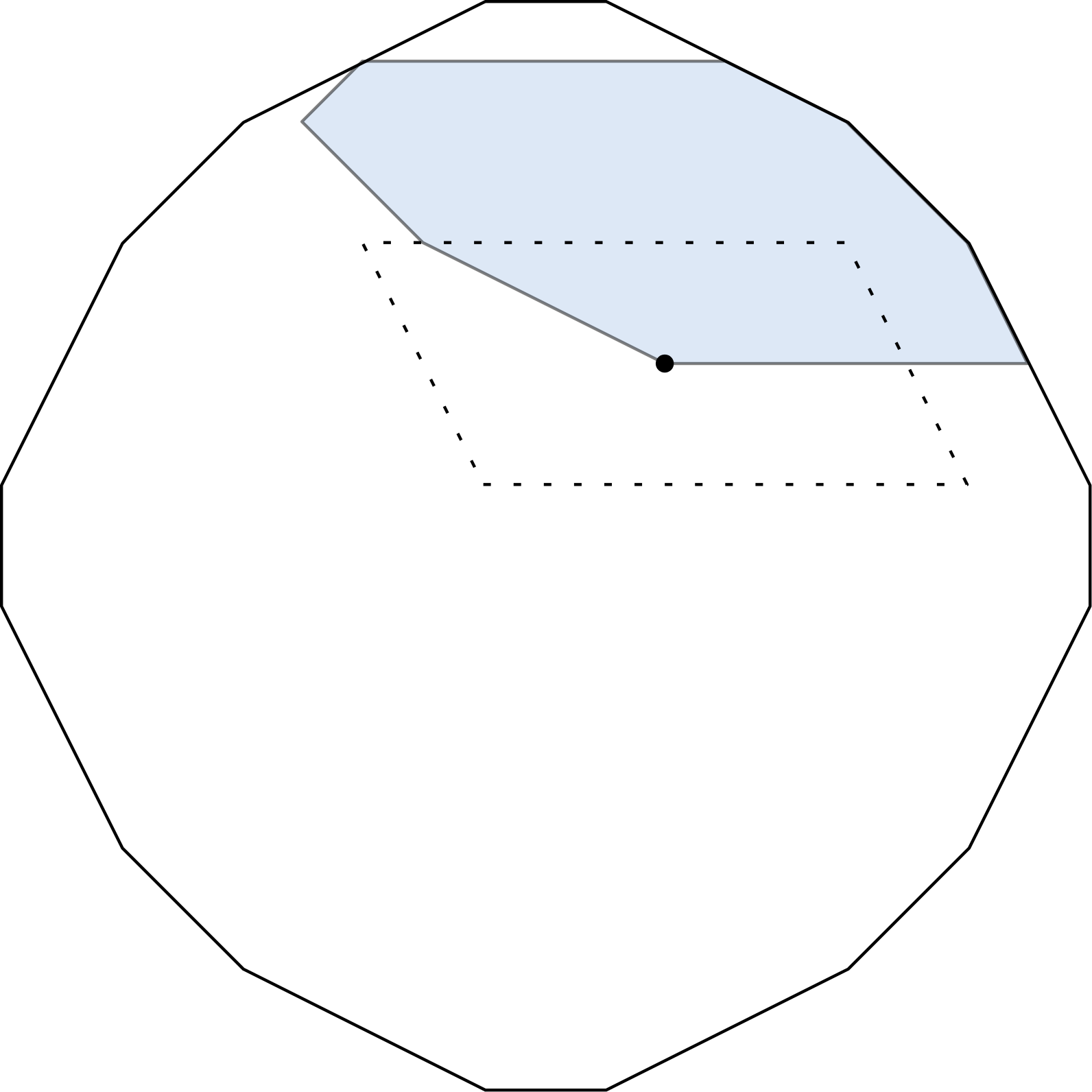}} & 		
 $s(x,y)$ & $=$ & $x^a y^b$ & \multirow{6}{*}{$\frac{119}{8}-\frac{3}{4}\a+\frac{1}{8}\b$} \\
 &  & & & & $(x-1)^{\frac{15}{2}-\a+\frac{3}{2}\b}$ & \\
 & $5 \leq \b \leq 7$, & & & & $(y-1)^{-\frac{5}{8}+\frac{1}{4}\a-\frac{3}{8}\b}$ & \\
 & $5 \leq 2\a-3\b \leq 13 $& & & & $(xy-1)$ & \\
 & & & & & $(x^3y^2-3xy$ &\\
 & & & & & $+y+1)^{\frac{7}{2}-\frac{1}{2}\b}$ &\\
  & & & & & &\\
   \hline
    & & & & & &\\
 \multirow{5}{*}{11} &  & \multirow{5}{*}{\includegraphics[width=1in]{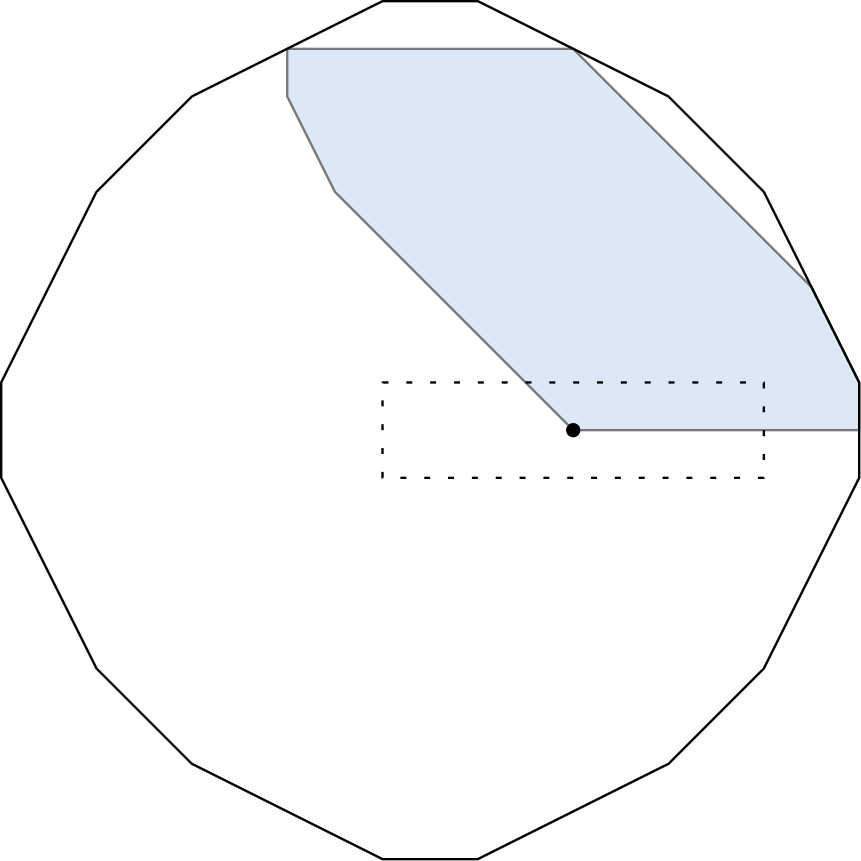}} & 		
 $s(x,y)$ & $=$ & $x^a y^b$ & \multirow{5}{*}{$13-\frac{3}{4}\a+\frac{1}{2}\b$} \\
 &  & & & & $(x-1)^{5-\a+2\b}$ & \\
 & $4 \leq \b \leq 5$, & & & & $(x^2y-1)^{5-\b}$ & \\
 & $0 \leq \a-2\b \leq 4 $ & & & & $(xy-1)^{1+\frac{3}{4}\a-\frac{3}{2}\b}$ & \\
  & & & & & $(x^3y^2-3xy$ &\\
  & & & & & $+y+1)^{1-\frac{1}{4}\a +\frac{1}{2}\b}$ &\\
  & & & & & &\\
   \hline
 & & & & & &\\
 \multirow{7}{*}{12} &  & \multirow{7}{*}{\includegraphics[width=1in]{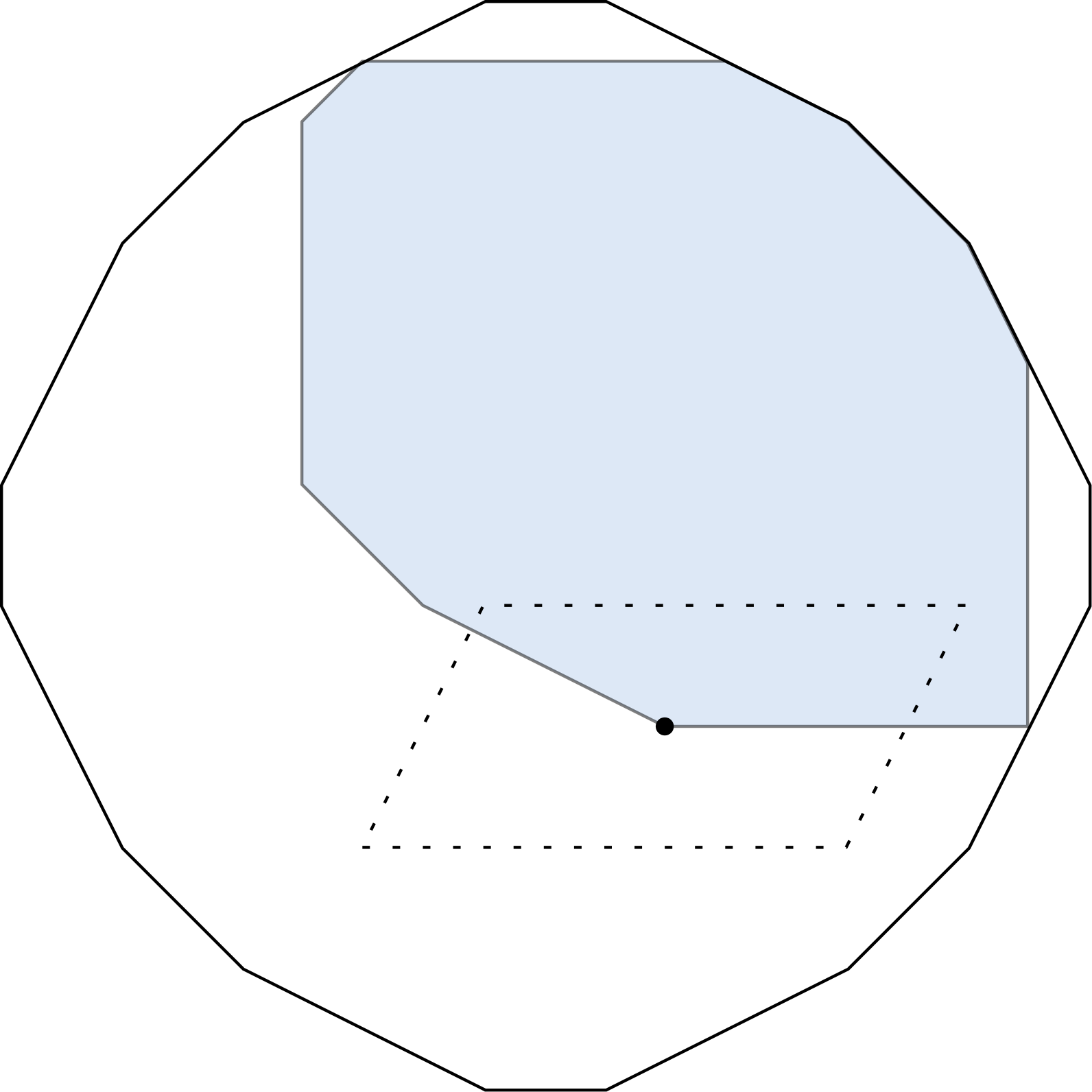}} & 		
 $s(x,y)$ & $=$ & $x^a y^b$ & \multirow{7}{*}{$\frac{23}{2}-\frac{3}{4}\a+\frac{7}{8}\b$} \\
 &  & & & & $(x-1)^{3-\a+\frac{5}{2}\b}$ & \\
   & $2 \leq \b \leq 4$, & & & & $(y-1)^{\frac{1}{2}+\frac{1}{4}\a-\frac{5}{8}\b}$ &\\
 & $-4 \leq -2\a+5\b $, & & & & $(x^2y-1)^{9-2\b}$ & \\
 & $-2\a+5\b \leq 4$ & & & & $(xy-1)$ & \\
 & & & & & $(x^3y^2-3xy$ &\\
 & & & & & $+y+1)^{-1+\frac{1}{2}\b}$ &\\
& & & & & &\\
\hline
  & & & & & &\\
 \multirow{5}{*}{13} &  & \multirow{5}{*}{\includegraphics[width=1in]{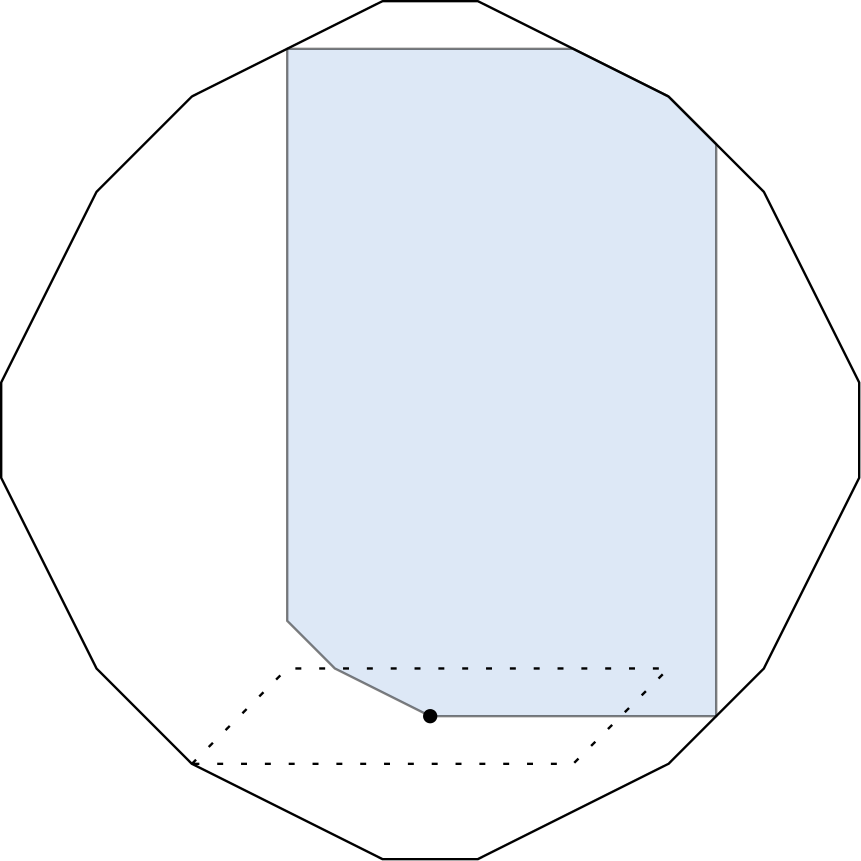}} & 		
 $s(x,y)$ & $=$ & $x^a y^b$ & \multirow{5}{*}{$\frac{43}{4}-\frac{3}{4}\a+\frac{5}{4}\b$} \\
 & $1 \leq \b \leq 2$, & & & & $(x-1)^{2-\a+3\b} $ & \\
 & $-3 \leq \a-3\b \leq 1 $ & & & & $(y-1)^{\frac{3}{4}+ \frac{1}{4}\a-\frac{3}{4}\b} $ & \\
 & & & & & $(x^2y-1)^{9-2\b}$ &\\ 
 & & & & & $(xy-1)^{-1+\b}$ & \\
  & & & & & &\\
 \hline
\end{tabular}}
\end{table}

  \begin{table}[ht!]
\makebox[\linewidth]{
 \begin{tabular}{|c||  >{\centering\arraybackslash}m{3cm}  >{\centering\arraybackslash}m{4cm}  c l l  c|} 
 \hline
 Region & Inequalities & Newton Polytope & \multicolumn{3}{c}{Section} & $\ord_\gen(s)$ \\ [0.5ex] 
 \hline\hline
 & & & & & &\\
 \multirow{4}{*}{14} &  & \multirow{5}{*}{\includegraphics[width=1in]{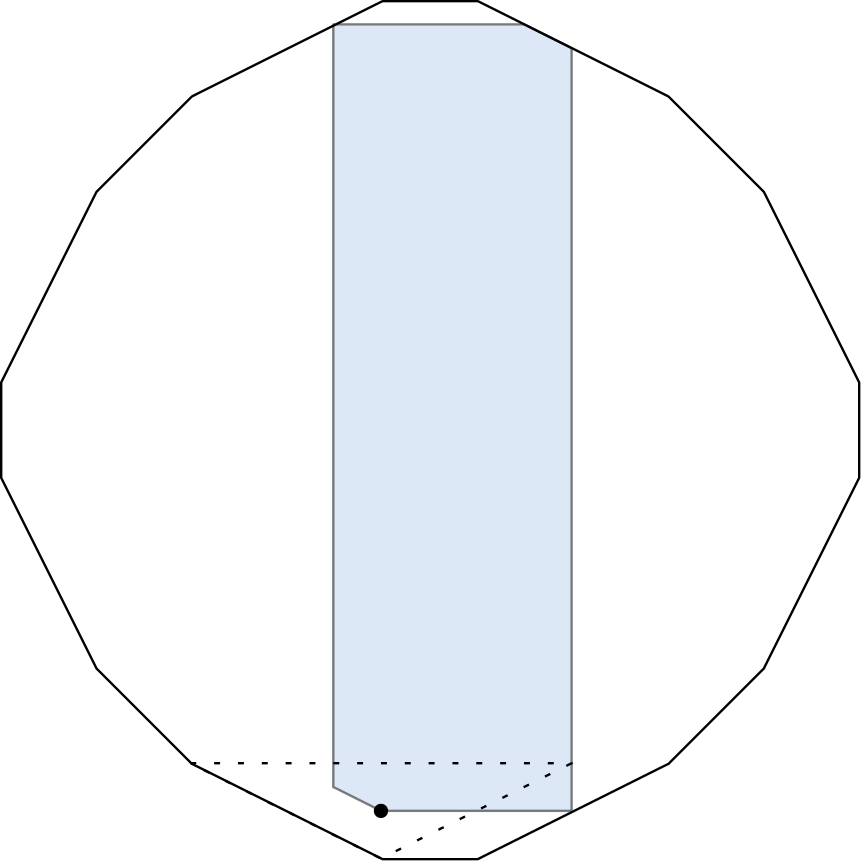}} & 		
 $s(x,y)$ & $=$ & $x^a y^b$ & \multirow{4}{*}{$10-\frac{3}{4}\a+2\b$} \\
 & $5 \leq \b \leq 7$, & & & & $(x-1)^{1-\a+4\b}$ & \\
 & $5 \leq 2\a-3\b \leq 13 $ & & & & $(y-1)^{\frac{1}{4}\a} $ & \\
 & & & & & $(x^2y-1)^{9-2\b}$ &\\ 
  & & & & & $(xy-1)^{-1+\b}$ & \\
   & & & & &  & \\
 \hline
 & & & & & &\\
 \multirow{4}{*}{15} &  & \multirow{4}{*}{\includegraphics[width=1in]{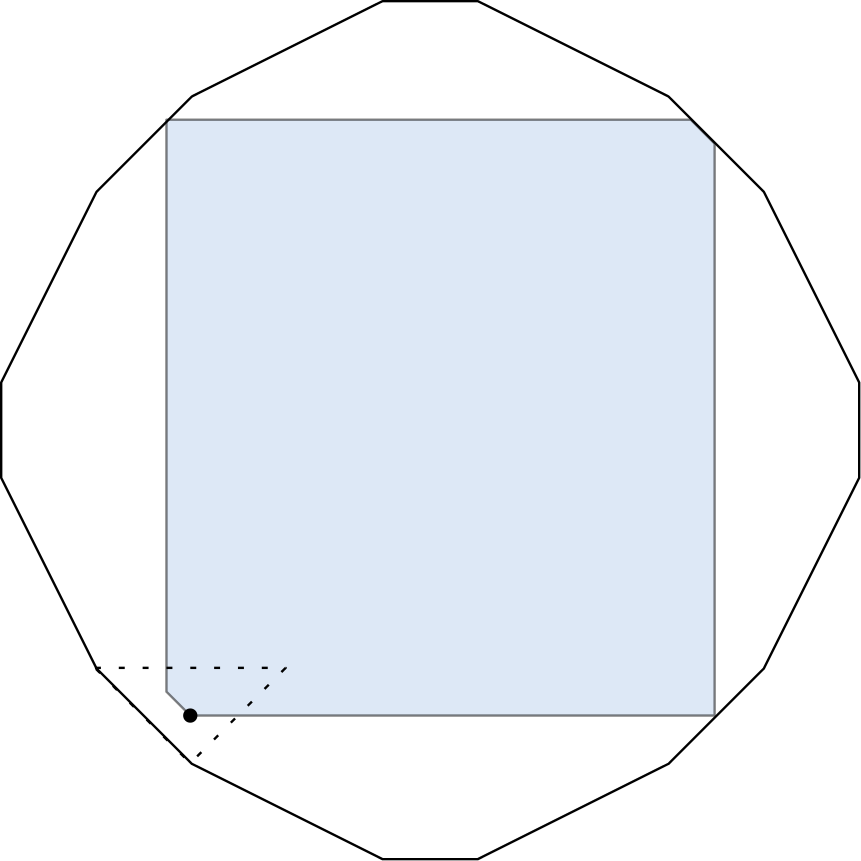}} & 		
 $s(x,y)$ & $=$ & $x^a y^b$ & \multirow{4}{*}{$\frac{23}{2}-\frac{1}{2}\a+\frac{1}{2}\b$} \\
 & $1 \leq \b \leq 2$, & & & & $(x-1)^{2-\a+3\b} $ & \\
 & $-\a+\b \leq 1 $, & & & & $(x^2y-1)^{9-2\b}$ & \\
 & $-a+3\b \geq 3$ & & & & $(xy-1)^{\frac{1}{2}+\frac{1}{2}\a-\frac{1}{2}\b}$ &\\ 
 & & & & &  & \\
  \hline
& & & & & &\\
 \multirow{4}{*}{19} &  & \multirow{4}{*}{\includegraphics[width=1in]{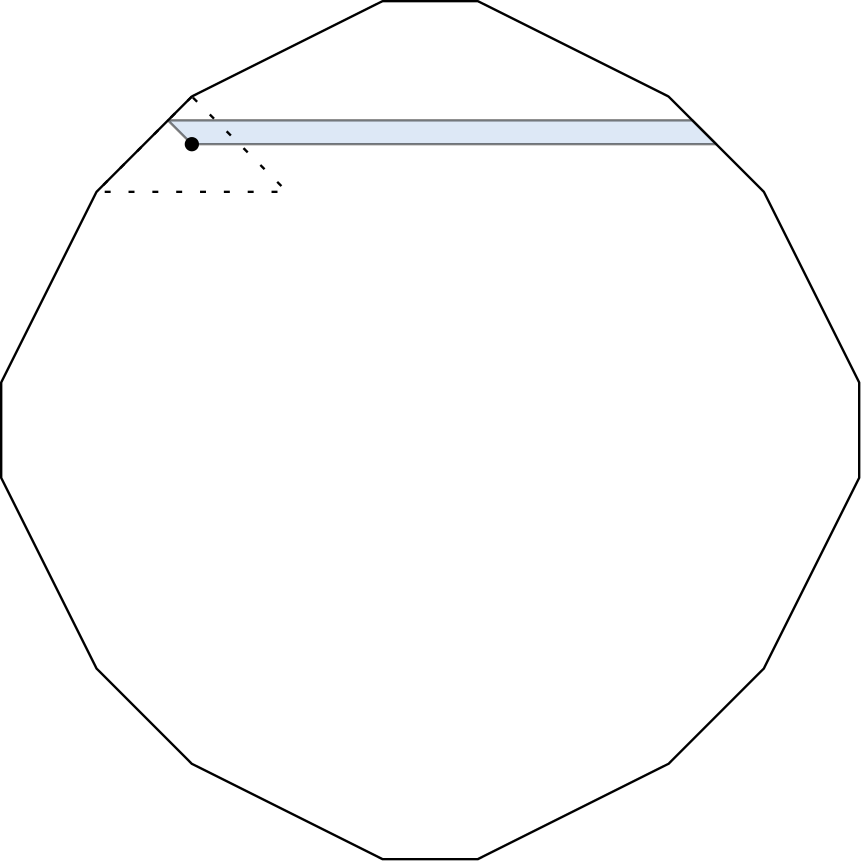}} & 		
 $s(x,y)$ & $=$ & $x^a y^b$ & \multirow{4}{*}{$16-\frac{1}{2}\a-\frac{1}{2}\b$} \\
 & $7 \leq \b \leq 8$, & & & & $(x-1)^{11-\a+\b}$ & \\
 & $\a-\b \leq 6$, & & & & $(xy-1)^{5+\frac{1}{2}\a-\frac{3}{2}\b}$ & \\
 & $-\a+3\b \leq 10$ & & & & &\\ 
 & & & & &  & \\
  \hline
  \end{tabular}}
  \caption{Sections $s$ such that $s^k$ are global sections of $\sps(\var,\she_\var(k\ds))$ that realize lower bounds for the order of vanishing $\ord_\gen$ for respective $k \in \N$.}\label{table_2}
  \end{table}
\newpage
$ $
\newpage

\bibliographystyle{amsalpha}
\bibliography{references}
\end{document}